\documentclass[12pt]{amsart}
\usepackage{amssymb}
\usepackage{amsmath}

\usepackage[displaymath]{lineno}

\theoremstyle{plain}
\newtheorem{thm}{Theorem}[section]

\newtheorem{theorem}[thm]{Theorem}
\newtheorem{lemma}[thm]{Lemma}
\newtheorem{corollary}[thm]{Corollary}
\newtheorem{proposition}[thm]{Proposition}

\theoremstyle{definition}

\newtheorem{remark}[thm]{Remark}

\newtheorem{definition}[thm]{Definition}

\newtheorem{example}[thm]{Example}

\newtheorem{defn-thm}[thm]{Definition-Theorem}

\newtheorem*{myrem}{Remark}

\numberwithin{equation}{thm}

\catcode`\@=11
\def\opn#1#2{\def#1{\mathop{\kern0pt\fam0#2}\nolimits}}
\def\underrightarrow{\mathpalette\underrightarrow@}
\def\underrightarrow@#1#2{\vtop{\ialign{$##$\cr
 \hfil#1#2\hfil\cr\noalign{\nointerlineskip}%
 #1{-}\mkern-6mu\cleaders\hbox{$#1\mkern-2mu{-}\mkern-2mu$}\hfill
 \mkern-6mu{\to}\cr}}}

\def\underleftarrow{\mathpalette\underleftarrow@}
\def\underleftarrow@#1#2{\vtop{\ialign{$##$\cr
 \hfil#1#2\hfil\cr\noalign{\nointerlineskip}#1{\leftarrow}\mkern-6mu
 \cleaders\hbox{$#1\mkern-2mu{-}\mkern-2mu$}\hfill
 \mkern-6mu{-}\cr}}}
\let\amp@rs@nd@\relax
\newdimen\ex@
\newdimen\bigaw@
\newdimen\minaw@
\newdimen\minCDaw@
\newif\ifCD@
\def\minCDarrowwidth#1{\minCDaw@#1}

\def\@CD{\def\A##1A##2A{\llap{$\vcenter{\hbox
 {$\scriptstyle##1$}}$}\Big\uparrow\rlap{$\vcenter{\hbox{%
$\scriptstyle##2$}}$}&&}%
\def\V##1V##2V{\llap{$\vcenter{\hbox
 {$\scriptstyle##1$}}$}\Big\downarrow\rlap{$\vcenter{\hbox{%
$\scriptstyle##2$}}$}&&}%
\def\={&\hskip.5em\mathrel
 {\vbox{\hrule width\minCDaw@\vskip3\ex@\hrule width
 \minCDaw@}}\hskip.5em&}%
\def\verteq{\Big\Vert&&}%
\def\noarr{&&}%
\def\vspace##1{\noalign{\vskip##1\relax}}\relax\let\amp@rs@nd@&\iffalse}\fi
 \CD@true\vcenter\bgroup\relax\let\\=\cr\iffalse}\fi\tabskip\z@skip\baselineskip20\ex@
 &\hfill$\m@th##$\hfill\cr}
\def\@endCD{\cr\egroup\egroup}
\def\>#1>#2>{\amp@rs@nd@\setbox\z@\hbox{$\scriptstyle
 \;{#1}\;\;$}\setbox\@ne\hbox{$\scriptstyle\;{#2}\;\;$}\setbox\tw@
 \hbox{$#2$}\ifCD@
 \global\bigaw@\minCDaw@\else\global\bigaw@\minaw@\fi
 \ifdim\wd\z@>\bigaw@\global\bigaw@\wd\z@\fi
 \ifdim\wd\@ne>\bigaw@\global\bigaw@\wd\@ne\fi
 \ifCD@\hskip.5em\fi
 \ifdim\wd\tw@>\z@
 \mathrel{\mathop{\hbox to\bigaw@{\rightarrowfill}}\limits^{#1}_{#2}}\else
 \mathrel{\mathop{\hbox to\bigaw@{\rightarrowfill}}\limits^{#1}}\fi
 \ifCD@\hskip.5em\fi\amp@rs@nd@}
\def\<#1<#2<{\amp@rs@nd@\setbox\z@\hbox{$\scriptstyle
 \;\;{#1}\;$}\setbox\@ne\hbox{$\scriptstyle\;\;{#2}\;$}\setbox\tw@
 \hbox{$#2$}\ifCD@
 \global\bigaw@\minCDaw@\else\global\bigaw@\minaw@\fi
 \ifdim\wd\z@>\bigaw@\global\bigaw@\wd\z@\fi
 \ifdim\wd\@ne>\bigaw@\global\bigaw@\wd\@ne\fi
 \ifCD@\hskip.5em\fi
 \ifdim\wd\tw@>\z@
 \mathrel{\mathop{\hbox to\bigaw@{\leftarrowfill}}\limits^{#1}_{#2}}\else
 \mathrel{\mathop{\hbox to\bigaw@{\leftarrowfill}}\limits^{#1}}\fi
 \ifCD@\hskip.5em\fi\amp@rs@nd@}
\newenvironment{CDS}{\@CDS}{\@endCDS}
\def\@CDS{\def\A##1A##2A{\llap{$\vcenter{\hbox
 {$\scriptstyle##1$}}$}\Big\uparrow\rlap{$\vcenter{\hbox{%
$\scriptstyle##2$}}$}&}%
\def\V##1V##2V{\llap{$\vcenter{\hbox
 {$\scriptstyle##1$}}$}\Big\downarrow\rlap{$\vcenter{\hbox{%
$\scriptstyle##2$}}$}&}%
\def\={&\hskip.5em\mathrel
 {\vbox{\hrule width\minCDaw@\vskip3\ex@\hrule width
 \minCDaw@}}\hskip.5em&}
\def\verteq{\Big\Vert&}
\def\novarr{&}
\def\noharr{&&}
\def\SE##1E##2E{\slantedarrow(0,18)(4,-3){##1}{##2}&}
\def\SW##1W##2W{\slantedarrow(24,18)(-4,-3){##1}{##2}&}
\def\NE##1E##2E{\slantedarrow(0,0)(4,3){##1}{##2}&}
\def\NW##1W##2W{\slantedarrow(24,0)(-4,3){##1}{##2}&}
\def\slantedarrow(##1)(##2)##3##4{%
\thinlines\unitlength1pt\lower 6.5pt\hbox{\begin{picture}(24,18)%
\put(##1){\vector(##2){24}}%
\put(0,8){$\scriptstyle##3$}%
\put(20,8){$\scriptstyle##4$}%
\end{picture}}}
\def\vspace##1{\noalign{\vskip##1\relax}}\relax\let\amp@rs@nd@&\iffalse}\fi
 \CD@true\vcenter\bgroup\relax\let\\=\cr\iffalse}\fi\tabskip\z@skip\baselineskip20\ex@
 &\hfill$\m@th##$\hfill\cr}
\def\@endCDS{\cr\egroup\egroup}
\newdimen\TriCDarrw@
\newif\ifTriV@

\newenvironment{TriCDA}{\@TriCDA}{\@endTriCD}
\def\@TriCDV{\TriV@true\def\TriCDpos@{6}\@TriCD}
\def\@TriCDA{\TriV@false\def\TriCDpos@{10}\@TriCD}
\def\@TriCD#1#2#3#4#5#6{%
\setbox0\hbox{$\ifTriV@#6\else#1\fi$} \TriCDarrw@=\wd0
\advance\TriCDarrw@ 24pt \advance\TriCDarrw@ -1em
\def\SE##1E##2E{\slantedarrow(0,18)(2,-3){##1}{##2}&}
\def\SW##1W##2W{\slantedarrow(12,18)(-2,-3){##1}{##2}&}
\def\NE##1E##2E{\slantedarrow(0,0)(2,3){##1}{##2}&}
\def\NW##1W##2W{\slantedarrow(12,0)(-2,3){##1}{##2}&}
\def\slantedarrow(##1)(##2)##3##4{\thinlines\unitlength1pt
\lower 6.5pt\hbox{\begin{picture}(12,18)%
\put(##1){\vector(##2){12}}%
\put(-4,\TriCDpos@){$\scriptstyle##3$}%
\put(12,\TriCDpos@){$\scriptstyle##4$}%
\end{picture}}}
\def\={\mathrel {\vbox{\hrule
   width\TriCDarrw@\vskip3\ex@\hrule width
   \TriCDarrw@}}}
\def\>##1>>{\setbox\z@\hbox{$\scriptstyle
 \;{##1}\;\;$}\global\bigaw@\TriCDarrw@
 \ifdim\wd\z@>\bigaw@\global\bigaw@\wd\z@\fi
 \hskip.5em
 \mathrel{\mathop{\hbox to \TriCDarrw@
{\rightarrowfill}}\limits^{##1}}
 \hskip.5em}
\def\<##1<<{\setbox\z@\hbox{$\scriptstyle
 \;{##1}\;\;$}\global\bigaw@\TriCDarrw@
 \ifdim\wd\z@>\bigaw@\global\bigaw@\wd\z@\fi
 \mathrel{\mathop{\hbox to\bigaw@{\leftarrowfill}}\limits^{##1}}
 }
 \CD@true\vcenter\bgroup\relax\let\\=\cr\iffalse}\fi
 \tabskip\z@skip\baselineskip20\ex@
 \ifTriV@
 \halign\bgroup
 &\hfill$\m@th##$\hfill\cr
#1&\multispan3\hfill$#2$\hfill&#3\\
&#4&#5\\
&&#6\cr\egroup%
\else
 \halign\bgroup
 &\hfill$\m@th##$\hfill\cr
&&#1\\%
&#2&#3\\
#4&\multispan3\hfill$#5$\hfill&#6\cr\egroup \fi}
\def\@endTriCD{\egroup}

\newcounter{Myenumi}
\newenvironment{myenumi}%
{\begin{list}{}{\usecounter{Myenumi}%
\settowidth{\leftmargin}{2.n}\settowidth{\labelwidth}{2.n}%
\setlength{\labelsep}{0pt}}}{\end{list}}
\newcounter{Myenumii}
\newenvironment{myenumii}%
{\begin{list}{}{\usecounter{Myenumii}%
\settowidth{\leftmargin}{a)n}\settowidth{\labelwidth}{a)n}%
\setlength{\labelsep}{0pt}}}{\end{list}}
\newcounter{Myenumiii}
\newenvironment{myenumiii}%
{\begin{list}{}{\usecounter{Myenumiii}%
\settowidth{\leftmargin}{iv.n}\settowidth{\labelwidth}{iv.n}%
\setlength{\labelsep}{0pt}}}{\end{list}}

{\end{list}}

\renewenvironment{itemize}%
{\begin{list}{}{%
\settowidth{\leftmargin}{2.n}\settowidth{\labelwidth}{2.n}%
\setlength{\labelsep}{0pt}}}{\end{list}}

\newsymbol\onto 1310

\newcommand{\sA}{{\mathcal A}}

\newcommand{\sN}{{\mathcal N}}
\newcommand{\sO}{{\mathcal O}}
\newcommand{\sP}{{\mathcal P}}

\newcommand{\sS}{{\mathcal S}}

\newcommand{\sV}{{\mathcal V}}
\newcommand{\sW}{{\mathcal W}}

\newcommand{\sY}{{\mathcal Y}}

\newcommand{\A}{{\mathbb A}}

\newcommand{\C}{{\mathbb C}}

\newcommand{\G}{{\mathbb G}}

\newcommand{\Q}{{\mathbb Q}}
\newcommand{\R}{{\mathbb R}}

\newcommand{\V}{{\mathbb V}}

\newcommand{\Z}{{\mathbb Z}}

\newcommand{\Ad}{{\rm Ad}}
\newcommand{\rank}{{\rm rank}}

\newcommand{\End}{{\rm End}}

\newcommand{\GL}{{\rm GL}}

\newcommand{\Sp}{{\rm Sp}}

\newcommand{\Hom}{{\rm Hom}}
\newcommand{\Ker}{{\rm Ker}}

\newcommand{\Aut}{{\rm Aut}}

\newcommand{\im}{{\rm Im}}

\newcommand{\Gr}{{\rm Gr}}

\newcommand{\Id}{{\rm Id}}

\newcommand{\Spec}{{\rm Spec}}

\begin{document}

\title[{\tiny The Geometry on Smooth Toroidal Compactifications of Siegel varieties}]{The Geometry on Smooth Toroidal Compactifications of Siegel varieties}
\author{Shing-Tung Yau}
\address{\rm Department of Mathematics\\
Harvard University\\ Cambridge, MA 02138, USA}
\email{yau@math.harvard.edu}

\author{Yi Zhang}
\address{\rm School of Mathematical Sciences \\
Fudan University \\
Shanghai 200433, China}
\email{zhangyi\_math@fudan.edu.cn} 
\maketitle

{\small
\begin{abstract}
We study smooth toroidal
compactifications of Siegel varieties
thoroughly from the viewpoints of mixed Hodge theory and K\"ahler-Einstein metric.
We observe that any cusp of a Siegel space can be identified as a set of certain weight one polarized mixed Hodge structures.
We then study the infinity boundary divisors of toroidal
compactifications, and obtain a global volume form formula of an arbitrary smooth Siegel variety $\sA_{g,\Gamma}(g>1)$ with a smooth toroidal
compactification $\overline{\sA}_{g,\Gamma}$ such that $D_\infty:=\overline{\sA}_{g,\Gamma}\setminus \sA_{g,\Gamma}$ is normal crossing.  We use this volume form formula to show that the unique group-invariant K\"ahler-Einstein metric on $\sA_{g,\Gamma}$ endows some
restraint combinatorial conditions for all smooth toroidal
compactifications of $\sA_{g,\Gamma}.$  Again using the volume form formula, we study the asymptotic behaviour of logarithmical canonical line bundle on any smooth toroidal compactification of $\sA_{g,\Gamma}$ carefully and we obtain
that the logarithmical canonical bundle degenerate sharply even though
it is big and numerically effective.
\end{abstract}
}

\small\tableofcontents
\section{Introduction}

\vspace{0.5cm}

Throughout this paper, the number $g$ is an integer more than two.

Siegel varieties are locally symmetric varieties.
They are important and interesting in algebraic geometry and number theory
because they arise as moduli spaces for Abelian varieties with a polarization and a level structure.

The purpose of this paper is to study smooth toroidal compactifications of Siegel varieties and 	
their applications, we also try to understand the K\"ahler-Einstein metrics on Siegel varieties through the compactifications.
We discuss the geometric aspects of the theory after the works of Ash-Mumford-Rapoport-Tai and Faltings-Chai. Later advances in algebraic geometry 	
have given us many very effective tools for studying these varieties and their toroidal compactifications.

There is a general theory of compactifications of all locally symmetric varieties $D/\Gamma$($D$ a bounded symmetric domain, $\Gamma\subset\Aut(D)$ an arithmetic subgroup). Every variety $D/\Gamma$
has its Stake-Baily-Borel compactification, which is a canonical minimal compactification. But this
compactification has rather bad singularities. In another direction, Ash, Mumford, Rapoport and Tai, in their collaborated book \cite{AMRT}, use the theory of toroidal  embedding to construct a whole class of compactifications with mild singularities, including, when $\Gamma$ is neat, smooth  compactifications. Faltings and Chai use purely algebraic method to construct arithmetic  toroidal compactifications of Siegel varieties.

A  toroidal compactification $\overline{\sA}_{g,\Gamma}$ of a Siegel variety $\sA_{g,\Gamma}:=\mathfrak{H}_g/\Gamma$(here $\mathfrak{H}_g$ is the Siegel space of genus $g$ and $\Gamma\subset\Aut(\mathfrak{H}_g)$ is an arithmetic  subgroup) is totally determined by a combinatorial condition : an admissible family of polyhedral decompositions of certain positive cones. As well known, the natural Bergman metric on $\mathfrak{H}_g$ is K\"ahler-Einstein. The first author believes that the intrinsic K\"ahler-Einstein metric on a quasi-projective manifold $M$ should be helpful for finding a nice compactification $\overline{M}$ of $M,$ and he has thought this problem for a long time. In this paper, we can assert that the K\"ahler-Einstein metric on $\sA_{g,\Gamma}$ endows some restraint combinatorial conditions for all toroidal smooth compactifications of $\sA_{g,\Gamma}$(Theorem \ref{global-volume-form-on Siegel varieties}, Theorem \ref{combinatorial-condition-1} in Section 3). Let us explain this result : Let  $\sigma_{\max}$ be an admissible top-dimensional polyhedral cone with $N$($=\dim_\C\sA_{g,\Gamma})$ edges $\rho_1,\cdots \rho_N.$ Each edge $\rho_i$ of $\sigma_{\max}$ corresponds to an irreducible  components $D_i$ of the boundary divisor $D_\infty:= \overline{\sA}_{g,\Gamma}\setminus\sA_{g,\Gamma}.$
Assume that $D_\infty$ is normal crossing. For every $i=1,\cdots,N,$ let $s_i$ be the global section of the line bundle $[D_i]$ defining $D_i.$ Then, the $(s_1,\cdots,s_N)$  give us a global coordinate system on $\sA_{g,\Gamma}$ and we can choose a suitable Hermitian metric $||\cdot||_i$ on each $[D_i]$ such that the volume form on $\sA_{g,\Gamma}$ is represented by
\begin{equation}\label{global-volume-form-introduction}
\Phi_{g,\Gamma}=\frac{2^{\frac{g(g-1)}{2}}\mathrm{vol}_{\Gamma}(\sigma_{\max})^2d\sV_g}{(\prod_{j=1}^{N} ||s_i||_i^2)F^{g+1}_{\sigma_{\max}}(\log||s_1||_1,\cdots, \log||s_{N}||_{N})},
\end{equation}
where $d\sV_g$ is a continuous volume form on a partial compactification $\mathcal{U}_{\sigma_{\max}}$ of $\sA_{g,n}$
with $\sA_{g,\Gamma}\subset \mathcal{U}_{\sigma_{\max}}\subset \overline{\sA}_{g,\Gamma},$
$F_{\sigma_{\max}}$ is a homogenous rational polynomial of degree $g$ and
$\mathrm{vol}_{\Gamma}(\sigma_{\max})$ is the lattice volume of $\sigma_{\max},$ moreover the coefficients of $F_{\sigma_{\max}}$ are totally determined by $\sigma_{\max}$ with marking order of edges and $\Gamma$.
An interesting observation is that the unique K\"ahler-Einstein metric on $\sA_{g,\Gamma}$ guarantees a real Monge-Amp\'ere equation of elliptic type
\begin{equation}\label{Monge-Ampere-eq-0}
   \det(\frac{\partial^2 H}{\partial x_i\partial x_j})_{i,j}=2^{\frac{g(g-1)}{2}}\mathrm{vol}_{\Gamma}(\sigma_{\max})^2 \exp ((g+1)H)
\end{equation}
for  $H:=-\log F_{\sigma_{\max}}$ on the domain
$\{(x_1,\cdots, x_{\frac{g(g+1)}{2}})\in \R^{\frac{g(g+1)}{2}}\,\,|\,\,x_i \leq -C<0 \forall i\}$
(Theorem \ref{global-volume-form-on Siegel varieties}). This Monge-Amp\'ere equation \ref{Monge-Ampere-eq-0} defines a system of rational polynomials, and the system of all coefficients of $F_{\sigma_{\max}}$ gives a nature solution to that system of rational polynomials. Moreover, this system defines
an affine variety $\mathfrak{Q}_g$ over $\Q,$ which is dependent only on $\mathfrak{H}_g.$ The important thing is that the set of all admissible top-dimensional polyhedral cones has an injection into the set $\mathfrak{Q}_g(\Z)$ of all integral point of $\mathfrak{Q}_g$(Theorem \ref{combinatorial-condition-1}). Furthermore, we give a remark in \ref{Remark-mumford-volume} that the real elliptic Monge-Amp\'ere equation \ref{Monge-Ampere-eq-0} and Theorem \ref{combinatorial-condition-1} are always true for all smooth toroidal compactifications whether $D_\infty$ is normal crossing or not.

As an important application of the formula \ref{global-volume-form-introduction} in Algebraic geometry,
we study the asymptotic behaviour of logarithmical canonical line bundles on smooth toroidal compactifications of $\sA_{g,\Gamma}$(Theorem \ref{non-ample-logarithmical-cotangent-bundle}, Theorem \ref{recurrence-intersection-formula} and Theorem \ref{multiple-intersection-formula} in Section 4). We find all logarithmical cotangent bundles degenerate sharply even though $K_{\overline{\sA}_{g,\Gamma}}+D_\infty$'s are big and numerically effective(cf.\cite{Mum77}).
For convenience, we fix a compactification $\overline{\sA}_{g,\Gamma}$ and write $D_\infty=\bigcup\limits_{j} D_j,$ $ D_{i,\infty}:=\bigcup\limits_{j\neq i} D_j\cap D_i$ and $
D_i^*:=D_i\setminus D_{i,\infty}.$ Mumford also shows  that the form $\frac{\sqrt{-1}}{2\pi}\partial\overline{\partial}\log\Phi_{g,\Gamma}$ on $\sA_{g,\Gamma}$ is a current on $\overline{\sA}_{g,\Gamma}$ representing $c_1([K_{\overline{\sA}_{g,\Gamma}}+D_\infty])$ in cohomology class. Using the formula \ref{global-volume-form-introduction}, we get that the restriction of $\partial\overline{\partial}\log\Phi_{g,\Gamma}$ to each $D_i^*$ in sense of limit(denote by $\mathrm{Res}_{D_i}(\partial\overline{\partial}\log\Phi_{g,\Gamma})$) is a closed smooth form on $D_i^*.$ Moreover, the key point is that the form $\mathrm{Res}_{D_i}(\partial\overline{\partial}\log\Phi_{g,\Gamma})$ on $D_i^*$ has Poincar\'e growth on $D_{i,\infty}$ by Mumford's goodness property.
Therefore, we can regard $\frac{\sqrt{-1}}{2\pi}\mathrm{Res}_{D_i}(\partial\overline{\partial}\log\Phi_{g,\Gamma})$ as a positive closed current on $D_i.$
Let $||\cdot||_i$ be an arbitrary Hermitian metric on the line bundle $[D_i]$ for each $D_i,$
we get
\noindent\begin{eqnarray*}
   &&(K_{\overline{\sA}_{g,\Gamma}}+D_\infty)^{\dim_\C\sA_{g,\Gamma}-d}\cdot D_{i_1}\cdots D_{i_d} \\
   &=&(\frac{\sqrt{-1}}{2\pi})^{\dim_\C\sA_{g,\Gamma}-d}\int_{D_{i_l}}\mathrm{Res}_{D_{i_l}}((\partial\overline{\partial}\log\Phi_{g,\Gamma})^{\dim_\C\sA_{g,\Gamma}-d}) \wedge(\bigwedge_{1\leq j\leq d,j\neq l}c_1([D_{i_j}],||\cdot||_{i_j}))
\end{eqnarray*}
for any $d$($1\leq d\leq N-1$) irreducible components $D_{i_1},\cdots, D_{i_d}$ of the boundary divisor $D_\infty$ and any integer $l\in [1,d]$(Theorem \ref{recurrence-intersection-formula}). Furthermore, we observe that irreducible components of $D_\infty$ are all from lower genus Siegel varieties and the type of $\mathrm{Res}_{D_i}(\partial\overline{\partial}\log\Phi_{g,\Gamma})$ is similar with the type of $\partial\overline{\partial}\log\Phi_{g,\Gamma}.$
Due to this structure of self-similarity,
we use the method of recursion to deduce an integral formula : For any $d$($1\leq d\leq N-1$) different irreducible components $D_1,\cdots, D_d$ of $D_\infty$ satisfying  that $\bigcap\limits_{l=1}^d D_l\neq \emptyset,$ there is
\begin{eqnarray*}
   &&(K_{\overline{\sA}_{g,\Gamma}}+D_\infty)^{\dim_\C\sA_{g,\Gamma}-d}\cdot D_1\cdots D_{d}\\
   &=& \int_{\bigcap\limits_{k=1}^d D_k}
   \mathrm{Res}_{\bigcap\limits_{i=1}^d D_i}(\mathrm{Res}_{\bigcap\limits_{i=1}^{d-1} D_i}\cdots(\mathrm{Res}_{D_1}((\frac{\sqrt{-1}}{2\pi}\partial\overline{\partial}\log\Phi_{g,\Gamma})^{\dim_\C\sA_{g,\Gamma}-d})\cdots )).
\end{eqnarray*}
A direct consequence is that if $d\geq g-1$ then
the intersection number $$(K_{\overline{\sA}_{g,\Gamma}}+D_\infty)^{\dim_\C\sA_{g,\Gamma}-d}\cdot D_1\cdots D_{d}=0$$ for any $d$ different irreducible components $D_1,\cdots,D_d$ of $D_\infty.$ Therefore, the divisor $K_{\overline{\sA}_{g,\Gamma}}+D_\infty$ on $\overline{\sA}_{g,\Gamma}$ is never ample(Theorem \ref{multiple-intersection-formula}).

In general, the boundary divisors of  smooth  toroidal compactifications may have self-intersections(cf.\cite{AMRT} and \cite{FC}).
However, in most geometric applications, we would like to have a nice toroidal compactification $\overline{\sA}_{g,\Gamma}$ of $\sA_{g,\Gamma}$ such that the added infinity boundary $D_\infty =\overline{\sA}_{g,\Gamma}\setminus\sA_{g,\Gamma}$ is a normal crossing divisor, for example, in Mumford's work of Hirzebruch's
proportionality theorem in the non-compact case(cf.\cite{Mum77}).  In Section 2, we study the boundaries of  smooth  toroidal compactifications explicitly and  we actually obtain a sufficient and necessary combinatorial condition
for toroidal compactifications with normal crossing boundary divisor(Theorem \ref{Infity-divisor-on-toroidal-compactification} and
Theorem \ref{geometrically-fine-1-1-to non-self-intersections} in Section 2).

On the other hand, the degenerate limits of Abelian varieties have been studied by Mumford,Oda-Seshadri,Nakamura and Namikawa. Deligne's Theorem(cf.\cite{Del71}) shows that the $n$th cohomology group of an arbitrary complex variety $X$ carries a canonical mixed Hodge structure, and that the structure is reduced to an ordinary Hodge structure of pure weight $n$ if $X$ is a complete nonsingular variety.
Thus, toroidal compactifications of Siegel varieties can be related back to degenerations of Abelian varieties or to degenerations of weight one Hodge structures. Roughly, there is a correspondence
between the category of degenerations of Abelian varieties and the category of limits of weight one Hodge structures.
The Hodge-theoretic interpretation of the boundary $\overline{\sA}_{g,\Gamma}\setminus\sA_{g,\Gamma}$
of toroidal compactification is given by Carlson, Cattani and Kaplan in \cite{CCK79}.
Thus, we believe that any rational boundary component(cusp) of a Siegel variety must parameterize some class of mixed Hodge structures.
That is the motivation for our studying cusps of Siegel varieties. Recently, Kato and Usui generalize the work of Carlson-Cattani-Kaplan, and use the idea of logarithmic geometry to give toroidal compactifications of period domains from view of mixed Hodge theory(cf.\cite{KU}).
We explore this topic, and obtain that the Hodge-theoretic interpretation of the boundary of Siegel varieties  coincides with the classic description given by Satake-Baily-Borel in \cite{Sat} and \cite{BB66}. Actually, any cusp of Siegel space $\mathfrak{H}_g$ can be identified with a set of certain weight one polarized mixed Hodge structures(Theorem \ref{PMHS-rational boundary component} in Section 1).

The results of this paper,  the methods and the techniques in this paper, are essential to all locally symmetric varieties. Thus, the results of this paper can be generalized to general locally symmetric varieties by our methods and techniques in this paper.\\

\noindent{\bf Acknowledgements.} We thank Professor Kang Zuo and Professor Ching-Li Chai for useful suggestions, and the second author particularly thanks Doctor Xuan-Ming Ye.
We are grateful to Taida Institute for Mathematical Sciences and Mathematics Department National Taiwan University,
the final version of the paper was finished  during the period of our visiting TIMS.
The second author is also grateful to Mathematics Department Harvard University for hospitality during 2009-2010.

The second author is
supported partially by the NSFC Grant(\#11271070) and LNMS of Fudan University,
he was also supported in part by the NSFC Grant(\#10731030) of Key Project(Algebraic Geometry) during the period 2008-2011.

\vspace{1cm}

\noindent{\bf Notation.}\\

For any real Lie group $\sP,$ $\sP^{+}$ is the identity component of $\sP$ for the real topology.
For any linear space $L_k$ over a field $k,$ a finite field extension $k\subset K$ allows we to
define a $K$-linear space $L_K:= L_k \otimes K.$

Throughout this paper, we fix a real vector space $V_\R$  of dimensional $2g$ and fix a standard symplectic form $\psi=\left(
  \begin{array}{cc}
    0 & -I_g \\
    I_g & 0
  \end{array}
\right)$ on $V_\R.$
    For any non-degenerate skew-symmetric bilinear form $\widetilde{\psi}$
on $V_\R,$ it is known that there is an element $T\in \GL(V_\R)$ such that
$^tT\widetilde{\psi}T=\psi.$ We now fix a  symplectic basis
$\{e_{i}\}_{1\leq i\leq 2g}$ of the symplectic space
$(V_\R,\psi)$ such that
$ \psi(e_i,e_{g+i})=-1 $ for $1\leq i\leq g$ and $\psi(e_i,e_{j}) =0$ for $|j-i|\neq g.$
\begin{itemize}
    \item Denote by $V_\Z:= \oplus_{1\leq i\leq 2g}\Z e_i,$ then
$V_\R=V_\Z\otimes_\Z \R$ and  $V_\Z$ is a standard lattice in
$V_\R.$ In this paper, we fix the lattice $V_\Z$ and fix the
rational space $V_\Q:=V_\Z\otimes_\Z\Q.$

   \item Let $V^{(k)}_\Q$ be
the rational subspace of $V_\Q$ spanned by $\Q$-vectors
$\{e_{k+1},\cdots, e_{g}\}$ for $0\leq k\leq g-1,$ and
$V^{(g)}_\Q:=\{0\}.$ Let $V^{(k)}:=V^{(k)}_\Q\otimes\R$ for $0\leq k\leq g.$

   \item Define
$\Sp(g, \frak R):=\{h\in \GL(V_{\frak R})\,|\,
\psi(hu,hv)=\psi(u,v)\, \forall u,v \in V_{\frak R}\}
$ where $V_{\mathfrak{R}}:=V_\Z\otimes_\Z \mathfrak{R}$ for any $\Z$-algebra $\mathfrak{R}.$
Let
$\Gamma_g(n):=\{\gamma\in
\Sp(g,\Z)\,\,|\,\, \gamma\equiv I_{2g} \mod n\}$ for any integer $ n\geq
2$ and $\Gamma_g=\Gamma_g(1):=\Sp(g,\Z).$
Thus each congruent group $\Gamma_g(n)$ is a  normal subgroup of
$\Sp(g,\Z)$ with finite index.
\end{itemize}

For any free $\Z$-module $W_\Z,$ we use $W$ to represent a linear space over $\Z$( i.e, $W(\mathfrak{R}):=W_\Z\otimes_{\Z} \mathfrak{R}$
   for any $\Z$-algebra $\mathfrak{R}$), and we define $\GL(W)$ to be the algebraic group over $\Q$ representing the functor
($ \Q-\mbox{algebras} \>>> \mbox{Groups}, \,\,\,\,\mathfrak{R}\longmapsto \GL(W_\mathfrak{R})$)
(cf.\cite{Mil}). We always write $\GL(n)$ for $\GL(W)$ if $\rank W_\Z=n.$
For the fixed free $\Z$-module $V_\Z,$ we also define $\Sp(V,\psi)$ to be the
algebraic group over $\Q$ representing the functor
$$ \Q\mbox{-algebras} \>>> \mbox{Groups}, \,\,\,\,\mathfrak{R}\longmapsto \Sp(V,\psi)(\mathfrak{R}):=\Sp(g, \frak R).$$
We know that $\Sp(V,\psi)$ is an algebraic subgroup of $\GL(V).$
\begin{itemize}
     \item Two subgroups $S_1$ and $S_2$ of $\Sp(g,\Q)$ are commensurable if $S_1\cap S_2$ has finite index in both $S_1$ and $S_2.$
   A subgroup $\Gamma\subset \Sp(g,\Q)$ is arithmetic if
$\rho(\Gamma)$ is commensurable with $\rho(\Sp(g,\Q))\cap \GL(n,\Z)$ for some embedding
$\rho: \Sp(V,\psi)\> \hookrightarrow>> \GL(n).$
By a result of Borel, a subgroup $\Gamma\subset \Sp(g,\Q)$ is arithmetic if and only if
that $\rho^{'}(\Gamma)$ is commensurable with $\rho^{'}(\Sp(g,\Q))\cap \GL(n^{'},\Z)$ for every embedding
$\rho^{'}: \Sp(V,\psi)\> \hookrightarrow>> \GL(n^{'})$(cf.Chap. VI. \cite{Mil}).
Thus a subgroup $\Gamma\subset \Sp(g,\Z)$ is arithmetic if and only if
$[\Sp(g,\Z):\Gamma]<\infty.$

     \item Let $k^{'}$ be a subfield of $\C$ and $\widetilde{V}_{k^{'}}$ a $k^{'}$-vector space. Let $\GL(\widetilde{V})$ be an algebraic group
      defined over $k^{'}$ as above.
     An automorphism $\alpha$ of a $k^{'}$-vector space  is defined to be neat
(or torsion free) if its eigenvalues in $\C$ generate a torsion free subgroup of $\C.$ An element $h\in \Sp(g,\Q)$  is said to be neat(or torsion free) if $\rho(h)$ is neat for one faithful representation $\rho: \Sp(V,\psi)\> \hookrightarrow >>\GL(\widetilde{V}).$  A subgroup $\Gamma\subset\Sp(g,\R)$ is said to be neat if all elements of $\Gamma$ are torsion free.
We have that if $h\in \Sp(g,\Q)$ is neat then $\rho^{'}(h)$ is neat for
every representation $\rho^{'}$ of $\Sp(V,\psi)$ defined over $k^{'}$(cf.\cite{Mil}).
For example,  the $\Gamma_g(n)$  is a neat arithmetic subgroup of $\Sp(g,\Q)$ if $n\geq 3.$
    \item For any arithmetic subgroup $\Gamma\subset\Sp(g,\Q)$($g\geq 2$), we  can find a neat subgroup $\Gamma^{'}\subset \Gamma$ of finite index. In fact, the neat subgroup $\Gamma^{'}$ can be given by congruence conditions(cf.\cite{Mil}).
\end{itemize}

The Siegel space $\frak H_g$ of degree $g$
is a set of all symmetric matrices over $\C$ of degree $g$ whose imaginary parts are positive defined.
The simple Lie group $\Sp(g,\R)$ acts transitively on $\mathfrak{H}_g$ as
$\left(
         \begin{array}{cc}
           A & B \\
           C & D \\
         \end{array}
       \right)\bullet\tau:=\frac{A\tau+B}{C\tau+D}.$  Let $o:=\sqrt{-1}I_g$ be a fixed  point on $\mathfrak{H}_g.$
\begin{itemize}
  \item The stabilizer of $o$
is isomorphic to the unitary group $\mathrm{U}(g).$
We identify $o$ with $\pi(e)$ where
$\pi: \Sp(g,\R)\to \mathfrak{H}_g$ is the standard projection and $e\in \Sp(g,\R)$ is the identity.
The element $s_o:=\left(
  \begin{array}{cc}
    0 & I_g \\
    -I_g & 0
  \end{array}
\right)$ in $\Sp(g,\R)$ acts as an involution of $\mathfrak{H}_g$ leaving $o$ as the only isolated fixed point.
Therefore the Siegel space $\mathfrak{H}_g$ is an non-compact Hermitian symmetric space.

    \item From now on, let $G_\R=\Sp(g,\R)$ be the real
Lie group with Lie algebra
$\mathfrak{g}=\mathfrak{g}_\R:=\mathrm{Lie}(G_\R),$ and regard $K_\R:=\mathrm{U}(g)$ as a real
Lie group with Lie algebra
$\mathfrak{f}=\mathfrak{f}_\R=\mathrm{Lie}(K_\R).$
With respect to the standard symplectic basis,
we have :
\begin{eqnarray*}
  \mathfrak{g}_{\R}
   &=& \{ \left(
                      \begin{array}{cc}
                        A & B \\
                        C & D \\
                      \end{array}
                    \right)\in M(2g, \R)
  \,\,|\,\, D=-^tA,\, ^tB=B,\,\, ^tC=C \},\\
\mathfrak{f}_{\R}&=& \{ \left(
                      \begin{array}{cc}
                        A & B \\
                       -B & A \\
                      \end{array}
                    \right)\in M(2g, \R)
  \,\,|\,\, ^tA=-A,\,\, ^tB=B
  \}\cong u(g)
\end{eqnarray*}
\end{itemize}

   A \textbf{Siegel variety} is defined to be
$\sA_{g,\Gamma}:=\Gamma\backslash\mathfrak{H}_g,$
where $\Gamma$ is an arithmetic subgroup of $\Sp(g,\Q).$ Any Siegel variety is a normal quasi-project variety.
\begin{itemize}
  \item Any  neat arithmetic subgroup $\Gamma$ of $\Sp(g,\Q)$ acts freely on the Siegel Space $\mathfrak{H}_g,$ so that the induced
$\sA_{g,\Gamma}$ is a regular quasi-projective complex variety of
dimension $g(g+1)/2.$
 A \textbf{Siegel variety of degree $g$ with level $n$} is defined to be
$\sA_{g,n}:=\Gamma_g(n)\backslash\mathfrak{H}_g.$
Thus, the Siegel varieties $\sA_{g,n}\,\,$$n\geq 3$  are
quasi-projective complex manifolds.\\
\end{itemize}



\section{Cusps of Siegel varieties from the viewpoint of mixed Hodge theory}

\vspace{0.5cm}

\subsection{Typical homogenous rational polarized VHS on Siegel space}
Define $$\mathrm{U}^1:=\mathrm{U}(1)=\{ |z|=1 \,|\,z\in \C\}.$$
Let $\tau\in \mathfrak{H}_g$ be an arbitrary point.  Let
$T_\tau(\mathfrak{H}_g)$ be the real tangent space at $\tau$ and
$J_\tau$ the complex structure on $T_\tau(\mathfrak{H}_g)$
induced by the global complex structure $J$ of
$\mathfrak{H}_g.$

The Siegel space $\mathfrak{H}_g$ has a natural $\Sp(g,\R)$-invariant K\"ahler metric(Bergman metric). Regard $\mathfrak{H}_g$ as a Riemannian
symmetric space, we have a basic fact:
\begin{lemma}[Cf.\cite{Hel}\&\cite{Del73}]\label{basic lemma}
For any $z=a+\sqrt{-1}b\in \mathrm{U}^1,$ there is a unique
isometric automorphism $u_\tau(z):\mathfrak{H}_g \to \mathfrak{H}_g$ such
that
\begin{itemize}
    \item $u_\tau(z)$ maps $\tau$ to $\tau,$
    \item $du_\tau(z): T_\tau(\mathfrak{H}_g)\to T_\tau(\mathfrak{H}_g)$
    is given by $v \mapsto  z\cdot v:= av+bJ_\tau(v).$
\end{itemize}
\end{lemma}
The uniqueness of $u_\tau(z)$ ensures that $u_\tau(1)=\Id$ and
$$u_\tau(z_1z_2)=u_\tau(z_1)\circ u_\tau(z_2)=u_\tau(z_2)\circ u_\tau(z_1),\,\, \forall z_1, z_2\in\mathrm{U}^1.$$
Therefore, we obtain a group  homomorphism
\begin{equation}\label{basic lemma-appendix}
    u_\tau:\mathrm{U}^1 \>>>
\Aut(\mathfrak{H}_g)=\Sp(g,\R)/\{\pm I_{2g}\},\,\, z\mapsto u_\tau(z).
\end{equation}
Furthermore,  we can lift $u_\tau^2$ to be a group homomorphism
$ h_\tau: \mathrm{U}^1\>>> \Sp(g,\R).$
Since $\Sp(g,\R)$ is a simply-connected topological space. The
uniqueness of $u_\tau(z)$ guarantees that
\begin{equation}\label{Hodge-structure-1}
h_{M(\tau)}= M h_\tau M^{-1},  \forall \tau\in
\mathfrak{H}_g, \forall M\in \Sp(g,\R).
\end{equation}
The $u^2_o(\sqrt{-1})$ is the involution of $\mathfrak{H}_g$
fixing the point $o=\sqrt{-1}I_g,$ so $h_o(\sqrt{-1})$ is one of
 $\{\pm\left(
  \begin{array}{cc}
    0 & I_g \\
    -I_g & 0
  \end{array}
\right)\}.$  Since that $h_o: \mathrm{U}^1\to \Sp(g,\R)$ is a
group homomorphism and that $h_o(\mathrm{U}^1)$ is a commutative
group in $U(g),$ we must have $h_o(\sqrt{-1})=\left(
  \begin{array}{cc}
    0 & I_g \\
    -I_g & 0
  \end{array}
\right).$

Let $\G_m:=\GL(1)$ be an algebraic torus. A priori
$\G_m=\G_{m/\Q}$ is defined over $\Q,$ and thus
$\G_m(k)=k^{\times}$ for any field $k$ containing $\Q.$ Following
1.4.4 in \cite{Del73}, we define
 $\mathrm{GSp}(V,\psi)$ to be the quotient
 $\Sp(V,\psi)\times \G_m$ by the central subgroup $\{e,(\epsilon, -1) \},$
it is an algebraic group over $\Q.$
Let $\iota$ be the composed homomorphism
 $$\iota: \Sp(V,\psi)\>\hookrightarrow>>\Sp(V,\psi)\times \G_m\>>>
 \mathrm{GSp}(V,\psi),$$
and $t:\mathrm{GSp}(V,\psi)\to\G_m$ the homomorphism given by
 $$\mathrm{GSp}(V,\psi)(\R)\to \G_m(\R),\,\,\,   [(g,\lambda)]\mapsto \lambda^2.$$
We then have a split exact sequence
 \begin{equation}\label{Sp-to-GSp}
1\>>> \Sp(V,\psi)\>\iota>>\mathrm{GSp}(V,\psi)\>t>>\G_m\>>>1.
 \end{equation}

According to Deligne's Hodge theory(cf.\cite{Del71},\cite{Del73},\cite{Del79}), each  $h_\tau$
actually corresponds to a rational Hodge structure on $V_\R$ of pure weight one given by a Hodge
filtration
$$F_{\tau}^\bullet=\big( F_{\tau}^2\subset F_{\tau}^1\subset
F_{\tau}^0 \big)= \big(0\subset F_{\tau}^1\subset
V_\C \big) $$ on $V_\C$ such that
$$F_{\tau}^1=\mbox{ the subspace of $V_\C$ spanned by the column vectors of  $\left(
                                                                           \begin{array}{c}
                                                                             \tau \\
                                                                             I_g
                                                                           \end{array}
                                                                         \right).$ }$$
Moreover, we observe that $h_\tau$ is polarized automatically by the standard
symplectic form $\psi,$ i.e., $F_{\tau}^\bullet$ satisfies the
following two Riemann-Hodge bilinear relations:
\begin{myenumii}
  \item  $ \psi(F_{\tau}^p,F_{\tau}^q) = 0 \mbox{ for } p+q>1.$

  \item  The Hermitian form
$V_\C\times V_\C\>>> \C\,\,\,\,(x,y)\mapsto \psi(C_\tau(x), \overline{y})$
is positive definite.
\end{myenumii}
\begin{proposition}[Satake-Deligne \cite{Sat}\&\cite{Del73}]\label{Borel-embedding}
Define
$$
\mathfrak{S}_g=\mathfrak{S}(V_\R,\psi):=\{F^1\in \mathrm{Grass}(g,
V_\C)\,\,\, | \,\, \psi(F^1,F^1)=0, \sqrt{-1}\psi(F^1,
\overline{F^1})>0\}.
$$
The map
$h: \mathfrak{H}_g\>\cong>> \mathfrak{S}_g\,\,\,\,
   \tau  \longmapsto F^1_\tau $
identifies the Siegel space $\mathfrak{H}_g$ with the period domain
$\mathfrak{S}_g.$ Moreover, the map $h$ is biholomorphic.
\end{proposition}

Let $\pi: \Sp(g,\R)\to \mathfrak{H}_g\cong \Sp(g,\R)/ \mathrm{U}(g)$
be the standard projection mapping the identity $e$ to the point
$o=\sqrt{-1}I_g.$ Then, we have the following commutative diagrams of differentials
\begin{equation}\label{automorphism-to-lie morphism-1}
\begin{CDS}
\mathfrak{g}\> \Ad h_o(z)  >> \mathfrak{g} \\
\V d\pi V  V \novarr \V   V d\pi V \\
T_o(\mathfrak{H}_g) \>  d(h_o(z))|_o >> T_o(\mathfrak{H}_g),
\end{CDS} \,\, \forall z\in \mathrm{U}^1.
\end{equation}
In particular, we obtain that $\sigma:=\Ad(h_o(\sqrt{-1}))$ is an
involution of the Lie algebra $\mathfrak{g}=\mathrm{Lie}(\Sp(g,\R)),$
and there is a decomposition of $\mathfrak{g}$(orthogonal under the no-degenerate
Killing form)
$\mathfrak{g}=\mathfrak{f}\oplus\mathfrak{p}$ such that
$\mathfrak{f}=\mathrm{Lie}(K)\cong \mathrm{Lie}(\mathrm{U}(g))$ and
$\mathfrak{p}\cong T_o(\mathfrak{H}_g).$ Since $h_o(\mathrm{U}^{1})$ is a
commutative subgroup in $\Sp(g,\R),$ both  $\mathfrak{f}$ and
$\mathfrak{p}$
are $\Ad(h_o(\mathrm{U}^{1}))$-invariant. 
Moreover, $\Ad(h_o(\exp(\frac{2\pi\sqrt{-1}}{8})))|_\mathfrak{p}$ is compatible
with the complex structure $J_o$ of the tangent space
$T_o(\mathfrak{H}_g).$  Thus, the collection $\{J_{\tau}\}_{\tau\in
\mathfrak{H}_g}$ is $\Sp(g,\R)$-invariant and is same as the original
global complex structure $J$ on $\mathfrak{H}_g.$

Since $J_{\mathfrak{o}}$ gives a
decomposition $\mathfrak{p}_\C=\mathfrak{p}^{+}\oplus\mathfrak{p}^{-}$ into
$\pm\sqrt{-1}$-eigenspaces, we have :
\begin{corollary}\label{character of adjoint homomorphism}
For every $z\in \mathrm{U}^1,$ the adjoint homomorphism $\Ad(h_o(z)): \mathfrak{g}_\C\to \mathfrak{g}_\C$ is given  by
\begin{eqnarray*}
  \Ad(h_o(z))|_{\mathfrak{p}^{+}}: \mathfrak{p}^{+} \>>> \mathfrak{p}^{+} & & v\longmapsto z^2v, \\
  \Ad(h_o(z))|_{\mathfrak{p}^{-}}: \mathfrak{p}^{-} \>>> \mathfrak{p}^{-} & & v\longmapsto z^{-2}v, \\
  \Ad(h_o(z))|_{\mathfrak{f}_\C}\,\, :\,\,\,\mathfrak{f}_\C  \>>> \mathfrak{f}_\C  & & v\longmapsto v.
 \end{eqnarray*}
\end{corollary}

Now, we have the following composite homomorphism
\begin{equation}\label{Hodge-filtration-on-Lie-alg-1}
    \mathrm{U}^{1} \>h_o>>\Sp(g,\R)\>\Ad>>\GL(\mathfrak{g}).
\end{equation}
It determines a weight zero real Hodge structure on $\mathfrak{g}$(cf.\cite{Del73},\cite{Del79}). Actually, this real Hodge structure on $\mathfrak{g}$  has
  type of $\{(-1,1),(0,0),(1,-1)\}$  by the corollary \ref{character of adjoint homomorphism}  :
$$\mathfrak{g}^{0,0}=\mathfrak{f}_\C \mbox{ , }\mathfrak{g}^{-1,1}=\mathfrak{p}^{+}\mbox{ and }\mathfrak{g}^{1,-1}=\mathfrak{p}^{-}.$$

Denote by
$$F^1_o(\mathfrak{g}):=\mathfrak{g}^{1,-1},\,\, F^0_o(\mathfrak{g}):=\mathfrak{g}^{0,0}\oplus\mathfrak{g}^{1,-1}\mbox{ and } F^{-1}_o(\mathfrak{g})=\mathfrak{g}_\C.$$
\begin{lemma}\label{Hodge-filtration-on-Lie-alg}For $\mathfrak{g}=\mathrm{Lie}(\Sp(g,\R)),$ we have that
\begin{equation*}
    F^p_o(\mathfrak{g})=\{X\in \mathfrak{g}_\C \,|\, X(F^s_o)\subset F^{s+p}_o\,\,\forall s\} \,\, p=-1,0,1,
 \end{equation*}
where  $F^{\bullet}_o$ is the Hodge filtration on $V_\C$ given by $h_o.$
\end{lemma}
\begin{proof}
The composite homomorphism
$\mathrm{U}^{1} \>h_o>> \Sp(g,\R) \>\Ad
>> \GL(\End(V_\R))$ determines a weight zero  real Hodge
structure on $\End(V_\R)=\mathrm{Lie}(\GL(V_\R)).$ Also, the
composite homomorphism $\mathrm{U}^{1} \>h_o>> \Sp(g,\R) \>\Ad
>> \mathfrak{g}$  gives  a weight zero  real Hodge
structure on $\mathfrak{g}.$
The inclusion $\mathfrak{g}\subset \End(V_\R)$ is  a morphism of
Hodge structures by the following commutative diagrams
\begin{equation}\label{Hodge-filtration-on-Lie-alg-2}
  \begin{CDS}
\mathfrak{g}\>\subset >>\End(V_\R)  \\
\V \Ad(g)  V V \novarr \V V \Ad(g) V \\
\mathfrak{g}\>\subset >>\End(V_\R)
\end{CDS}  \,\,\,\hspace{1cm} \forall g \in \Sp(g,R).
\end{equation}

Write $H^{s,1-s}_o:=F^s_o/F^{s+1}_o$ for $s=0, 1.$ We
obtain that
\begin{equation*}
    \mathfrak{g}^{i,-i}=\{X\in \mathfrak{g}_\C \,|\, X(H^{s,1-s}_o)\subset
    H^{s+i,1-s-i}_o\,\,\forall s\} \,\,\mbox{ for } i=-1,0,1
\end{equation*}
as a subset of $\End(V_\R)^{i,-i}.$
\end{proof}
\begin{myrem}
$F^0_o(\mathfrak{g})$ and $F^1_o(\mathfrak{g})$ are then Lie subalgebras of
$\mathfrak{g}_\C.$ In particular, both $\mathfrak{g}^{-1,1}$ and $\mathfrak{g}^{1,1}$ are commutative complex Lie subalgebras of $\mathfrak{p}_\C.$
\end{myrem}

\begin{corollary}[Deligne \cite{Del79}]\label{homogeneus vhs on Siegel space}
Gluing  Hodge structures $h_{\tau}\, \forall\tau \in\mathfrak{H}_g$ altogether,
the local system $\V:=V_\Q\times \mathfrak{H}_g$  underlies a
homogenous rational variation of  polarized Hodge structure of
weight one on $\mathfrak{H}_g.$
\end{corollary}
\begin{proof}
The relation
$
    \mathfrak{g}^{-1,1}=\{X\in \mathfrak{g}_\C \,|\, X(H^{s,1-s}_o)\subset H^{s-1,2-s}_o\,\,\forall s\}
$
in the lemma \ref{Hodge-filtration-on-Lie-alg} shows that the
holomorphic tangent bundle of $\mathfrak{H}_g$ is horizontal(cf
\cite{Sch73}).
\end{proof}


Suppose that  $\Gamma$ is a neat arithmetic subgroup of $\Sp(g,\Q),$  we immediately have:
\begin{itemize}
  \item Since $\mathfrak{H}_g$ is simply connected, the fundamental group of $\sA_{g,\Gamma}$
has
$\pi_1(\sA_{g,\Gamma}, o)\cong\Gamma.$

  \item There is  a natural local system
$\V_{g,\Gamma}:=V_\Q\times_{\Gamma}\mathfrak{H}_g$
 on $\sA_{g,\Gamma}$ given by the fundamental representation
$\rho:\pi_1(\sA_{g,\Gamma}, o)\>>> \mathrm{GSp}(V,\psi)(\Q).$
\end{itemize}

\begin{proposition}\label{vhs on Siegel varieties}
Let $\Gamma$ be a neat  arithmetic subgroup of $\Sp(g,\Q).$ We have :
\begin{myenumi}
\item The local system  $\V_{g,\Gamma}$ underlies a rational variation of
 polarized Hodge structure on $\sA_{\Gamma}.$ Moreover,  the
associated period map
\begin{equation}\label{period-mapping}
 h_{\Gamma}: \sA_{g,\Gamma}\>\cong>> \Gamma \backslash\mathfrak{S}_g
\end{equation}
 is induced by the isomorphism $h$ in the proposition \ref{Borel-embedding}.

\item Let  $\widetilde{\sA}_{g,\Gamma}$ be an arbitrary smooth
compactification of $\sA_{g,\Gamma}$ with simple normal crossing
divisor
 $D_\infty:=\widetilde{\sA}_{g,\Gamma}\setminus\sA_{g,\Gamma}.$ Around the boundary divisor
 $D_\infty,$ all local monodromies of any rational PVHS $\widetilde{\mathbf{H}}$ on $\sA_{g,\Gamma}$  are unipotent.
\end{myenumi}
\end{proposition}
\begin{proof}
By the corollary \ref{homogeneus vhs on Siegel space}, the local
system $\V=V_\Q\times \mathfrak{H}_g$ admits a homogenous rational variation of
 polarized Hodge structure on $\mathfrak{H}_g.$ The arguments
in Section $4$ of \cite{Zuc81} show that the VHS attached to $\V$
can induce a locally homogenous rational  variation of polarized Hodge
structure on the local system $\V_{\Gamma},$ and so the period
map $h_{\Gamma}$ is given by the $\Sp(g,\R)$-equivariant
isomorphism $h: \mathfrak{H}_g\>\cong>> \mathfrak{S}_g$ in  the
proposition \ref{Borel-embedding}.

It is well-known that  all local monodromies of the  rational PVHS
$\widetilde{\mathbf{H}}$ around  $D_\infty$ are
quasi-unipotent(i.e., all eigenvalues of monodromies are roots of
the unit). Since $\mathfrak{H}_{g}$ is simply-connected and $\Gamma$
is neat,
all eigenvalues of monodromies must be the unity, and so these monodromies are unipotent.
\end{proof}

\subsection{Cusps on Siegel spaces}

Let  $F^0_o(G_\C)$(resp. $F^1_o(G_\C)$) be the subgroup of $G_\C$ satisfying that
$\mathrm{Lie}(F^0_o(G_\C))=F^0_o(\mathfrak{g})$ (resp. $\mathrm{Lie}(F^1_o(G_\C))=F^1_o(\mathfrak{g})$).
The lemma \ref{Hodge-filtration-on-Lie-alg} shows immediately that
$F^1_o(G_\C)$ is the parabolic subgroup preserving the Hodge
filtration $F_o^\bullet,$ and that $F^1_o(G_\C)$ is the unipotent
radical of $F^0_o(G_\C).$

\begin{proposition}[Harish-Chandra Embedding Theorem cf.\cite{Del73}]\label{Harish-Chandra-embedding}
 The set $\mathfrak{S}_g$ is contained in the largest cell of
$\check{\mathfrak{S}_g}.$ Precisely,  the map
$$\zeta:\mathfrak{S}_g\>>> \overline{F^1_o(\mathfrak{g})}\,\,\, h\longmapsto n_h $$
identifies $\mathfrak{S}_g$ with
a bounded open subset of  $\overline{F^1_o(\mathfrak{g})},$
where the space $\overline{F^1_o(\mathfrak{g})}$ is the closure of $F^1_o(\mathfrak{g})$ in $\mathfrak{g}=\mathrm{Lie}(\Sp(g,\R))$ and  the element $n_h\in \overline{F^1_o(\mathfrak{g})}$ is determined by $h=\exp(n_h)h_o.$
\end{proposition}

This Harish-Chandra embedding allows us to
define the closure
$
\overline{\mathfrak{S}}_{g}:= \exp(\overline{\zeta(\mathfrak{S}_g)}^{cl})h_o
$ and the boundary $\partial\mathfrak{S}_g:=\overline{\mathfrak{S}}_{g}\setminus\mathfrak{S}_g,$ where $\overline{\zeta(\mathfrak{S}_g)}^{cl}$
is the closure of $\zeta(\mathfrak{S}_g)$ in
$\overline{F^1_o(\mathfrak{g})}.$
\begin{corollary}\label{boundary of Siegel spaces}
\begin{eqnarray*}
\overline{\mathfrak{S}}_{g}   &=& \{F^1\in \mathrm{Grass}(g, V_\C)\,\,|\,\, \psi(F^1,F^1)=0, \sqrt{-1}\psi(F^1, \overline{F^1})\geq 0\}, \\
 \partial\mathfrak{S}_g  &=& \{F^1\in \check{\mathfrak{S}}_g \,\,|\sqrt{-1}\psi(F^1, \overline{F^1})\geq 0, \psi(\cdot, \overline{\cdot}) \mbox{ is degenerate on } F^1\}.\\
                     &=&\{F^1\in \overline{\mathfrak{S}}_{g} \,\,|F^1\cap \overline{F^1} \mbox{ is a non trivial isotropic space }\}
\end{eqnarray*}
\end{corollary}

A \textbf{boundary component} of the space $\mathfrak{S}_g=\mathfrak{S}(V,\psi)$ is a subset in $\partial\mathfrak{S}_g$ of type
$$\mathfrak{F}(W_\R):=\{ F^1\in \overline{\mathfrak{S}}_{g} \,\,|\,\,  F^1\cap \overline{F^1}=W_\R\otimes\C \mbox{ where $W_\R$ is an isotropic real subspace of $V_\R$ } \}.$$
A \textbf{$k$-th boundary component} of $\mathfrak{S}(V_\R,\psi)$ is a $\mathfrak{F}(W_\R)$ with $\dim_\R W_\R=k.$
We note $\mathfrak{F}(\{0\})=\mathfrak{S}_g$ and other boundary
components are subsets of $\partial\mathfrak{S}_g.$
The group $\Sp(g,\R)$
has a natural action  on the set of all boundary components as
$M\bullet\mathfrak{F}(W_\R):=\mathfrak{F}(M(W_\R))\,\,\,\,\forall M\in \Sp(g,\R).$
The compact space $\overline{\mathfrak{S}}_{g}$ is a disjoint union of all boundary components, and
$$\partial\mathfrak{S}_g=\bigcup^{\circ}_{\mbox{nontrivial isotropic }W_\R\subset V_\R}\mathfrak{F}(W_\R).$$
The \textbf{normalizer} of a boundary component $\mathfrak{F}$ is the
subgroup $\sN(\mathfrak{F})$ of $\Sp(g,\R)$ containing of those $g$
such that $g\mathfrak{F}=\mathfrak{F}.$
A boundary component $\mathfrak{F}$ is said to be \textbf{rational} if
its normalizer $\sN(\mathfrak{F})$ is defined over $\Q$(i.e., there is an algebraic subgroup $N^{\mathfrak{F}} \subset \Sp(V,\psi)$
defined over $\Q$ such that $\sN(\mathfrak{F})^{+}=N^{\mathfrak{F}}(\R)^{+}$ cf.\cite{Mil}). For convenience, a($k$-th) rational boundary component is called to be a($k$-th) \textbf{cusp}(or a cusp of depth $k$).
We note that for any two isotropic rationally-defined subspaces
$W_2\subset W_1$ of $V_\R,$ the set $\{F/W_2\,\,|\,\, F\in
\mathfrak{F}(W_1)\}$ is a cusp of $\mathfrak{S}(W_2^\perp/W_2,\psi).$
A $k$-th cusp of $\mathfrak{H}_g$ always corresponds to a  Siegel space of genus $g-k,$ in particular $\mathfrak{F}(V^{(k)})\cong \mathfrak{H}_k.$
\begin{remark}[Cf.(4.15)-(4.16) $\S5$ \cite{Nam}]\label{remark on rational boundary component}
The following conditions are equivalent for a $k$-th boundary
component $\mathfrak{F}(W)$ :
\begin{myenumiii}
\item $\mathfrak{F}(W)$ is rational;

\item $W$ is an isotropic rationally-defined subspaces, i.e.,
$W=W_\Q\otimes\R$ where $W_\Q$ is an isotropic subspace of
$V_\Q$(and so $W_\R=(W_\R\cap V_\Q)\otimes \R$);

\item $\exists M\in \Sp(g,\Q)$ such that $W=M(V^{(k)});$

\item $\exists M\in \Sp(g,\Z)$ such that $W=M(V^{(k)}).$
\end{myenumiii}
\end{remark}
Since the symplectic  group $\Sp(V,\psi)$ is simple, we have :
\begin{proposition}[Baily-Borel  cf.\cite{BB66},\cite{Del73},\cite{AMRT}]
The map $\mathfrak{F}\mapsto N^{\mathfrak{F}}$ is a bijection between the set of proper boundary components
of $\mathfrak{S}_g$ to the set of  maximal parabolic algebraic subgroups of $\Sp(V,\psi).$ Moreover,
the boundary component $\mathfrak{F}$ is rational if and only if $ N^{\mathfrak{F}}$  is a
maximal rational parabolic algebraic subgroup of $\Sp(V,\psi).$
\end{proposition}


For any  proper boundary component  $\mathfrak{F}$  of $\mathfrak{S}_g,$
there is an isotropic subspace $V_{\mathfrak{F}}$ of $(V_\R,\psi)$
with $\mathfrak{F}(V_{\mathfrak{F}})=\mathfrak{F},$ and so there will be an
increasing filtration of $V_\R$
\begin{equation}\label{Weight-filtration-1}
   W_{\bullet}^{\mathfrak{F}}(\R)=\big(0\subset  W_{0}^\mathfrak{F}\subset W_{1}^\mathfrak{F}
   \subset W_{2}^\mathfrak{F}\big):=\big(0\subset  V_{\mathfrak{F}}\subset(V_{\mathfrak{F}})^\perp
   \subset V_\R\big).
\end{equation}
This filtration \ref{Weight-filtration-1} corresponds to a unique
morphism
$w^{'}_{\mathfrak{F}}: \G_m\to \mathrm{GSp}(V,\psi)$ defined over $\R.$
We also define the cocharacter
$w_{\mathfrak{F}}:= w^{'}_{\mathfrak{F}}w_0^{-1} : \G_m\to\Sp(V,\psi)$
where
$$w_0^{-1} :
\G_m\>(e,(\cdot)^{-1})>>\Sp(V,\psi)\times\G_{m}\>>>\mathrm{GSp}(V,\psi).$$
We note that $w_{\mathfrak{F}}$ is defined over $\R,$ and
that $w_{\mathfrak{F}}$ is defined over $\Q$ if and only if the boundary
component $\mathfrak{F}$ is rational.
We have the following composed homomorphism :
\begin{equation}\label{weight-morphism-2}
   \G_m(\R)\> w_{\mathfrak{F}} >>\Sp(g,\R)\>\Ad>>\GL(\mathfrak{g})\subset \GL(\End(V_\R)).
\end{equation}
Define
$\End(V_\R)^{i}:=\{ X\in\End(V_\R)\,\,|\,\, \mathrm{Ad}(w_{\mathfrak{F}}(\lambda))X=\lambda^i X,\forall \lambda\in \G_m(\R) \}$
and
\begin{equation}\label{decomposition of lie algebra}
 \mathfrak{g}^{i}:=\End(V_\R)^{i}\cap \mathfrak{g}=\{ X\in\mathfrak{g}\,\,|\,\, \mathrm{Ad}(w_{\mathfrak{F}}(\lambda))X=\lambda^i X,\forall \lambda\in \R^\times \}.
\end{equation}
We then have that
\begin{eqnarray*}
  \mathfrak{g} &=&\mathfrak{g}^{-2}\oplus\mathfrak{g}^{-1}\oplus\mathfrak{g}^{-0}\oplus\mathfrak{g}^{1}\oplus\mathfrak{g}^{2} \\
  & \subset & \End(V_\R)^{-2}\oplus\End(V_\R)^{-1}\oplus\End(V_\R)^{0}\oplus\End(V_\R)^{1}\oplus\End(V_\R)^{2}=\End(V_\R)
\end{eqnarray*}
by \ref{Hodge-filtration-on-Lie-alg-2} and \ref{weight-morphism-2}. Thus the weight morphism $w_{\mathfrak{F}}$ endows an increasing filtration on Lie algebra $\mathfrak{g}$(respectively on $\End(V_\R)$)
$W_{\bullet}^\mathfrak{F}(\mathfrak{g})=(0\subset W_{-2}^\mathfrak{F}(\mathfrak{g})\subset \cdots\subset W_{2}^\mathfrak{F}(\mathfrak{g})=\mathfrak{g})$
with
\begin{equation}\label{parabolic Lie algebra}
 W_{-2}^\mathfrak{F}(\mathfrak{g})=\mathfrak{g}^{-2}, W_{-1}^\mathfrak{F}(\mathfrak{g})=\mathfrak{g}^{-2}\oplus \mathfrak{g}^{-1}
\mbox{ and } W_{0}^\mathfrak{F}(\mathfrak{g})=\mathfrak{g}^{-2}\oplus \mathfrak{g}^{-1}\oplus \mathfrak{g}^{0}.
\end{equation}

\begin{lemma}\label{weight-morphism-3}For $\mathfrak{g}=\mathrm{Lie}(\Sp(g,\R)),$
$$
W_{s}^\mathfrak{F}(\mathfrak{g})=\{X\in \mathfrak{g}\,\,|\,\, X(W^\mathfrak{F}_{l})\subset W_{s+l}^\mathfrak{F}\,\,\forall l\}
$$
and
\begin{equation*}
  [W_{s}^\mathfrak{F}(\mathfrak{g}),W_{t}^\mathfrak{F}(\mathfrak{g})]\subset W_{s+t}^\mathfrak{F}(\mathfrak{g}).
\end{equation*}
\end{lemma}
\begin{proof}
Choose subspaces $V^j$ of $V_\R$ such that we can write
$W^{\mathfrak{F}}_{i} =\oplus_{j\leq i}V^j,$ and define $V^j=\{0\}$ if
$j\notin [-2,2].$ Similarly as in the proof of the lemma
\ref{Hodge-filtration-on-Lie-alg}, we obtain
\begin{equation*}
\mathfrak{g}^i=\{X\in \mathfrak{g}\,\,|\,\, X(V^j)\subset V^{j+i} \,\,
\forall j \}\,\, \mbox{ for any integer } i\in[-2,2].
\end{equation*}
by the commutative diagram \ref{Hodge-filtration-on-Lie-alg-2} and
the definition of $\End(V_\R)^{i}$ show that

Therefore, the first statement is true and the second statement
follows it.
\end{proof}

\begin{corollary}[Cf.\cite{Del73},\cite{AMRT},\cite{Mil90}]\label{algebraic subgroups by boundary-component}
Let $\mathfrak{F}=\mathfrak{F}(W)$ be a boundary component of the Siegel space $\mathfrak{H}_g.$ We have :
\begin{itemize}
    \item $W_{0}^\mathfrak{F}(\mathfrak{g}),$  $W_{-1}^\mathfrak{F}(\mathfrak{g})$ and
$\mathfrak{g}^0$ are Lie subalgebras of $\mathfrak{g}=\mathrm{Lie}(\Sp(g,\R));$
    \item $W_{-2}^\mathfrak{F}(\mathfrak{g})$ and $\mathfrak{g}^2$ are commutative Lie
subalgebras of $\mathfrak{g}.$
\end{itemize}

Let
$\big(W_{-1}\subset  W_{0}\subset W_{1}\subset W_2\big):= \big(0\subset  W\subset W^{\perp}\subset V_\R\big)$
be the filtration corresponding to $\mathfrak{F}.$ For each integer  $i$ in $[-2,0],$ we define  $W_{i}^\mathfrak{F}(\Sp(V,\psi))$ to be the algebraic subgroup of $\Sp(V,\psi)$
of elements acting as the identity map on $\bigoplus\limits_{p}W_p/W_{p-i}.$
We have :
\begin{itemize}
  \item $W_0^\mathfrak{F}(\Sp(V,\psi))=N^\mathfrak{F}$ is a parabolic subgroup of $\Sp(V,\psi)$ with Lie algebra $W_{0}^\mathfrak{F}(\mathfrak{g});$
  \item $\sW^\mathfrak{F}:=W_{-1}^\mathfrak{F}(\Sp(V,\psi))$ has Lie algebra $W_{-1}^\mathfrak{F}(\mathfrak{g}),$ and it is the unipotent radical of $N^\mathfrak{F};$
  \item $U^\mathfrak{F}:=W_{-2}^\mathfrak{F}(\Sp(V,\psi))$ is the center of $\sW^\mathfrak{F},$ its  Lie algebra $W_{-2}^\mathfrak{F}(\mathfrak{g})=\mathfrak{g}^{-2}$ is commutative;
  \item  $Z^\mathfrak{F}:=\mbox{the centralizer of the morphism $w_{\mathfrak{F}}$ in  $N^{\mathfrak{F}}$ },$ it has Lie algebra  $\mathfrak{g}^{0};$
\item $V^\mathfrak{F}:=\sW^\mathfrak{F}/U^{\mathfrak{F}}$ is an Abelian group
whose Lie algebra  identifies with
the space $\mathfrak{g}^{-1}.$
\end{itemize}
All above algebraic subgroups will be defined over $\Q$ if $\mathfrak{F}$ is
rational.
Similarly, we can define  Lie subgroups $W_{-2}^\mathfrak{F}(G),W_{-1}^\mathfrak{F}(G),W_{0}^\mathfrak{F}(G)$ of $G=\Sp(g,\R).$
\end{corollary}

\subsection{Polarized Mixed Hodge structures attached to cusps}
For an isotropic real subspace $L_\R$ of $(V_\R,\psi)$(resp.
rational subspace $L_\Q$ of $(V_\Q,\psi)$), let
$$L_\R^{\perp}:=\{v\in V_\R\,\,|\,\, \psi(v,L_\R)=0\}\mbox{(resp. $ L_\Q^{\perp}:=\{v\in V_\Q\,\,|\,\, \psi(v,L_\Q) \}$ )},$$
let $L_\R^{\vee}$ be the dual space of $L_\R$ in
$V_\R$ with respect to
$(V_\R,\psi)$(resp. $L_\Q^{\vee}$ the dual space of $L_\Q$ in
$V_\Q$ with respect to
$(V_\Q,\psi)$).
For any $N\in\mathfrak{g},$ there is a symmetric bilinear form $\psi_N : V_\R\times V_\R\to \R$ given by
\begin{equation}\label{symmetric-bilinear-form}
  \psi_N(v, u):= \psi(v,N(u)).
\end{equation}

\begin{lemma}\label{weight filtration-positive cone}
Let $\mathfrak{F}=\mathfrak{F}(W)$ be a cusp  of $\mathfrak{H}_g$ and let $N$ be an arbitrary nonzero element in
$W_{-2}^{\mathfrak{F}}(\mathfrak{g})=\mathrm{Lie}(U^{\mathfrak{F}}(\R)).$
Let $W^{\mathfrak{F}}_\bullet(\R)=\big(0\subset
W^{\mathfrak{F}}_{\R,0} \subset W^{\mathfrak{F}}_{\R,1}\subset W^{\mathfrak{F}}_{\R,2}\big):=
\big(0\subset W \subset W^{\perp}\subset V_\R\big)$ be
the weight filtration associated to the cusp $\mathfrak{F}.$
\begin{myenumi}
\item The inclusions $\mathrm{Im}(N)\subset W $ and $W^\perp \subset \Ker(N)$ are held. The element $N$ induces a weight filtration $W_\bullet(N)=(0\subset W_{-1}(N) \subset W_0(N)\subset W_1(N))$
 by setting  $ W_1(N):=V_\R, W_{-1}(N):=\mathrm{Im}(N: V_\R\to V_\R)$ and $W_{0}(N):=\Ker(N: V_\R\to V_\R).$

\item For any two $N_1, N_2\in
W_{-2}^{\mathfrak{F}}(\mathfrak{g}),$ $N_1N_2=N_2N_1=0.$

\item The $\psi_N$ can be regarded as a symmetric bilinear form on $V_\R/ W^{\perp}.$
 If  $\psi_N$ is non-degenerate on $V_\R/ W^{\perp}$ then
   $W^{\mathfrak{F}}_\bullet(\R)=(W(N)[-1])_\bullet,$
  where $(W(N)[-1])_\bullet$ is the weight filtration given by $(W(N)[-1])_j:=W_{j-1}(N)\,\, \forall j.$
\end{myenumi}
\end{lemma}
\begin{proof}
\begin{myenumi}

\item  The lemma \ref{weight-morphism-3} shows that $\mathrm{Im}(N)\subset W $ as $N\in W_{-2}^{\mathfrak{F}}(\mathfrak{g}).$
It is easy to obtain that
$\mathrm{Im}(N)\subset W  \Longleftrightarrow W^\perp \subset \Ker(N).$

\item It is obvious by that $\mathrm{Im}(N_2)\subset W\subset W^\perp\subset \Ker(N_1).$

\item If there is a vector $0\neq v\in W_0(N)\setminus W^\perp$
then  $N(v)=0,$ and so $\psi(v, N(v))=0.$ Thus, we must have
$W^\perp=W_0(N)=\Ker(N: V_\R\to V_\R).$

 \noindent\textbf{Claim}: $\mathrm{Im}(N)=W
\Longleftrightarrow W^\perp = \Ker(N).$

\noindent{\it Proof of the claim.}
\begin{itemize}
  \item "$\Rightarrow$" : Suppose $\mathrm{Im}(N)= W.$ Then,
$\psi(y, W)=\psi(y,\mathrm{Im}(N))=-\psi(N(y),V_\R)=0$
for any $y\in \Ker(N).$ Thus $\Ker(N)\subset W^\perp.$

  \item "$\Leftarrow$" : Suppose $W^\perp = \Ker(N).$ Since
$N: \frac{V_\R}{\Ker(N) }\>\cong>>\mathrm{Im}(N),$
we get $$\dim_\R\mathrm{Im}(N)=\dim_\R V_\R-\dim_\R\Ker(N)=\dim_\R V_\R-\dim_\R W^\perp=\dim_\R W.$$
\end{itemize}
\end{myenumi}
\end{proof}

Given a rational boundary component $\mathfrak{F}(W),$ we have a convex cone in $\mathfrak{g}$ which does not contain any linear subspace :
\begin{equation}\label{positive-cone}
 C(\mathfrak{F}(W)):=\{ N\in W_{-2}^{\mathfrak{F}(W)}(\mathfrak{g})\,\,|\,\, \psi_{N}>0 \mbox{ on } \frac{V_\R}{W^\perp}\}
\end{equation}
where $\psi_N$ is a symmetric bilinear form on $V_\R$ defined in \ref{symmetric-bilinear-form}(cf.\cite{AMRT},\cite{CCK79}), we always call $C(\mathfrak{F}(W))$
positive cone in $W_{-2}^{\mathfrak{F}(W)}(\mathfrak{g}).$
With respect to the cusp $\mathfrak{F}_k:=\mathfrak{F}(V^{(k)}),$ the positive
cone is
$$C(\mathfrak{F}_k)= \{ \left(
   \begin{array}{cccc}
     0_{k} & 0     & 0     & 0 \\
     0       & 0_{g-k} & 0     & u \\
     0       & 0     &0_{k}& 0 \\
     0       & 0     & 0     & 0_{g-k}\\
   \end{array}
 \right) \,\, |\,\, u\in \mathrm{Sym}^{+}_{g-k}(\R) \},$$
where $\mathrm{Sym}^{+}_{g-k}(\R)$ is defined to be the set of positive-definite matrices in
$\mathrm{Sym}_{g-k}(\R).$

\begin{corollary}\label{weight filtration-positive cone-1}
Let $\mathfrak{F}=\mathfrak{F}(W)$ be a cusp of $\mathfrak{H}_g$ and
$N$  an arbitrary nonzero elements in $C(\mathfrak{F}).$
We have $N^2=0$ and
$$W(N)[-1]_\bullet= W^{\mathfrak{F}}_\bullet(\R)=\big(0\subset
W^{\mathfrak{F}}_{\R,0}(=W) \subset W^{\mathfrak{F}}_{\R,1}(=W^{\perp})\subset W^{\mathfrak{F}}_{2,\R}(=V_\R)\big).$$
\end{corollary}
\begin{proof}
It is obvious by  the lemma \ref{weight filtration-positive cone}.\\
\end{proof}

For any isotropic rationally-defined subspace $W$ of
$V_\R,$ the  dual space $W^\vee$ of $W$  is also an isotropic rationally-defined subspace of
$V_\R$ satisfying $\dim_\R W^\vee=\dim_\R W.$ For any cusp $\mathfrak{F}=\mathfrak{F}(W)$ of the Siegel
space $\mathfrak{H}_g,$
\begin{equation}\label{dual-cusp }
  \mathfrak{F}^\vee:=\mathfrak{F}(W^\vee)
\end{equation}
is defined to be the \textbf{dual cusp} of $\mathfrak{F}.$

We observe that the space $\check{\mathfrak{S}}_g$ can be identified with a set of
Hodge filtrations
$$\{F^{\bullet}=(0\subset F^1\subset V_\C)\,\, | \,\,F^1\in \mathrm{Grass}(g, V_\C),\,\, \psi(F^1,F^1)=0\}$$
and  any cusp $\mathfrak{F}$ can also be identified with a set of certain Hodge filtrations.

\begin{lemma}\label{nilpotent orbit}
Let $\mathfrak{F}=\mathfrak{F}(W)$ be a cusp of the Siegel space
$\mathfrak{H}_g$ and $W^{\mathfrak{F}}_\bullet=(0\subset W_\Q\subset W^{\perp}_\Q\subset
V_\Q)$ the corresponding weight filtration on $V_\Q$ where  $W_\Q$ is an isotropic subspace of $V_\Q$
given by $W=W_\Q\otimes\R.$ Let $ N\in C(\mathfrak{F})$ be an element in the positive cone $C(\mathfrak{F})\subset W^{\mathfrak{F}}_{-2}(\mathfrak{g})$
and $F^\bullet=(0\subset F^1\subset V_\C)$ a filtration in
$\mathfrak{F}.$
\begin{myenumi}
\item There are following direct sum decompositions
\begin{eqnarray*}
  W^\perp_\Q &=&W_\Q\bigoplus((W^{\vee}_\Q)^\perp\cap W^\perp_\Q),  \\
 (W^\vee_\Q)^\perp &=&W^\vee_\Q\bigoplus((W^{\vee}_\Q)^\perp\cap W^\perp_\Q),  \\
  F^1  &=&  W_\C \bigoplus((W^{\vee}_\C)^\perp\cap W^\perp_\C\cap F^1).
\end{eqnarray*}
Each above decomposition is orthogonal under the form $\psi.$

Define $$\check{F}^1:=W^\vee_\C\bigoplus
((W^{\vee}_\C)^\perp\cap W^\perp_\C\cap F^1).$$ The dual
filtration $\check{F}^\bullet:=(0\subset \check{F}^1 \subset
V_\C)$ is in $\mathfrak{F}^\vee,$ and there is a bijection $$
\mathfrak{F}\>>> \mathfrak{F}^\vee\, \,\,  F^\bullet \longmapsto
\check{F}^\bullet.$$

\item  That $N(V_\R)=N(W^\vee_\R)=W_\R,$ $N(\check{F}^1)=W_\C$ and $N(F^1)=0.$ Also, there is an isomorphism $N|_{W^\vee_\R}:
W^\vee_\R\>\cong>>W_\R.$

\item There holds $\exp(\sqrt{-1}N)\check{F}^1\in \mathfrak{S}_g.$
\end{myenumi}
\end{lemma}
\begin{proof}
There  exists a $M\in \Sp(g,V_\Q)$ such that
$\mathfrak{F}_k:=\mathfrak{F}(V^{(k)})=M(\mathfrak{F})$ for some  $k,$ we then  have
$w_{\mathfrak{F}_{k}}=Mw_{\mathfrak{F}}M^{-1}$ and
$$W^{\mathfrak{F}_k}_\bullet=MW^{\mathfrak{F}}_\bullet M^{-1},\,\,
W_\bullet^{\mathfrak{F}_k}(\mathfrak{g})=
MW_\bullet^{\mathfrak{F}}(\mathfrak{g})M^{-1} \mbox{ and }
C(\mathfrak{F}_k)=MC(\mathfrak{F})M^{-1}. $$ In particular, we have
$(V^{(k)_\Q})^{\vee}=M(W^{\vee}_\Q) \mbox{ and }
((V^{(k)}_\Q)^{\vee})^\perp=M((W^{\vee}_\Q)^\perp).$

Thus, it is
sufficient to prove the statements in the case of
$\mathfrak{F}=\mathfrak{F}_k.$ Now, we obtain :
$$(V^{(k)}_\Q)^\vee=\mathrm{span}_\Q\{e_{g+k+1},\cdots, e_{2g}\},
(V^{(k)}_\Q)^\perp=\mathrm{spac}_\Q\{e_1,\cdots, e_g,
e_{g+1},\cdots, e_{g+k}\}$$ and $((V^{(k)}_\Q)^{\vee})^\perp\cap
(V^{(k)}_\Q)^\perp=\mathrm{span}_\Q\{e_1,\cdots, e_k,
e_{g+1},\cdots,e_{g+k}\}.$
\begin{myenumi}
\item
Obviously, there is a direct sum decomposition
$$(V^{(k)}_\Q)^\perp=V^{(k)}_\Q\oplus(((V^{(k)}_\Q)^\vee)^\perp\cap(V^{(k)}_\Q)^\perp).$$
This decomposition is orthogonal with respect to the form $\psi.$ By duality, we also get the second equality in the statement (1).
Since $\psi(F^1, F^1)=0,$ we have $F^1\subset W_\C^\perp.$ By the
first equality, any $f\in F^1$ can be written as $f=v_1+v_2$ where
$v_1\in W_\C$ and $v_2\in(W^{\vee}_\C)^\perp\cap W_\C^\perp.$ Due
to $W_\C\subset F^1,$ we have $v_2=f-v_1\in F^1$ and so $v_2\in
(W^{\vee}_\C)^\perp\cap W^\perp_\C\cap F^1.$ Thus, the third
equality in the statement (1) is
true.

\item  Since $$N=\left(
   \begin{array}{cccc}
     0_{k} & 0     & 0     & 0 \\
     0       & 0_{g-k} & 0     & B \\
     0       & 0     &0_{k}& 0 \\
     0       & 0     & 0     & 0_{g-k}\\
   \end{array}
 \right)$$ where $B$ is a positive-definite matrix in  $\mathrm{Sym}_k(\R),$
 We must have
$$N(W^\vee_\R)=W_\R\,\,\mbox{ and }\,N(W^\perp_\R)=0.$$
The other equalities in the statement (2) are obvious.

\item Let $f$ be a  an arbitrary nontrivial vector in $\check{F}^1.$ By the statement (1), we can write
$$f=w_1+w_2 \mbox{ where } w_1\in W_\C^\vee \mbox{ and }w_2\in(W^{\vee}_\C)^\perp\cap W^\perp_\C\cap F^1.$$
Then we have : 
\begin{eqnarray*}
  \sqrt{-1}\psi(\exp(\sqrt{-1}N)f, \overline{\exp(\sqrt{-1}N)f} ) &=& \sqrt{-1}\psi(\exp(\sqrt{-1}N)f, \exp(-\sqrt{-1}tN)\overline{f} ) \\
   &=& \sqrt{-1}\psi(f, \exp(-2\sqrt{-1}N)\overline{f}) \\
   &=& \sqrt{-1}\psi(f, \overline{f})+2\psi(f, N\overline{f}) \\
   &=& \sqrt{-1}\psi(f, \overline{f})+2\psi(w_1, N\overline{w_1})+ 2\psi(w_2, N\overline{w_1})\\
   &=& \sqrt{-1}\psi(f, \overline{f})+2\psi(w_1, N\overline{w_1})+ 2\psi(\overline{w_1}, Nw_2)\\
   &=& \sqrt{-1}\psi(w_2, \overline{w_2})+2\psi(w_1, B\overline{w_1}).
\end{eqnarray*}
We get
$\sqrt{-1}\psi(w_2, \overline{w_2})\geq 0,$ and
obtain that $\sqrt{-1}\psi(w_2, \overline{w_2})=0$ if and only if $w_2=0.$
Also, we get $\psi(w_1,
B\overline{w_1})\geq 0$  by $N\in C(\mathfrak{F}),$ and obtain that $\psi(w_1, B\overline{w_1})=0$ if and
only if $w_1=0.$ Therefore $\sqrt{-1}\psi(\exp(\sqrt{-1}N)f,
\overline{\exp(\sqrt{-1}N)f} ))>0$ and so $\sqrt{-1}\psi(\cdot,
\overline{\cdot})$ is Hermitian positive on
$\exp(\sqrt{-1}N)\check{F}^1.$ Now we can finish the proof of the statement (3) by the fact that a point $F\in \overline{\mathfrak{S}}_{g}$
is in $\mathfrak{S}_g$ if and only if $\sqrt{-1}\psi(\cdot, \overline{\cdot})$ is Hermitian positive on $F.$

\end{myenumi}

\end{proof}

Some calculations in the lemma \ref{nilpotent orbit} are taken from  \cite{CCK79} and \cite{Fal} with minor modification.
In \cite{CCK79}, Carlson, Cattani and Kaplan  originally use Hodge theory to construct toroidal compactifications of Siegel varieties. In the following theorem \ref{PMHS-rational boundary component}, we further observe that the Hodge-theoretic interpretation of boundary components naturally coincides with the classic description given by Satake-Baily-Borel in \cite{Sat} and \cite{BB66}:
Any cusp of Siegel space $\mathfrak{H}_g$ can be  identified as a set of certain weight one polarized mixed Hodge structures.

A \textbf{mixed Hodge structure}(MHS) on $V_\Q$(resp.$V_\R$) consists of two filtrations,
an increasing filtration on $V_\Q$(resp. $V_\R$)-the weight
filtration $W_\bullet,$ and a decreasing filtration $F^\bullet$ on
$V_\C$-the Hodge filtration, such that the filtration  $F^\bullet$
induces a Hodge structure on each $\Gr_r^{W}V_\Q:=W_{r}/W_{r-1}$
(resp. $\Gr_r^{W}V_\R:=W_{r}/W_{r-1}$) of pure weight $r,$ where
$F^p(\Gr_r^{W}V_\C):= \frac{F^p\cap(W_r\otimes \C)}{F^p\cap
(W_{r-1}\otimes \C)}\,\,\forall p$(cf.\cite{Del71},\cite{Sch73}).

\begin{definition}[Cf.\cite{CKS86}\&\cite{CK87}]
A \textbf{polarized mixed Hodge structure}(PMHS) of  weight $l$ on $V_\R$ consists of a MHS $(W_\bullet,F^\bullet)$
on $V_\R,$ a nilpotent element $N\in F^{-1}\mathfrak{g}_\C\cap\mathfrak{g}_\R$ and a non-degenerate bilinear form
$Q$ such that
\begin{itemize}
\item $N^{l+1}=0,$ and
$W_\bullet=(W(N)[-l])_\bullet$ where $W(N)[-l]_j:=W_{j-l}(N)\,\,\forall j;$

\item $Q(F^p, F^{l-p+1})=0\,\,\forall p;$

\item $N(F^p)\subset F^{p-1}\,\,\forall p;$

\item the weight $l+r$ Hodge structure induced by $F^\bullet$ on $P_{l+r}:=\ker(N^{r+1}:\Gr_{l+r}^{W_\bullet}\to \Gr_{l-r-2}^{W_\bullet})$ is polarized by $Q(\cdot, N^r(\cdot)),$ i.e., that $Q(\cdot, N^r(\cdot))$ is $(-1)^{l+r}$-symmetric on $P_{l+r}$ and
    \begin{eqnarray*}
      Q(P_{l+r,\C}^{p_1,q_1}, N^r(P_{l+r,\C}^{p_2, q_2})) &=& 0\,\,\,\,\mbox{ unless } p_1=q_2 \mbox{ and } p_2=q_1 \\
      (\sqrt{-1})^{p-q}Q(v, N^r(\overline{v}))           &>& 0\,\,\,\,\mbox{ for any nonzero } v\in P_{l+r,\C}^{p,q}.
    \end{eqnarray*}
\end{itemize}
\end{definition}

\begin{theorem}\label{PMHS-rational boundary component}
 Let $\mathfrak{F}$ be a rational component of Siegel space
$\mathfrak{H}_g$ and
$$W^\mathfrak{F}_\bullet=\big(0\subset W^\mathfrak{F}_0\subset W^\mathfrak{F}_1\subset W^\mathfrak{F}_2(=V_\Q)\big)$$
the corresponding weight filtration on the rational space $V_\Q.$ Let $\psi$ be the standard symplectic form  on $V_\R=V_\Q\otimes_\Q\R.$
We have :
\begin{myenumi}
\item For any Hodge filtration $F^\bullet \in \mathfrak{F}^\vee$ and any element $N\in C(\mathfrak{F}),$ the quadruple
 $(F^\bullet_{\tau}, W^{\mathfrak{F}}_{\bullet},N, \psi)$  determines a polarized mixed Hodge structure of weight one on $V_\R.$

\item Any pair $(F^\bullet_{\tau}, W^{\mathfrak{F}}_{\bullet})$ with
$F^\bullet_\tau \in \mathfrak{S}_g$ is a mixed Hodge structure of
weight one on $V_\Q.$

\item Moreover,
$$\mathfrak{F}^\vee=\{F^\bullet \in \check{\mathfrak{S}}_g \,\,|\,\,(F^\bullet, W^{\mathfrak{F}}_{\bullet},N,\psi) \mbox{ is a PMHS of weight one for all } N\in C(\mathfrak{F})\}.$$
\end{myenumi}
\end{theorem}
\begin{proof}
Let $G=\Sp(g,\R).$
\begin{myenumi}
\item  By the corollary \ref{weight
filtration-positive cone-1}, we have that
$$
 (W(N)[-1])_\bullet=(W^{\mathfrak{F}}\otimes\R)_\bullet:=\big(0\subset W^\mathfrak{F}_{0,\R}\subset W^\mathfrak{F}_{1,\R}\subset W^\mathfrak{F}_{2,\R}(=V_\R)\big)\,\, \forall N\in C(\mathfrak{F}).
 $$
Let $F^\bullet=(0\subset F^1\subset F^0=V_\C)$ be a
filtration $\mathfrak{F}^\vee.$ By definition, we have :
$$ F^1\cap \overline{F^1}=\mathrm{Im}(N)^\vee_\C=(W^{\mathfrak{F}}_{0})^\vee\otimes\C \,\, \forall N\in C(\mathfrak{F}).
$$
Since $\psi(F^1, F^1)=\psi(\overline{F^1},\overline{F^1})=0,$ we
have that
$$F^1\subset(\mathrm{Im}(N)^\vee_\C)^\perp \mbox{ and }
\overline{F^1}\subset(\mathrm{Im}(N)^\vee_\C)^\perp.$$

\noindent\textbf{Claim i}. For any integer $r$ in $[0,2],$
$$F^p(\Gr_r^{W^\mathfrak{F}}V_\C)\oplus \overline{F^{r-p+1}(\Gr_r^{W^\mathfrak{F}}V_\C)}
\>\cong>>\Gr_r^{W^{\mathfrak{F}}}V_\C \,\, \forall p.$$

The claim i is not difficult to check.  Then, we obtain the decreasing filtration
$F^\bullet(\Gr_r^{W^\mathfrak{F}}V_\C)$ on $\Gr_r^{W^{\mathfrak{F}}}V_\Q$
arises as a Hodge structure on $\Gr_r^{W^{\mathfrak{F}}}V_\Q$ of pure
weight $r$ for $r=0,1,2$(cf.\cite{Gri}\&\cite{Sch73}), i.e., the
pair $(F^\bullet, W^{\mathfrak{F}}_\bullet)$ is a mixed Hodge
structure on $V_\Q.$ The polarization is automatically true by the definition of $C(\mathfrak{F})$(cf.\ref{positive-cone}).

\item Let $\check{F}^\bullet=(0\subset
\check{F}^1\subset\check{F}^0(= V_\C))$ be a fixed decreasing
filtration in $\mathfrak{F}^\vee.$ The lemma \ref{nilpotent orbit}
shows that
$$F^1_\delta:=\exp(\sqrt{-1}N)\check{F}^1 \in \mathfrak{S}_g.$$

Write $F^\bullet_\delta:=( 0\subset F^1_\delta \subset
F^0_{\delta}(=V_\C)).$ We begin to show that
the  pair $(F^\bullet_\delta, W^{\mathfrak{F}}_\bullet)$ is a mixed Hodge structure on $V_\Q$
with same Hodge numbers as the MHS $(\check{F}^\bullet, W^{\mathfrak{F}}_\bullet).$

Since $F^p_{\delta}\oplus
\overline{F^{2-p}_\delta}=V_\C \,\forall p,$ we get
$F^p_{\delta} \bigcap \overline{F^{2+k-p}_\delta}\subset F^p_{\delta} \bigcap \overline{F^{2-p}_\delta}=\{0\}\,\forall k\geq 0 \,\,\forall p$
and so
\begin{equation}\label{*-condition}
F^p_{\delta}\Gr_{a}^{W^{\mathfrak{F}}}V_\C \bigcap
\overline{F^{1+a-p}_\delta \Gr_{a}^{W^{\mathfrak{F}}}V_\C}=\{0\}\,\,
\forall a\geq 1\,\, \forall p.
\end{equation}

Let $M:=\exp(\sqrt{-1}N).$ Since $M\in W_{-2}^{\mathfrak{F}}(G)$ is unipotent, $M$ respects the weight filtration $(W^{\mathfrak{F}}\otimes\C)_\bullet$ and acts as identity on $\Gr_{a}^{W^{\mathfrak{F}}}V_\C$ for all $a.$
Moreover, $M$ induces a complex linear isomorphisms
$M: \check{F}^p\Gr^{W^{\mathfrak{F}}}_a V_\C \>\cong >>
 F_\delta^p\Gr^{W^{\mathfrak{F}}}_aV_\C \,\forall a,p.$
Since
$$F^1_{\delta}\Gr^{W^{\mathfrak{F}}}_0V_\C\cong\check{F}^1\Gr^{W^{\mathfrak{F}}}_0V_\C=\{0\}\mbox{ and }
F^0_{\delta}\Gr^{W^{\mathfrak{F}}}_0V_\C\cong
\check{F}^0\Gr^{W^{\mathfrak{F}}}_0V_\C\cong \Gr^{W^{\mathfrak{F}}}_0V_\C,
$$  by the above \ref{*-condition} we  have that
$$F^p_{\delta}\Gr_{a}^{W^{\mathfrak{F}}}V_\C \bigoplus \overline{F^{1+a-p}_\delta \Gr_{a}^{W^{\mathfrak{F}}}V_\C}
\cong \Gr_{a}^{W^{\mathfrak{F}}}V_\C \,\,\forall a,p. $$

The following claim ii guarantees the statement 2 is true.

\noindent\textbf{Claim ii}.  Let $\tau$ be an arbitrary point in $\mathfrak{H}_g.$  If $(F^\bullet_\tau,W^{\mathfrak{F}}_\bullet) $ is a mixed Hodge structure on $V_\Q$ then
for any $\tau'\in \mathfrak{H}_g,$ $(F^\bullet_{\tau'},W^{\mathfrak{F}}_\bullet)$ is again a MHS on $V_\Q$ and
there is an isomorphism  of MHSs $(F^\bullet_\tau,W^{\mathfrak{F}}_\bullet)\>\cong>>(F^\bullet_{\tau^{'}},W^{\mathfrak{F}}_\bullet).$

\noindent{\it Proof of the claim ii.}
It is known that $G$ acts transitively on $\mathfrak{S}_g,$ and
$$\mathfrak{S}_g\cong G/K_\tau,\,\,\,\mbox{  where  }K_\tau=\{M\in G\,\,|\,\, M F^\bullet_\tau=F^\bullet_\tau\}.$$
There always holds
$G= W_{0}^{\mathfrak{F}}(G)K_\tau$(cf.\cite{Sch73} (5.24),p. 242). Thus the group $W_{0}^{\mathfrak{F}}(G)$ acts transitively on $\mathfrak{S}_g$ and
there is a $M\in W_{0}^{\mathfrak{F}}(G)$ such that
$$\tau^{'}=M(\tau) \mbox{ and }
F^\bullet_{\tau^{'}}=M F^\bullet_\tau, \,\,  \overline{F^\bullet_{\tau^{'}}}=\overline{M F^\bullet_\tau} =M \overline{F^\bullet_\tau}.$$ Since $M$ respects
$(W^{\mathfrak{F}}\otimes\C)_\bullet,$ we have isomorphisms
\begin{eqnarray*}
  M &:& F^\bullet_\tau\Gr_{a}^{W^{\mathfrak{F}}}V_\C \>\cong >> F^\bullet_{\tau^{'}}\Gr_{a}^{W^{\mathfrak{F}}}V_\C,\,\,\forall a\\
  M &:& \overline{F^\bullet_\tau\Gr_{a}^{W^{\mathfrak{F}}}V_\C} \>\cong >> \overline{F^\bullet_{\tau^{'}}\Gr_{a}^{W^{\mathfrak{F}}}V_\C }\,\,\forall a.
\end{eqnarray*}

Thus, $(F^\bullet_{\tau^{'}},W^{\mathfrak{F}}_\bullet)$ is  a MHS on $V_\Q$ and $M$ induces an isomorphism  of MHSs $$M:(F^\bullet_\tau,W^{\mathfrak{F}}_\bullet)\>\cong>>(F^\bullet_{\tau^{'}},W^{\mathfrak{F}}_\bullet).$$


\item Let $F^\bullet=(0\subset F^1\subset F^2(=V_\C))$ be a Hodge filtration in $\check{\mathfrak{S}}_g$ and
$N\in C(\mathfrak{F}).$\\
\noindent\textbf{Claim iii}. Suppose that $(F^\bullet, W^{\mathfrak{F}}_{\bullet},N,\psi)$ is a PMHS,
then $F^\bullet$ is in $\mathfrak{F}^{\vee}.$

After the claim iii, the statement 3 will be true by using the corollary \ref{weight filtration-positive
cone-1} and the statement 1 in this theorem : We have the fact that if
$(F^\bullet, W^{\mathfrak{F}}_{\bullet},N_1,\psi)$ is a PMHS for one
certain $N_1\in C(\mathfrak{F})$ then $(F^\bullet,
W^{\mathfrak{F}}_{\bullet},N,\psi)$ is again a  PMHS for any $ N\in
C(\mathfrak{F}).$

We now begin to prove the claim iii by the following steps (A) to (C).
\begin{description}
\item[(A)] From the MHS $(F^\bullet, W^{\mathfrak{F}}_{\bullet}),$ we
obtain the following facts :
\begin{itemize}
 \item $(\Gr_0^{W^{\mathfrak{F}}}V_\Q,
F^\bullet\Gr_0^{W^{\mathfrak{F}}}V_\C)$ is a Hodge structure of pure
weight zero, and so
$$F^1\cap \mathrm{Im}(N)_\C=\overline{F^1}\cap \mathrm{Im}(N)_\C=\{0\}.$$

\item $(\Gr_1^{W^{\mathfrak{F}}}V_\Q,
F^\bullet\Gr_1^{W^{\mathfrak{F}}}V_\C)$ is a Hodge structure of pure
weight one. By definition,
$$F^1\Gr_1^{W^{\mathfrak{F}}}V_\C\cong F^1\cap \Ker(N)_\C.$$

\item $(\Gr_2^{W^{\mathfrak{F}}}V_\Q,
F^\bullet\Gr_2^{W^{\mathfrak{F}}}V_\C)$ has a rational Hodge structure
of pure type $(1,1).$ Thus
$$F^1\Gr_2^{W^{\mathfrak{F}}}V_\C=\overline{F^1\Gr_2^{W^{\mathfrak{F}}}V_\C}=\Gr_2^{W^{\mathfrak{F}}}V_\C,$$
and so $ N(F^1)=N(\overline{F^1})=\mathrm{Im}(N)_\C,\, \ker(N)_\C+F^1=\Ker(N)+\overline{F^1}=V_\C.$
\end{itemize}

\item[(B)]  We show that $F^1\cap
\overline{F^1}=\mathrm{Im}(N)_\C^\vee=(W_0^{\mathfrak{F}}\otimes\C)^\vee.$
\begin{itemize}
  \item $\mathrm{Im}(N)_\C^\vee\subseteq F^1\cap\overline{F^1}$ :
  Since $(\Gr_2^{W^{\mathfrak{F}}}V_\Q, F^\bullet\Gr_2^{W^{\mathfrak{F}}}V_\C)$ is polarized $\psi(\cdot, N(\cdot)),$
we have that $\mathrm{Im}(N)_\C^\vee=(N(\overline{F^1}))^\vee\subset
F^1$ and $\mathrm{Im}(N)_\C^\vee=(N(F^1))^\vee\subset
\overline{F^1}.$ Then
 $$\mathrm{Im}(N)_\C^\vee\subseteq F^1\cap\overline{F^1}\mbox{ and } F^1\subset(\mathrm{Im}(N)_\C^\vee)^\perp.$$

  \item $F^1\cap\overline{F^1}\subseteq \mathrm{Im}(N)_\C^\vee$ :
Let $v$ be an arbitrary vector in $ E:=F^1\cap \overline{F^1}.$ By
the second equality in(1) of the lemma \ref{nilpotent orbit}, $v$
can be written uniquely as
$$v=v_1+v_2, \mbox{ where } v_1\in \mathrm{Im}(N)^\vee_\C, v_2\in \Ker(N)_\C.$$
As $\mathrm{Im}(N)_\C^\vee\subset E,$ we have $v_2=v-v_1\in E$ and
so $v_2\in \Ker(N)_\C\cap E.$
Since the weight one Hodge structure $(\Gr_1^{W^{\mathfrak{F}}}V_\Q,
F^\bullet\Gr_1^{W^{\mathfrak{F}}}V_\C)$ is polarized by the form $\psi,$ we
must have that $\sqrt{-1}\psi( F^1\cap\Ker(N)_\C, \overline{F^1\cap\Ker(N)_\C})>0.$
On the other hand, the $\sqrt{-1}\psi(v_2,\overline{v_2})$ is zero by
that $\psi(E, \overline{E})=\psi(E, E)=0.$
Thus $v_2=0$ and  $v\in F^1\cap \overline{F^1}.$
\end{itemize}

\item[(C)] Due to the fact $\mathrm{Im}(N)_\C^\vee\cap\Ker(N)_\C=\{0\},$ we get
that $$V_\C=\mathrm{Im}(N)_\C^\vee\oplus\Ker(N)_\C\,\, \mbox{ and }
F^1=\mathrm{Im}(N)_\C^\vee \bigoplus  F^1\cap\Ker(N)_\C.$$
Since that $\psi(F^1, F^1)=\psi(\overline{F^1}, \overline{F^1})=0$ and
$$\sqrt{-1}\psi( F^1\cap\Ker(N)_\C,
\overline{F^1\cap\Ker(N)_\C})>0,$$ we must have that
$\sqrt{-1}\psi(F^1, \overline{F^1})\geq 0,$
which is equivalent to that $F^\bullet\in \overline{\mathfrak{S}}_{g}.$
\end{description}
\end{myenumi}
\end{proof}



\section{Toroidal compactifications and their infinity boundary divisors}

\vspace{0.5cm}

Let $\Gamma\subset \Sp(g,\Z)$ be an arithmetic subgroup.Denote by $\mathfrak{F}_{k}$ the cusp $\mathfrak{F}(V^{(k)}),$ which is a cusp of depth $g-k$ isomorphic to $\mathfrak{H}_k.$

According to  the
corollary \ref{algebraic subgroups by boundary-component}, we define
$u^{\mathfrak{F}}:=\mathrm{Lie}(U^\mathfrak{F}(\R))=W_{-2}^\mathfrak{F}(\mathfrak{g})$ and $v^{\mathfrak{F}}:=\mathrm{Lie}(V^\mathfrak{F}(\R))=W_{-2}^\mathfrak{F}(\mathfrak{g})$ for any cusp $\mathfrak{F}.$
We note that the Lie algebra $v^{\mathfrak{F}}$ is  identified with the
space $\mathfrak{g}^{-1}$ and that
$\dim_\R(C(\mathfrak{F}))=k(k+1)/2$ if $\mathfrak{F}$ is a $k$-th cusp.

\subsection{Equivalent toroidal embedding}
Let $N$ be a lattice, i.e., a  free  $\Z$-module of finite rank
and $M=N^\vee:=\Hom_\Z(N,\Z)$ its dual. We fix isomorphisms
$M\cong \Z^d,\,\, N\cong \Z^d.$
The lattice $N$ can be regarded as the group of $1$-parameter
subgroups of the algebraic torus $T_N:=\Spec \C[M].$ Actually, any $a=(a_1,\cdots,a_d)\in N\cong \Z^d$ corresponds to a unique
one-parameter subgroup  $\lambda_a:\G_m \to T_N$ given by
$\lambda_a(t)=(t_1^{a_1}, \cdots,t_d^{a_d})\in T_N(\C)\forall t\in \G_m(\C).$
We also note that $T_N\cong
\Hom_\Z(M,\G_m).$ On the other hand, the dual lattice $M$ can be
regarded as $X(T_N)$(the group of characters of $T_N$).
Any $m=(m_1,\cdots,m_d)\in M$ corresponds to a unique $\chi^m\in
X(T_N)$ given by
$$\chi^m(t_1,\cdots, t_d)= t_1^{m_1}\cdots t_d^{m_d}\in \G_m(\C)=\C^*\,\forall (t_1,\cdots, t_d)\in T_N(\C)\cong(\C^*)^d.$$ Obviously, the lattices $M,N$
are related by a non-degenerated canonical pairing
$$M\times N \longrightarrow \Z\,\,(m,a)\longmapsto <m,a>,$$
where $<m,a>$ is determined by
$\chi^m(\lambda_a(t))=t^{<m,a>}\,\,\,t\in \C^*.$

A \textbf{convex rational polyhedral cone} in $N_\R$ is a subset
$\sigma\subset N_\R$ such that
$$\sigma=\{\sum_{i=1}^t \lambda_iy_i | \lambda_i\in \R_{\geq 0},\, \,\, i=1,\cdots, t\}$$
for a finite number of vectors $y_i\in N_\Q,i=1,\cdots,t;$ its
dual of $\sigma$ is defined to be
$$\sigma^{\vee}= \{m\in M_\R \, |\, <m,u>\geq 0 \,\,\forall u \in
\sigma  \},$$ which in fact is  a convex rational polyhedral cone
in $M_\R;$ $\dim\sigma$ is defined to be the
dimension of the smallest subspace of $N_\R$ containing $\sigma;$
a \textbf{face} of $\sigma$ is a convex rational polyhedral cone
$\sigma'$ in $N_\R$ such that $\sigma'=\sigma\cap\{v\in
N_\R|\lambda(x)=0\}$ for some $\lambda\in \sigma^{\vee}\cap M_\Q,$
denoted by $\sigma'\prec\sigma.$ A $1$-dimensional convex rational
polyhedral cone is called an \textbf{edge}. Any convex rational
polyhedral cone $\sigma$ endows an \textbf{affine
 toric variety} $X_{\sigma}=\Spec \C[\sigma^{\vee}\cap M]$
where $\C[\sigma^{\vee}\cap M]:=\{\sum\limits_{m=(m_1,\cdots,m_d)\in
\sigma^{\vee}\cap M}a_m x_1^{m_1}\cdots x_d^{m_d}\,\,|\,\,
a_m\in \C \}.$
 A face $\tau$ of $\sigma$ induces an open
immersion $X_\tau\hookrightarrow X_\sigma$ of affine varieties, so
that $X_\tau$ can be identified with an open subvariety of
$X_\sigma.$

A convex rational polyhedral cone $\sigma$ of $N_\R$ is called
\textbf{strong} if and only if $\sigma\cap(-\sigma)=\{0\}.$ A
strong rational convex polyhedral cone $\sigma$ in $N_\R$ is said
\textbf{regular}(with respect to $N$) provided that $\sigma$ is generated by part of a
$\Z$-basis of $N.$ A \textbf{fan} of $N$ is defined to be a
nonempty collection $\Sigma=\{\sigma_\alpha\}$ of convex rational
polyhedral cones in $N_\R$ such that
\begin{itemize}
\item all cones in $\Sigma$ are strong;

\item if $\sigma'$ is a face of a cone $\sigma\in \Sigma,$ then
$\sigma'\in\Sigma;$

\item for any two $\sigma,\tau\in \Sigma,$ the intersection
$\sigma\cap \tau$ is a face of both $\sigma$ and $\tau.$
\end{itemize}
The set of all edges in a fan $\Sigma$ is denoted by $\Sigma(1).$
A fan $\Sigma$ is said \textbf{regular}(with respect to $N$) if all $\sigma\in \Sigma$
are regular. A fan $\Sigma=\{\sigma_\alpha\}$ of $N$ determines a
separate scheme
$X_{\Sigma}:=\bigcup\limits_{\sigma\in
\Sigma}X_\sigma$ by patching together the $X_{\sigma_\alpha}$'s
along the $X_{\sigma_\alpha\cap\sigma_\beta}$'s.

Let $\Sigma$ be an arbitrary fan of $N.$ In general, the
associated scheme $X_\Sigma$ is normal and locally of finite
type over $\C.$ The $X_\Sigma$ is smooth if and only if $\Sigma$
is regular, moreover $X_\Sigma$ is of finite type over $\C$ if and only if
$\Sigma$ is a finite collection of convex rational polyhedral
cones. Let $\sigma_\alpha$ be any cone in $\Sigma.$ Since
$\sigma^{\vee}_\alpha$ spans $M_\R,$ there is an open immersion of
the algebraic torus $T_N:=\Spec \C[M]$ in $X_\sigma=\Spec
k[\sigma^{\vee}_\alpha\cap M]$(We call $T_N\subset X_\sigma$ a \textbf{ toric embedding}). The action $T_N\times T_N\to T_N$
given by the translation in $T_N$ can be extended to an action $T_N\times
X_{\sigma_\alpha} \to X_{\sigma_\alpha}.$ The open immersion
$X_\tau\hookrightarrow X_{\sigma_\alpha}$ induced by a face
$\tau\prec \sigma_\alpha$ is certainly equivariant with respect to
the actions of $T_N.$ Therefore, there is a  natural open
immersion $T_N\>\hookrightarrow>>X_\Sigma$ with the unique action
$T_N$ on $X_\Sigma$ extending $T_N$'s action on each
$X_{\sigma_\beta}.$

\begin{proposition}[Cf.\cite{AMRT},\cite{KKMS}\&\cite{Ful}]\label{kkms-orbit-cone-correspondence}
Let $\Sigma=\{\sigma_\alpha\}$ be a fan of a lattice $N.$
\begin{myenumi}

\item There is a bijection 
between the set of cones in $\Sigma$ and the set of orbits in
$X_{\Sigma},$ and there holds $\dim \sigma_\alpha +
\dim_\C\sO^{\sigma_\alpha}=\dim_\C T_N.$ Moreover, $\sigma_\alpha\subseteq \sigma_\beta$ 
if and only if $\sO^{\sigma_\beta}\subset
\overline{\sO^{\sigma_\alpha}}^{\mathrm{cl}},$ where
$\overline{\sO^{\sigma_\alpha}}^{\mathrm{cl}}$ denotes the closure
in both the classical and Zariski topologies of $X_\Sigma.$ In
particular, each edge $\rho\in\Sigma$ gives a codimension one
closed subscheme $D_\rho:=\overline{\sO^\rho}^{\mathrm{cl}}$ in
$X_\Sigma,$ which actually is a $T_N$-invariant prime divisor of
$X_\Sigma.$

\item The collection $\{\sO^{\sigma_\alpha}\}$ is a stratification of $X_\Sigma$
in the classical analytic topology. Furthermore,
$
X_\tau=\coprod\limits_{\delta\prec \tau} \sO^\delta, \,\, \,\,
\overline{\sO^\tau}^{\mathrm{cl}}=\coprod\limits_{\delta\in\Sigma,\delta\succeq
\tau}\sO^\delta
$
for any cone $\tau\in\Sigma.$
\end{myenumi}
\end{proposition}

\begin{corollary}\label{divisor-cone}
Let $\Sigma$ be a  fan of a lattice $N$ and $X_\Sigma$ the associated
scheme  with torus embedding $T_N:=\Spec \C[N^\vee].$
\begin{myenumi}

\item For any cone $\sigma\in\Sigma,$
$\overline{\sO^\sigma}^{\mathrm{cl}}$ is a closed subscheme of $X_{\Sigma}$ with only normal
singularities; moreover, $\overline{\sO^\sigma}^{\mathrm{cl}}$ is  smooth if $\Sigma$ is a regular fan.

\item Assume that every low-dimensional cone $\tau\in \Sigma$($\dim
\tau< \rank N$) is a face of some top-dimensional cone $\sigma_{\max}\in \Sigma$(i.e., $\dim\sigma_{\max}=\rank N$).
If $\Sigma$ is a regular fan then the infinity boundary $D_\infty:=X_\Sigma\setminus T_N$ is simple normal crossing, i.e.,
all irreducible components of $D_\infty$ are smooth and they intersect each other transversely.
\end{myenumi}
\end{corollary}
\begin{proof}
It is a direct consequence of the proposition \ref{kkms-orbit-cone-correspondence}.
\end{proof}


\subsection{Admissible families of polyhedral decompositions}

In \cite{AMRT}, Mumford and his coworkers have constructed explicitly a class of
toroidal compactifications of $D/\Gamma$ for each bounded symmetric domain $D$
with an arithmetic  subgroup $\Gamma\subset \Aut(D).$
Actually, the compactification is determined by a certain combinatorial
$\Gamma$-admissible rational polyhedral cone decompositions.

We define a partial order on the set of cusps of $\mathfrak{H}_g$
: For any two cusps $\mathfrak{F}(W_1)$ and $\mathfrak{F}(W_2),$ we say
$\mathfrak{F}(W_1)\prec \mathfrak{F}(W_2)$ if and only if $W_2 \subset
W_1.$ According to this partial order,
$\mathfrak{F}(\{0\})=\mathfrak{H}_g$ is the unique maximal element, and a
cusp of depth  $g$ is called a \textbf{minimal cusp}(or \textbf{minimal rational boundary component}) of $\mathfrak{H}_g.$ We call  $\mathfrak{F}_0$
the \textbf{standard minimal cusp} of $\mathfrak{H}_g.$

\begin{definition}[Cf.\cite{Nam}]\label{proper-boundary-component-of-cone}
Suppose that $C$ is an open cone in a real vector space $E_\R,$
where $E_\R$ has an underlying integral structure $E_\Z,$
i.e.,$E_\R=E_\Z\otimes_\Z\R.$ A \textbf{(rational) boundary
component of $C$ is a cone } $C^{'}=(\overline{C}\cap
E^{'})^{\circ}$(denote by $C^{'}\prec C$) given by a linear
(rationally-defined) subspace $E^{'}$ of $E_\R$ with $E^{'}\cap
C=\emptyset,$ where $\overline{C}$ is the closure of the cone $C$ in $E_\R.$ The  \textbf{rational closure}
$\overline{C}^{\mathrm{rc}}$ of $C$ is the union of all rational
boundary components of $C.$
\end{definition}
We note that any proper rational boundary component of $C(\mathfrak{F}_0)$ is
of form $C(\mathfrak{F}^{'})$ where $\mathfrak{F}^{'}$ is a cusp
with $\mathfrak{F}_0\prec\mathfrak{F}^{'}$(cf.Theorem 3 in $\S4.4$ of Chap.III \cite{AMRT}).
\begin{lemma}\label{convex-cone}
Let $\mathfrak{F}(W_1)$ and $\mathfrak{F}(W_2)$ be two cusps of
$\mathfrak{H}_g.$ The cusp $\mathfrak{F}(W_1\cap W_2)$ has following
properties :
\begin{myenumi}
\item There is $u^{\mathfrak{F}(W_1\cap W_2)}\subset u^{\mathfrak{F}(W_1)}\cap u^{\mathfrak{F}(W_2)}.$
Moreover, if there is a maximal isotropic subspace $W$ of $V_\R$ containing $W_1\cup W_2$
then
$u^{\mathfrak{F}(W_1\cap W_2)}=  u^{\mathfrak{F}(W_1)}\cap u^{\mathfrak{F}(W_2)}.$

\item If $W_1\cap W_2$ is a proper subspace of $W_1$ then
$$u^{\mathfrak{F}(W_1\cap W_2)}\cap C(\mathfrak{F}(W_1))=\emptyset\mbox{ and }\,C(\mathfrak{F}(W_1\cap W_2))\cap C(\mathfrak{F}(W_1))=\emptyset.$$

\item The equalities $$\overline{C(\mathfrak{F}(W_1\cap
W_2))}=\overline{C(\mathfrak{F}(W_1))}\cap u^{\mathfrak{F}(W_1\cap W_2)}
=\overline{C(\mathfrak{F}(W_2))}\cap u^{\mathfrak{F}(W_1\cap W_2)}$$ are held.
If there is a maximal isotropic subspace $W$ of $V_\R$
containing $W_1\cup W_2$ then
$$\overline{C(\mathfrak{F}(W_1))}\cap
\overline{C(\mathfrak{F}(W_2))}=\overline{C(\mathfrak{F}(W_1\cap W_2))}.$$
\end{myenumi}
\end{lemma}
\begin{myrem}
The  lemma implies that $C(\mathfrak{F}(W_1\cap W_2))$ is a rational
boundary component of both $C(\mathfrak{F}(W_1))$ and
$C(\mathfrak{F}(W_2)).$ The equalities in (3) are also true
even if we replace $\overline{C(\mathfrak{F})}$ with
$\overline{C(\mathfrak{F})}^{\mathrm{rc}}.$
\end{myrem}
\begin{proof}[Proof of the lemma \ref{convex-cone}]
 Let $U=W_1\cap W_2.$ Then $U^\vee \subset W_1^\vee\cap W_2^\vee$ and $W_1^\perp+W_2^\perp \subset U^\perp.$
For any cusp $\mathfrak{F}(W),$ we have
\begin{equation}\label{convex-cone-1}
 u^{\mathfrak{F}(W)}=W_{-2}^{\mathfrak{F}(W)}=\{H\in\mathfrak{g}\,\,|\,\, H(W^\vee)\subset W \mbox{ and } H(W^\perp)=0\}.
\end{equation}

\begin{myenumi}
\item Since $W_1^\vee\oplus W_1^\perp= W_2^\vee\oplus W_2^\perp=U^\vee\oplus U^\perp=V_\R,$
the equality \ref{convex-cone-1} says there holds
$$u^{\mathfrak{F}(U)}\subset u^{\mathfrak{F}(W_1)}\cap u^{\mathfrak{F}(W_2)}.$$
Because there is a maximal isotropic subspace $W$ of $V_\R$ containing both $W_1$ and $W_2,$
the space $W_1+W_2$ is  isotropic in $W$ so that $\dim(W_1+W_2)^\perp =2g-\dim(W_1+W_2)$
and $U^\perp=W_1^\perp+W_2^\perp.$
Let $N\in u^{\mathfrak{F}(W_1)}\cap u^{\mathfrak{F}(W_2)}.$  Then we have
$$N(U^\vee)\subset N(W_1^\vee\cap W_2^\vee) \subset N(W_1^\vee)\cap N(W_2^\vee)\subset W_1\cap W_2=U$$
and $N(U^\perp)=N(W_1^\perp+W_2^\perp)=0.$
Thus $u^{\mathfrak{F}(W_1)}\cap u^{\mathfrak{F}(W_2)}\subset u^{\mathfrak{F}(U)}.$

\item
Recall
$C(\mathfrak{F}(W)):=\{ N\in u^{\mathfrak{F}(W)}\,\,|\,\, \psi(\cdot,N(\cdot))>0 \mbox{ on } \frac{V_\R}{W^\perp}\}.$
Let $N\in u^{\mathfrak{F}(U)}$ be an arbitrary element.
Consider the filtration
$0\subset U\subsetneq W_1\subset W_1^\perp\subset U^\perp \subsetneq V_\R,$
we obtain that $\psi(\cdot,N(\cdot))$ is semi-positive but not positive on $V_\R/W_1^\perp$ since
$N(U^\perp)=0.$ Thus, $u^{\mathfrak{F}(U)}\cap C(\mathfrak{F}(W_1))=\emptyset.$

\item It is sufficient to prove that
$\overline{C(\mathfrak{F}(U))}=\overline{C(\mathfrak{F}(W_1))}\cap u^{\mathfrak{F}(U)}
=\overline{C(\mathfrak{F}(W_2))}\cap u^{\mathfrak{F}(U)}$
for any rational defined subspace $U$ of $W_1$ with $U\subsetneq W_1,U\subsetneq W_2.$
Let $U$ be an arbitrary rational defined subspace $U$ of $W_1$ such that $U\subsetneq W_1,U\subsetneq W_2.$ We have
that $\overline{C(\mathfrak{F}(W))}=\{ N\in u^{\mathfrak{F}(W)}\,\,|\,\, \psi(\cdot,N(\cdot)) \mbox{ is semi-positive on } \frac{V_\R}{W^\perp}\}.$
Let $N\in \overline{C(\mathfrak{F}(U))}$ be an arbitrary element. Using the above  argument in (2), we get that the bilinear form
$\psi(\cdot,N(\cdot))$ is semi-positive on $V_\R/W_1^\perp$ and
$N\in \overline{C(\mathfrak{F}(W_1))}\cap u^{\mathfrak{F}(U)}.$
On the other hand, $\overline{C(\mathfrak{F}(W_1))}\cap
u^{\mathfrak{F}(U)} \subset \overline{C(\mathfrak{F}(U))}$ is clear.
\end{myenumi}
\end{proof}

Let $\mathfrak{F}$ be an arbitrary cusp of
$\mathfrak{H}_g.$
Since the Lie group $U^{\mathfrak{F}}(\C)$ is connected  and
$N^2=0$ for any $N\in u^{\mathfrak{F}}_\C$ by the lemma \ref{weight
filtration-positive cone}, the exponential map
$\exp :  u^\mathfrak{F}_\C \>\cong >> U^\mathfrak{F}(\C)\,\,\, \zeta \mapsto I_{2g}+ \zeta$
is an isomorphism. We can identify $U^{\mathfrak{F}}(\C)$ with its Lie algebra $u^\mathfrak{F}_\C$ by this isomorphism and regard $U^{\mathfrak{F}}(\C)$ as a complex space.
Moreover, for any ring $\mathfrak{R}$ in $\{\Z,\Q,\R,\C\},$ $U^{\mathfrak{F}}(\mathfrak{R})$(the set of all $\mathfrak{R}$-points of the algebraic group $U^{\mathfrak{F}}$) can  be regarded
as an $\mathfrak{R}$-module by
\begin{equation}\label{multiplication group- additive group}
 U^{\mathfrak{F}}(\mathfrak{R})\cong M_{k(k+1)/2}(\mathfrak{R})\cap u^{\mathfrak{F}}_\C,
\end{equation}
Therefore $U^{\mathfrak{F}}(\C)$ has a natural integer structure $U^{\mathfrak{F}}(\Z)$
and for any ring $\mathfrak{R}$ in $\{\Z,\Q,\R,\C\},$ there is an isomorphism $U^{\mathfrak{F}}(\Z)\otimes_\Z\mathfrak{R}=U^{\mathfrak{F}}(\mathfrak{R}).$
The corollary \ref{algebraic subgroups by boundary-component} ensures
that any element $\gamma \in N^{\mathfrak{F}}(\Z)$ defines an automorphism
$\overline{\gamma}:  U^{\mathfrak{F}}(\Z)\>>> U^{\mathfrak{F}}(\Z), \,\,\, u\mapsto \gamma u\gamma^{-1}.$
Thus we obtain a group morphism 
$j_{\mathfrak{F}}: N^{\mathfrak{F}}\to
\Aut(U^{\mathfrak{F}})$
such that there is $$j_{\mathfrak{F}}: N^{\mathfrak{F}}(\mathfrak{R})\to
\Aut(U^{\mathfrak{F}}(\mathfrak{R}))\,\, \gamma\mapsto \overline{\gamma}:=((\cdot) \mapsto \gamma(\cdot) \gamma^{-1})$$ for any $\Z$-algebra $\mathfrak{R}.$
We see that if $\gamma \in U^{\mathfrak{F}}$ then
$\overline{\gamma}$ is the identity in $\Aut(U^{\mathfrak{F}}).$

In general, there is the Levi-decomposition of $N^{\mathfrak{F}},$ i.e., a
semi-product of rational algebraic groups(cf.\cite{Del73}\&\cite{AMRT}):
$N^{\mathfrak{F}}=\underbrace{(G_h^{\mathfrak{F}}\times G_l^{\mathfrak{F}})}_{\mbox{direct product }}\cdot \sW^{\mathfrak{F}}.$
Moreover, we have :
\begin{itemize}
\item The
$p_{h,\mathfrak{F}}: N^{\mathfrak{F}}\to  G_h^{\mathfrak{F}} \,\,\mbox{ and }\,\,p_{l,\mathfrak{F}}: N^{\mathfrak{F}}\to G_l^{\mathfrak{F}}$ are surjective and defined over $\Q,$
\item the $G_l^{\mathfrak{F}}(\R)\cdot \sW^{\mathfrak{F}}(\R)$ acts trivially on $\mathfrak{F}$
and the $G_h^{\mathfrak{F}}$ is semi-simple,

\item the $G_h^{\mathfrak{F}}\cdot \sW^{\mathfrak{F}}$ centralizes $U^{\mathfrak{F}}$
 and  the $G_l^{\mathfrak{F}}$ is reductive without compact factors.
\end{itemize}

\begin{example}[Cf.\cite{Chai86} and \cite{Nam}]\label{example-Levi-decomposition}
Consider the cusp $\mathfrak{F}_k,$ we compute that

\begin{eqnarray*}
\sN(\mathfrak{F}_k)=\{\left(
  \begin{array}{ccccc}
   A_{11} & 0_{k,g-k} &A_{12}  & *   \\
   *        & f      & *      & *    \\
   A_{21}& 0_{k,g-k} & A_{22}  &*  \\
   0_{g-k,k}    & 0_{g-k,g-k}  & 0_{g-k,k}     & -^tf^{-1}  \\
  \end{array}
\right)\in \Sp(g,\R)   &|& \\ \left(
               \begin{array}{cc}
                  A_{11} &  A_{12} \\
                  A_{21} &  A_{22} \\
               \end{array}
             \right)\in \Sp(k,\R),
   &&  f\in\GL(g-k,\R)\}.
\end{eqnarray*}
\begin{eqnarray*}
G_l^{\mathfrak{F}_k}(\R)^+ &=&  \{ \left(
   \begin{array}{cccc}
     I_{k} & 0     & 0     & 0 \\
     0       & f & 0     & 0 \\
     0       & 0     &I_{k}& 0 \\
     0       & 0     & 0     & ^tf^{-1}\\
   \end{array}
 \right) \,\, |\,\, f\in \mathrm{GL}(g-k,\R),\,\det f>0\},\\
 &\cong& \mathrm{GL}(g-k,\R)^{+}  \\
G_h^{\mathfrak{F}_k}(\R)=G_h^{\mathfrak{F}_k}(\R)^+  &=&  \{\left(
  \begin{array}{ccccc}
   A_{11} & 0_{k,g-k} &A_{12}  & 0  \\
   0        & I_{g-k}      & 0      & 0   \\
   A_{21}& 0_{k,g-k} & A_{22}  & 0  \\
   0_{g-k,k}    & 0_{g-k,g-k}  & 0_{g-k,k}     & I_{g-k}  \\
  \end{array}
\right)   |  \left(
               \begin{array}{cc}
                  A_{11} &  A_{12} \\
                  A_{21} &  A_{22} \\
               \end{array}
             \right)\in \Sp(k,\R) \}\\
      &\cong & \Sp(k,\R).
\end{eqnarray*}
Thus, the action
$G_l^{\mathfrak{F}_k}(\R) \times C(\mathfrak{F}_k)\>>> C(\mathfrak{F}_k)\,\,\, (M, A)\longmapsto MAM^{-1}$
is equivalent to the action
$\GL(g-k,\R)\times \mathrm{Sym}^{+}_{g-k}(\R) \>>> \mathrm{Sym}^{+}_{g-k}(\R)\,\,\,(f, u)
\longmapsto f\cdot u\cdot ^tf.$
\end{example}
\begin{lemma}Let $\mathfrak{F}=\mathfrak{F}(W_\R)$ be a cusp of $\mathfrak{H}_g$ where $W_\R=W_\Z\otimes\R$ is a rationally define subspace of $V_\R.$
Let $W$ represent a linear space over $\Z$ given by $W(\mathfrak{R}):=W_\Z\otimes_{\Z} \mathfrak{R}$
   for any $\Z$-algebra $\mathfrak{R}.$
Both $ G_l^{\mathfrak{F}}$ and $G_h^{\mathfrak{F}}$ are algebraic group defined over $\Q,$ and there are two isomorphisms
$ G_l^{\mathfrak{F}}\cong \GL(W) \,\,\mbox{ and }\,\, G_h^{\mathfrak{F}}\cong \Sp(W^{\bot}/W,\psi).$
\end{lemma}
\begin{lemma}\label{key-in-Gamma_F}
The homomorphism $j_{\mathfrak{F}}: N^{\mathfrak{F}}(\R)\to
\Aut(U^{\mathfrak{F}}(\R))$ induces a homomorphism $j_{C(\mathfrak{F})}: N^{\mathfrak{F}}(\R)\to
\Aut(C(\mathfrak{F})).$ Moreover, $j_{C(\mathfrak{F})}$ factors through $p_{l,\mathfrak{F}}: N^{\mathfrak{F}}(\R)\>>> G_l^{\mathfrak{F}}(\R),$
i.e, there is a commutative diagram
$$
\begin{TriCDA}
{G_l^{\mathfrak{F}}(\R)} {\NE  E p_{l,\mathfrak{F}} E}{\SE E {j_{C(\mathfrak{F})}} E} {N^{\mathfrak{F}}(\R)}
{\>j_{C(\mathfrak{F})}>>} {\Aut(C(\mathfrak{F}))}
\end{TriCDA}.
$$
\end{lemma}
\begin{proof}It is sufficient to prove the statements in case of $\mathfrak{F}=\mathfrak{F}_k.$
Let $\Omega_{\mathfrak{F}}$ be a fixed point in $C(\mathfrak{F}).$
Since $G_h^{\mathfrak{F}}\cdot \sW^{\mathfrak{F}}$ centralizes $U^{\mathfrak{F}},$
We obtain that
\begin{eqnarray*}
  C(\mathfrak{F}) &=&\mbox{the orbit of $\Omega_{\mathfrak{F}}$ for the adjoint action of $G_l^{\mathfrak{F}}(\R)$ on $U^{\mathfrak{F}}(\R)$} \\
   &=& \mbox{the orbit of $\Omega_{\mathfrak{F}}$ for the adjoint action of $N^{\mathfrak{F}}(\R)$ on $U^{\mathfrak{F}}(\R)$ }.
\end{eqnarray*}
The first equality is given by the computation in the example \ref{example-Levi-decomposition}.
Therefore, we get the $j_{C(\mathfrak{F})}: N^{\mathfrak{F}}(\R)\to
\Aut(C(\mathfrak{F}))$ and $j_{C(\mathfrak{F})}$ factors through the morphism $p_{l,\mathfrak{F}}.$
\end{proof}

Define $\Gamma_{\mathfrak{F}}:=\Gamma\cap
\sN(\mathfrak{F}),\,\,\, \overline{\Gamma_{\mathfrak{F}}}:= j_{C(\mathfrak{F})}(\Gamma_{\mathfrak{F}}).$
Since $\sN(\mathfrak{F})\subset N^{\mathfrak{F}}(\R)$ and $\sN(\mathfrak{F})^{+}=N^{\mathfrak{F}}(\R)^{+},$
the group  $\Gamma_{\mathfrak{F}}$ is  a discrete subgroup of the real Lie group $\sN(\mathfrak{F})$(cf. \cite{Mil}),
and there is an inclusion
$ \overline{\Gamma_{\mathfrak{F}}}=j_{C(\mathfrak{F})}(\Gamma\cap G_l^{\mathfrak{F}}(\R))\>\subset>>\Aut(U^{\mathfrak{F}}(\Z)) \cap \Aut(C(\mathfrak{F})).$
\begin{definition}[Cf.\cite{AMRT}\&\cite{FC}]\label{admissible polyhedral decomposition}
Let $\mathfrak{F}$ be a cusp  of $\mathfrak{H}_g$ and $\mathbb{G}$ an
arithmetic subgroup of $\Sp(g,\Q)$ with an action on $C(\mathfrak{F}).$
A \textbf{$\mathbb{G}$-admissible polyhedral decomposition} of
$C(\mathfrak{F})$ is a collection of convex rational polyhedral cones
$\Sigma_{\mathfrak{F}}=\{\sigma_\alpha^{\mathfrak{F}}\}_{\alpha}\subset
\overline{C(\mathfrak{F})}$ satisfying  that $\Sigma_{\mathfrak{F}}$ is a fan,
$\overline{C(\mathfrak{F})}^{\mathrm{rc}}=\bigcup\limits_{\alpha}
\sigma_\alpha^{\mathfrak{F}}$ and
$\mathbb{G}$ has an action on the set $\Sigma_{\mathfrak{F}}$ with finitely many orbits.
A $\mathbb{G}$-admissible polyhedral decompositions
 $\Sigma_{\mathfrak{F}}$ of $C(\mathfrak{F})$
is \textbf{regular} with respect to  an arithmetic subgroup $\Gamma^{'}\subset \Sp(g,\Z)$ if $\Sigma_{\mathfrak{F}}$ is regular  with respect to the lattice
$\Gamma^{'}\cap U^{\mathfrak{F}}(\Z).$
\end{definition}
\begin{myrem}
Each convex rational polyhedral cone in $C(\mathfrak{F})$
is automatically strong since $C(\mathfrak{F})$ is
non-degenerate. Moreover, the construction in \cite{AMRT} implies that
every cone in a $\overline{\Gamma_{\mathfrak{F}}}$-admissible
polyhedral decomposition $\Sigma_{\mathfrak{F}}=\{\sigma_\alpha\}$ of
$C(\mathfrak{F})$ is actual a face of one top-dimensional cone in $\Sigma_{\mathfrak{F}}$(i.e., a cone
$\sigma_{\max}^\mathfrak{F}\in\Sigma_{\mathfrak{F}}$ with $\dim\sigma_{\max}^\mathfrak{F}=\dim_\R
C(\mathfrak{F})$).
\end{myrem}

\begin{definition}[Cf.\cite{AMRT}]\label{admissible-family-polyhedral-decomposition}
Let $\Gamma\subset \Sp(g,\Z)$ be an arithmetic subgroup.
\begin{myenumi}
\item
A \textbf{$\Gamma$-admissible family of polyhedral decompositions}
is a collection $\{\Sigma_{\mathfrak{F}}\}_{\mathfrak{F}}$ of $\overline{\Gamma_{\mathfrak{F}}}$-admissible polyhedral decompositions
 $\Sigma_{\mathfrak{F}}=\{\sigma_\alpha^{\mathfrak{F}}\}$ of $C(\mathfrak{F}),$ $\mathfrak{F}$ running over the cusps of $\mathfrak{H}_g$ such that
\begin{itemize}
  \item if $\mathfrak{F}^2=\gamma\mathfrak{F}^1$ for some $\gamma\in
\Gamma$ then
$\Sigma_{\mathfrak{F}^2}=\gamma(\Sigma_{\mathfrak{F}^1}),$
  \item if $\mathfrak{F}^1\prec\mathfrak{F}^2$ then
$\Sigma_{\mathfrak{F}^2}=\{\sigma_\alpha^{\mathfrak{F}^1}\cap \overline{C(\mathfrak{F}^2)} \,\,|\,\,
\sigma_\alpha^{\mathfrak{F}^1}\in \Sigma_{\mathfrak{F}^1}\}.$
\end{itemize}

\item A $\Gamma$-admissible family of polyhedral decompositions
$\{\Sigma_{\mathfrak{F}}\}_{\mathfrak{F}}$ is called \textbf{regular} if
for any cusp $\mathfrak{F}$ the $\overline{\Gamma_{\mathfrak{F}}}$-admissible polyhedral decompositions
 $\Sigma_{\mathfrak{F}}$ of $C(\mathfrak{F})$
is regular with respect to $\Gamma.$
\end{myenumi}
\end{definition}
\begin{lemma}[Cf.\cite{Chai85} and \&\cite{Nam}]\label{family-decomposion from a minimal boundary}
Let $\Gamma\subset \Sp(g,\Z)$ be an arithmetic subgroup. Let
$\Sigma_{\mathfrak{F}_0}:=\{\sigma_\alpha^{\mathfrak{F}_0}\}$ be a
$\overline{\Gamma_{\mathfrak{F}_0}}$(or $\mathrm{GL}(g,\Z)$)-admissible polyhedral decomposition of
$C(\mathfrak{F}_0),$ where $\mathfrak{F}_0$ is the standard minimal cusp of $\mathfrak{H}_g.$
The $\Sigma_{\mathfrak{F}_0}$ endows a $\Gamma$-admissible family of
polyhedral decompositions $\{\Sigma_{\mathfrak{F}}\}_{\mathfrak{F}}$
 as follows :
\begin{description}
  \item[Step 1]For any minimal cusp $\mathfrak{F}_{\mathrm{min}}=M(\mathfrak{F}_0)$ with $M\in\Sp(g,\Z),$
we define
  $$\Sigma_{\mathfrak{F}_{\mathrm{min}}}:=M(\Sigma_{\mathfrak{F}_0})=\{ M\sigma^{\mathfrak{F}_0}_\alpha M^{-1}\,\,|\,\, \sigma^{\mathfrak{F}_{0}}_\alpha\in \Sigma_{\mathfrak{F}_0}\}.$$

  \item[Step 2] For any cusp $\mathfrak{F},$  if $\mathfrak{F}_{\mathrm{min}}$ is a minimal cusp with $\mathfrak{F}_{\mathrm{min}}\prec\mathfrak{F}$
then we define
$$\Sigma_{\mathfrak{F}}:=\Sigma_{\mathfrak{F}_{\mathrm{min}}}|_{\overline{C(\mathfrak{F})}}
  =\{\sigma^{\mathfrak{F}_{\mathrm{min}}}_\alpha \cap \overline{C(\mathfrak{F})} \,\,|\,\, \sigma^{\mathfrak{F}_{\mathrm{min}}}_\alpha\in \Sigma_{\mathfrak{F}_{\mathrm{min}}}\}.$$

  \end{description}
Moreover, we have :
\begin{myenumi}
\item  The $\Sigma_{\mathfrak{F}_0}$ is regular with respect to $\Gamma$ if $\Sigma_{\mathfrak{F}_0}$ is regular with respect to $\Sp(g,\Z).$

\item If $\Sigma_{\mathfrak{F}_0}$ is regular with respect to $\Gamma$ then the family $\{\Sigma_{\mathfrak{F}}\}_{\mathfrak{F}}$ is regular.
\end{myenumi}

\end{lemma}

\subsection{General toroidal compactifications of $\sA_{g,\Gamma}$}
Let
$D(\mathfrak{F}):=\bigcup\limits_{\alpha\in \mathrm{U}^{\mathfrak{F}}(\C)}\alpha\mathfrak{S}_g$
for any cusp $\mathfrak{F}.$ Actually,
$D(\mathfrak{F})=\bigcup\limits_{C\in u^\mathfrak{F}}\exp(\sqrt{-1}C)\mathfrak{S}_g\subset \check{\mathfrak{S}}_g$
since the group $W_{0}^{\mathfrak{F}}(G)$ acts transitively on $\mathfrak{S}_g.$ Here is another version of the lemma \ref{nilpotent orbit}:
\begin{proposition}[Cf.\cite{Sat},\cite{AMRT},\cite{KW65},\cite{Fal}]\label{embedding of the third type}
Let $\mathfrak{F}=\mathfrak{F}(W)$ be a cusp of $\mathfrak{S}_g.$ We have that
$
  D(\mathfrak{F})=\{F\in \check{\mathfrak{S}}_g\,\,|\,\, \sqrt{-1}\psi(v,\overline{v})>0\,\, \forall \,0\neq v\in F\cap W^\perp
  \}
$
and a diffeomorphism
$$\varphi : u^{\mathfrak{F}}_\C\times v^{\mathfrak{F}}_\R\times \mathfrak{F}\>\cong>> D(\mathfrak{F})\,\,\,(a+\sqrt{-1}b, c,F)\longmapsto \exp(a+\sqrt{-1}b)\exp(c)(\check{F})$$
such that
$ \varphi^{-1}(\mathfrak{S}_g)=(u^{\mathfrak{F}}+\sqrt{-1}C(\mathfrak{F}))\times v^{\mathfrak{F}}_\R\times \mathfrak{F},$ where $\check{F}$ is defined in the lemma \ref{nilpotent orbit}.
\end{proposition}

\begin{corollary}\label{Siegel domain of the third type}
Let $\mathfrak{F}=\mathfrak{F}(W)$ be a cusp of
$\mathfrak{S}_g$ with $\dim_\R W=k.$ The space $V^\mathfrak{F}(\R)$ has a
natural complex structure such that it is isomorphic to
$M_{g-k,k}(\C).$ Moreover, there is an isomorphism
$\Phi: U^\mathfrak{F}(\C)\times M_{g-k,k}(\C)\times \mathfrak{F} \>\cong>>D(\mathfrak{F}).$
\end{corollary}

We sketch the construction of a general toroidal compactification
$\overline{\sA}^{\mathrm{tor}}_{g,\Gamma}$ of
$\sA_{g,\Gamma}$ by following \cite{AMRT} in the  complex
analytic topology. Let $\Sigma^\Gamma_{\mathrm{tor}}=\{\Sigma_\mathfrak{F}\}_{\mathfrak{F}}$
be a general $\Gamma$-admissible family of polyhedral
decompositions.
Let $\mathfrak{F}$ be an arbitrary cusp of
$\mathfrak{H}_g.$ Define
$L_{\mathfrak{F}}:=\Gamma\cap U^{\mathfrak{F}}(\Q),  \,\, M_\mathfrak{F}=L_{\mathfrak{F}}^{\vee}:=\Hom_\Z(L_{\mathfrak{F}},\Z).$
We have $L_{\mathfrak{F}}=\Gamma\cap U^{\mathfrak{F}}(\Z)$ as $\Gamma\subset \Sp(g,\Z).$
Using \ref{multiplication group- additive group}, $L_{\mathfrak{F}}$ is a full
lattice in the vector space $U^{\mathfrak{F}}(\C)\cong u^{\mathfrak{F}}_\C.$
The algebraic torus
$T_{\mathfrak{F}}:=\Spec\C[M_\mathfrak{F}]$ is isomorphic analytically to
$\frac{U^{\mathfrak{F}}(\C)}{\Gamma\cap U^{\mathfrak{F}}(\Q)}=(\C^{\times})^{\dim u^{\mathfrak{F}}}.$
Then, there is an analytic isomorphism
$\frac{D(\mathfrak{F})}{\Gamma\cap U^{\mathfrak{F}}(\Q)}\cong T_{\mathfrak{F}}\times v^{\mathfrak{F}}_\R\times \mathfrak{F}$
by the embedding of Siegel domain of third type
$D(\mathfrak{F})\cong u^{\mathfrak{F}}_\C\times v^{\mathfrak{F}}_\R\times\mathfrak{F},$
so that there is a principal $T_{\mathfrak{F}}$-bundle
$\frac{D(\mathfrak{F})}{\Gamma\cap U^{\mathfrak{F}}(\Q)} \to \frac{D(\mathfrak{F})}{u^{\mathfrak{F}}_\C}.$
We note that if $\mathfrak{F}_{\min}$ is a minimal cusp of $\mathfrak{H}_g$ then $D(\mathfrak{F})/u^{\mathfrak{F}}_\C$ is a space of one single point.
For any cone $\sigma\in \Sigma_{\mathfrak{F}},$
we replace $T_{\mathfrak{F}}$ with $X_\sigma$ by the open embedding $T_{\mathfrak{F}}\>\hookrightarrow>>X_\sigma,$
and obtain a fiber bundle
$P_{\mathfrak{F},\sigma}: X_{\sigma}\times_{T_{\mathfrak{F}}} \frac{D(\mathfrak{F})}{\Gamma\cap U^{\mathfrak{F}}(\Q)}\to \frac{D(\mathfrak{F})}{u^{\mathfrak{F}}_\C}.$
Define
\begin{equation}\label{neighborhood-toiroidal-embedding}
\mbox{$\widetilde{\Delta}_{\mathfrak{F},\sigma}$= the interior of the
closure of $\frac{\mathfrak{H}_g}{\Gamma\cap U^{\mathfrak{F}}(\Q)}$ in $
X_{\sigma}\times_{T_{\mathfrak{F}}} \frac{D(\mathfrak{F})}{\Gamma\cap
U^{\mathfrak{F}}(\Q)}.$}
\end{equation}
Using the similar method that gluing $X_{\sigma}$'s to construct the scheme $X_{\Sigma_\mathfrak{F}},$
we glue all $P_{\mathfrak{F},\sigma}: X_{\sigma}\times_{T_{\mathfrak{F}}} \frac{D(\mathfrak{F})}{\Gamma\cap U^{\mathfrak{F}}(\Q)}\to \frac{D(\mathfrak{F})}{u^{\mathfrak{F}}_\C}$'s to obtain a fiber bundle
$P_{\mathfrak{F}}: X_{\Sigma_{\mathfrak{F}}}\times_{T_{\mathfrak{F}}} \frac{D(\mathfrak{F})}{\Gamma\cap U^{\mathfrak{F}}(\Q)}\to \frac{D(\mathfrak{F})}{u^{\mathfrak{F}}_\C}$
with fiber $X_{\Sigma_\mathfrak{F}},$ and we also glue
$\Delta_{\mathfrak{F},\sigma}$'s altogether to obtain an analytic space
$Z_{\mathfrak{F}}^{'}.$ We call $Z_{\mathfrak{F}}^{'}$ the \textbf{partial
compactification in the direction $\mathfrak{F}$} of the Siegel
variety $\sA_{g,\Gamma}.$ For
any cusp $\mathfrak{F}$ and any element $\gamma\in \Gamma$ there is
an analytic isomorphism
$
\Pi_{\mathfrak{F},\gamma\mathfrak{F}}^{'}: Z_{\mathfrak{F}}^{'}\>\cong>>
Z_{\mathfrak{\gamma F}}^{'};
$
and for any two cusps $\mathfrak{F}_1\prec \mathfrak{F}_2$ there is an
analytic \'etale morphism
$\Pi_{\mathfrak{F}_2,\mathfrak{F}_1}^{'}: Z_{\mathfrak{F_2}}^{'}\to
Z_{\mathfrak{F_1}}^{'}$
(cf.Lemma 1 in $\S 5$ Chap.III \cite{AMRT}).

\begin{example}[Cf.\cite{Chai86}]\label{local-coordinate-system}Let $\mathfrak{F}=\mathfrak{F}(W)$ be  a cusp of $\dim_\R
W=k>0.$ Assume $\Gamma$ is neat. We can describe
$Z_{\mathfrak{F},\sigma}$ in a local coordinate system : Let $\sigma$
be a regular cone in $C(\mathfrak{F})$ of top-dimension
$k(k+1)/2.$
 As in \cite{AMRT} and \cite{Mum77}, we take
 a $\Z$-basis $\{\zeta_\alpha\}_{1}^{k(k+1)/2}$ of $\Gamma\cap
U^\mathfrak{F}(\Q)$ such that
$\R_{+}\zeta_1,\cdots,\R_{+}\zeta_{k(k+1)/2}$ are  all edges of
$\sigma.$ Any $u\in
U^{\mathfrak{F}}(\C)$ can be written as $u=\sum_{\alpha}u_\alpha\zeta_{\alpha}.$ Then we get an open
embedding
$$\delta_{\mathfrak{F}}: \mathfrak{H}_g \>\hookrightarrow >> U^{\mathfrak{F}}(\C)\times M_{g-k,k}(\C)\times \mathfrak{F}(W)\>\cong>>
\C^{k(k+1)/2}\times
M_{g-k,k}(\C)\times\mathfrak{S}(W^\perp/W,\psi_W),$$
so that
$
(u_\alpha, s_i, t_j)\in \C^{k(k+1)/2}\times
M_{g-k,k}(\C)\times\mathfrak{S}(W^\perp/W,\psi_W)
$
endows a coordinate system of $\mathfrak{H}_g.$ Thus, we have the following commutative diagram
\begin{equation}\label{local-coordinate-1-1}
 \begin{CDS}
\mathfrak{H}_g \>\subset>>\C^{k(k+1)/2}\times M_{k,g-k}(\C)\times\mathfrak{S}(W^\perp/W,\psi_W) \\
\V V  V \novarr \V V(z_\alpha:=\exp(2\pi\sqrt{-1}u_{\alpha}), \,s_i, \,t_j) V  \\
\frac{ \mathfrak{H}_{g}}{\Gamma\cap U^{\mathfrak{F}}(\Q)}\>\subset >>
(\C^*)^{k(k+1)/2}\times
M_{k,g-k}(\C)\times\mathfrak{S}(W^\perp/W,\psi_W)
\end{CDS}
\end{equation}
and there holds
$
\bigcup\limits_\alpha\{(z_\alpha,
s_i,t_j)\in\widetilde{\Delta}_{\mathfrak{F},\sigma} \,\,|\,\,
z_\beta=0 \} \>\subset >>
\widetilde{\Delta}_{\mathfrak{F},\sigma}\setminus \frac{\mathfrak{H}_{g}}{\Gamma\cap U^{\mathfrak{F}}(\Q)}.
$
\end{example}

\begin{lemma}\label{toroidal-embedding-partial-compactification}
Let $\Sigma^\Gamma_{\mathrm{tor}}=\{\Sigma_\mathfrak{F}\}_{\mathfrak{F}}$
be a $\Gamma$-admissible family of polyhedral decompositions.
\begin{myenumi}
\item Let $\mathfrak{F}$ be a cusp of $\mathfrak{H}_g.$
The collection $\{\sS(\mathfrak{F},\sigma)\}_{\sigma\in
\Sigma_\mathfrak{F}}$ is a stratification of $Z_{\mathfrak{F}}^{'}.$ In
particular,
$\overline{\sS(\mathfrak{F},\sigma)}^{\mathrm{cl}}=\coprod_{\delta\in \Sigma_{\mathfrak{F}},\delta\succeq\sigma}
\sS(\mathfrak{F},\delta)\,\,  \forall \sigma\in \Sigma_{\mathfrak{F}},$
where $\overline{\sS(\mathfrak{F}, \sigma)}^{\mathrm{cl}}$ is  the
closure of $\sS(\mathfrak{F}, \sigma)$ in $Z_{\mathfrak{F}}^{'}.$
 Moreover, the open embedding
$\mathbb{U}_{\mathfrak{F}}(:=\frac{\mathfrak{H}_g}{\Gamma\cap U^{\mathfrak{F}}(\Q)})\>\subset>>Z_{\mathfrak{F}}^{'}$
is a toroidal embedding without self-intersections, i.e., every
irreducible component of $Z_{\mathfrak{F}}^{'}\setminus
\mathbb{U}_{\mathfrak{F}}$ is normal.

\item For any two cusps $\mathfrak{F}^1, \mathfrak{F}^2$  with
$\mathfrak{F}^1\prec\mathfrak{F}^2,$
$\Pi_{\mathfrak{F}^2,\mathfrak{F}^1}^{'}:(Z_{\mathfrak{F^2}}^{'},
\mathbb{U}_{\mathfrak{F}^2})\to
(Z_{\mathfrak{F^1}}^{'},\mathbb{U}_{\mathfrak{F}^1})$ is a toroidal
morphism.
\end{myenumi}
\end{lemma}
\begin{proof}
By carefully reading \cite{AMRT}, one can get the proof easily.
\end{proof}

The disjoint union
$\widetilde{\sA_{g,\Gamma}}:=\bigsqcup\limits^{\circ}_{\mathfrak{F}}Z_{\mathfrak{F}}^{'}$
has a natural $\Gamma$-action. An equivalent relation $R$ is
defined on $\widetilde{\sA_{g,\Gamma}}$ : We say $x\sim^R y$
for $x\in Z_{\mathfrak{F}_1}^{'}, y\in Z_{\mathfrak{F}_2}^{'}$  if and
only if there exists a cusp $\mathfrak{F}_3,$ a point $z\in Z_{\mathfrak{F}_3}^{'}$ and  a $\gamma\in \Gamma$ such that
$\mathfrak{F}_1\preceq\mathfrak{F}_3, \gamma\mathfrak{F}_2\preceq\mathfrak{F}_3$
and
\begin{equation}\label{gluing-condition}
\mbox{($\Pi_{\mathfrak{F}_3,\mathfrak{F}_1}^{'}:
Z_{\mathfrak{F}_3}^{'}\to Z_{\mathfrak{F}_1}^{'}$) maps
$z$ to $x,$ \,\,\,
   ($\Pi_{\mathfrak{F}_3,\gamma \mathfrak{F}_2}^{'}: Z_{\mathfrak{F}_3}^{'}\to Z_{\gamma\mathfrak{F}_2}^{'}$) maps $z$ to
  $\Pi_{\mathfrak{F}_2,\gamma\mathfrak{F}_2}^{'}(y).$}
\end{equation}
Shown in $\S5-\S6$ Chap. III. \cite{AMRT}, the transitivity
condition of the relation $R$ holds and the relation graph in
$\widetilde{\sA_{g,\Gamma}}\times
\widetilde{\sA_{g,\Gamma}}$ is closed. We then obtain a compact
Hausdorff analytic variety
$\overline{\sA}^{\mathrm{tor}}_{g,\Gamma}:=\frac{\widetilde{\sA_{g,\Gamma}}}{\sim^R},$
which is called a  \textbf{toroidal compactification} of
$\sA_{g,\Gamma}.$ The $\overline{\sA}^{\mathrm{tor}}_{g,\Gamma}$
is an algebraic space, but not projective in general. However, a
theorem of Tai(cf.Chap.IV \cite{AMRT}) shows that if
$\Sigma_{\mathrm{tor}}^\Gamma$ is projective(cf. Chap. IV. of \cite{AMRT}) then
$\overline{\sA}^{\mathrm{tor}}_{g,\Gamma}$ is a projective
variety.
The main theorem I in \cite{AMRT} shows that
$\overline{\sA}^{\mathrm{tor}}_{g,\Gamma}$ is the unique
Hausdorff analytic variety containing $\sA_{g,\Gamma}$ as an
open dense subset such that $\overline{\sA}^{\mathrm{tor}}_{g,\Gamma}=\bigcup\limits_{\mathfrak{F}}\pi_{\mathfrak{F}}^{'}(Z_{\mathfrak{F}}^{'})$
and for every cusp $\mathfrak{F}$ of $\mathfrak{H}_g$ there is an open
morphisms $\pi_{\mathfrak{F}}^{'}$ making the following  diagram
commutative
\begin{equation}\label{compactification-digram-1}
\begin{CDS}
\frac{ \mathfrak{H}_{g}}{\Gamma\cap U^{\mathfrak{F}}(\Q)} \> \hookrightarrow >> Z_{\mathfrak{F}}^{'}\\
\V V  V \novarr \V V \pi_{\mathfrak{F}}^{'} V \\
  \sA_{g,\Gamma} \> \hookrightarrow >>
  \overline{\sA}^{\mathrm{tor}}_{g,\Gamma}.
\end{CDS}\\
\end{equation}

\subsection{Infinity boundary divisors on toroidal compactifications}
For a polyhedral decomposition
$\Sigma_{\mathfrak{F}_0}:=\{\sigma_\alpha^{\mathfrak{F}_0}\}$ of
$C(\mathfrak{F}_0),$ all edges in
$\Sigma_{\mathfrak{F}_0}$ are taken into two disjoint sets :
\begin{itemize}
    \item Interior-edge=$\{\rho\in \Sigma_{\mathfrak{F}_0}(1)\,\,|\,\, \mathrm{Int}(\rho)
\subset C(\mathfrak{F}_0) \},$

    \item Boundary-edge=$\{\rho\in \Sigma_{\mathfrak{F}_0}(1)\,\,|\,\, \mathrm{Int}(\rho)
\cap C(\mathfrak{F}_0)= \emptyset  \},$
\end{itemize}
where $\mathrm{Int}(\sigma)$ is defined to be the set of relative interior points of $\sigma\in \Sigma_{\mathfrak{F}_0}.$
These two sets are both $\Gamma_{\mathfrak{F}_0}(=\Gamma\cap
\sN(\mathfrak{F}_0))$-invariant for any arithmetic subgroup $\Gamma$ of
$\Sp(g,\Z).$

\begin{lemma}\label{edge-to-boundary-component}
Let $\Gamma$ be a neat arithmetic subgroup of $\Sp(g,\Q)$ and  $\Sigma_{\mathfrak{F}_0}=\{ \sigma_\alpha^{\mathfrak{F}_0} \}_\alpha$ a $\overline{\Gamma_{\mathfrak{F}_0}}$-admissible polyhedral decomposition of
$C(\mathfrak{F}).$ Let $\rho$ be an edge in the set Boundary-edge.

Assume that $\Sigma_{\mathfrak{F}_0}$ is regular with respect to $\Gamma.$
There is  a unique rationally-defined one dimensional isotropic
subspace $W_\rho$ of $V(=V^{(0)})$ such that $\mathrm{Int}(\rho)=C(\mathfrak{F}(W_\rho)).$
Moreover, for any cone $\sigma\in \Sigma_{\mathfrak{F}_0}$ there exists a unique cusp $\mathfrak{F}_\sigma$ such that
$\mathfrak{F}_0\preceq \mathfrak{F}_\sigma$ and $\mathrm{Int}(\sigma)\subset C(\mathfrak{F}_\sigma).$
\end{lemma}
\begin{proof}
By Theorem 3 in $\S4.4$ of Chap.III \cite{AMRT}, any proper
rational boundary component of $C(\mathfrak{F}_0)$(cf.Definition \ref{proper-boundary-component-of-cone}) is of form
$C(\mathfrak{F}_1)$ by a cusp $\mathfrak{F}_1$ with
$\mathfrak{F}_0\prec\mathfrak{F}_1.$
Thus, there is a cusp $\mathfrak{F}^{'}$ different with $\mathfrak{F}_0$ such that
$\mathfrak{F}_0\prec\mathfrak{F}^{'}$ and $\rho \in \overline{C(\mathfrak{F}^{'})}.$

Suppose $\mathfrak{F}^{'}=\mathfrak{F}(W^{'})$ has $\dim_\Q W^{'}\geq 2.$
By the lemma \ref{family-decomposion from a minimal boundary}, $\Sigma_{\mathfrak{F}_0}$ endows a $\Gamma$-admissible family of
polyhedral decompositions $\{\Sigma_{\mathfrak{F}}\}_{\mathfrak{F}}$ and the decomposition $\Sigma_{\mathfrak{F}^{'}}$ is regular to with respect to $\Gamma.$
Since  $\overline{C(\mathfrak{F}_0)}^{\mathrm{rc}}=\bigcup\limits_{\alpha}
\sigma_\alpha^{\mathfrak{F}_0},$
there exists a top-dimensional cone $\sigma_{\max}\in \Sigma_{\mathfrak{F}_0}$ and a face $\tau$ of $\sigma_{\max}$
such that $\rho$ is an edge of $\tau$ and $\tau\in \Sigma_{\mathfrak{F}^{'}}$ with $\dim\tau=\dim C(\mathfrak{F}^{'}).$
Since the cone $\sigma_{\max}$ is regular, there is a face $\delta$ of $\sigma$ satisfying  $\rho\in \delta$ and $\delta\notin \Sigma_{\mathfrak{F}^{'}}.$
Thus, we have another cusp $\mathfrak{F}^{''}=\mathfrak{F}(W^{''})$ that
$W^{''}$ is not a subspace of $W^{'}$ and $\delta \in \Sigma_{\mathfrak{F}^{''}}.$
Then, $\rho\in u^{\mathfrak{F}(W^{'}\cap W^{''})}$ by the lemma \ref{convex-cone}, and so $\mathrm{Int}(\rho)$ is in a proper rational boundary component of $C(\mathfrak{F}^{'}).$

By recursion, we obtain that $\mathrm{Int}(\rho)$ is a proper rational boundary component of $C(\mathfrak{F}_0)$
and  there is a rationally-defined one dimensional isotropic
subspace $W_\rho$ of $V$ such that $\mathrm{Int}(\rho)=C(\mathfrak{F}(W_\rho)).$
The uniqueness is due to(3) of the lemma \ref{convex-cone}.

The rest can be obtained by similar method.
\end{proof}

\begin{definition} Let $\Gamma\subset \Sp(g,\Z)$ be an arithmetic subgroup and $\mathfrak{F}$ a cusp of $\mathfrak{H}_g.$
A $\overline{\Gamma_{\mathfrak{F}}}$-admissible polyhedral decomposition
$\Sigma_{\mathfrak{F}}$ of $C(\mathfrak{F})$ is \textbf{$\Gamma$-separable} if a $\gamma\in \overline{\Gamma_{\mathfrak{F}}}$ satisfies
$\gamma(\sigma)\cap \sigma\neq \{0\}$ for a cone $\sigma\in\Sigma_{\mathfrak{F}}$
then $\gamma$ acts as the identity on the cone $\sigma.$
\end{definition}
\begin{myrem}
Note that any $\overline{\Gamma_{\mathfrak{F}}}$-admissible polyhedral decomposition
$\Sigma_{\mathfrak{F}}$ of $C(\mathfrak{F})$ can be subdivided into
another regular $\overline{\Gamma_{\mathfrak{F}}}$-admissible polyhedral decomposition $\widetilde{\Sigma}_{\mathfrak{F}}$(cf.\cite{AMRT},\cite{FC}),
and it is obvious that the regular refinement $\widetilde{\Sigma}_{\mathfrak{F}}$ is also  $\Gamma$-separable provided that  $\Sigma_{\mathfrak{F}}$ is  $\Gamma$-separable.
\end{myrem}

In fact, our definition of a $\overline{\Gamma_{\mathfrak{F}}}$-admissible polyhedral decomposition
with $\Gamma$-separability is compatible with the condition (ii) in \S 2.4 Chap IV \cite{FC}.The following is easy :
\begin{lemma}\label{lemma for non-selfintersection-1}
Let $\Gamma$ be an arithmetic subgroup of $\Sp(g,\Q).$ Let $\Sigma_{\mathfrak{F}}$  be a $\overline{\Gamma_{\mathfrak{F}}}$-admissible polyhedral decomposition of $C(\mathfrak{F}),$  where $\mathfrak{F}$ is a cusp of $\mathfrak{H}_g.$

Assume that the decomposition $\Sigma_{\mathfrak{F}}$ is regular with respect to $\Gamma.$ The following two conditions are equivalent :
(i) $\Sigma_{\mathfrak{F}}$ is $\Gamma$-separable; (ii) if an element  $\gamma\in \overline{\Gamma_{\mathfrak{F}}}$ satisfies
$\gamma(\sigma)\cap \sigma\neq \{0\}$ for a cone $\sigma\in\Sigma_{\mathfrak{F}}$
then $\gamma$ acts as the identity on $C(\mathfrak{F}).$
\end{lemma}
\begin{proof}A regular top-dimensional cone $\sigma_{\max}^\mathfrak{F}=\{\sum\limits_{i=1}^t \lambda_iy_i | \lambda_i\in \R_{\geq 0},\, \,\, i=1,\cdots, t\}$
has $\{y_1,\cdots,y_t\}$ as a $\Z$-basis of $\Gamma\cap U^{\mathfrak{F}}(\Z).$
Then that $\Sigma_{\mathfrak{F}}$ is $\Gamma$-separable if and only if that any $\gamma\in \overline{\Gamma_{\mathfrak{F}}}$ satisfying
$\gamma(\sigma)\cap \sigma\neq \{0\}$ for some cone $\sigma\in\Sigma_{\mathfrak{F}}$
acts as the identity on all cones in $\{\tau\in\Sigma_{\mathfrak{F}}\,\,|\,\, \tau\cap\gamma(\sigma)\cap \sigma\neq \{0\}\}.$
So if $\gamma$ acts as the identity on $\sigma_{\max}^\mathfrak{F}$ then $\gamma$ must act as the identity on $C(\mathfrak{F}).$\\
\end{proof}

\begin{definition}
 Let $\Gamma\subset \Sp(g,\Z)$ be an arithmetic subgroup. A
 \textbf{symmetric $\Gamma$-admissible family of polyhedral
 decompositions} is the $\Gamma$-admissible family of polyhedral
 decompositions induced by a
 $\overline{\Gamma_{\mathfrak{F}_0}}$(or $\mathrm{GL}(g,\Z)$)-admissible polyhedral decomposition of
 $C(\mathfrak{F}_0)$ as in the lemma \ref{family-decomposion from a minimal
 boundary}. A \textbf{symmetric toroidal compactification} of a Siegel variety
$\sA_{g,\Gamma}$ is a compactification constructed by a symmetric admissible family of polyhedral
 decompositions.
\end{definition}
When we say a toroidal compactification constructed by some admissible polyhedral decomposition of
 $C(\mathfrak{F}_0),$ we always mean a symmetric toroidal compactification.

Due to the lemma \ref{family-decomposion from a minimal boundary}, we have :
\begin{lemma}\label{lemma for non-selfintersection-2} Let $\Gamma\subset \Sp(g,\Z)$ be an arithmetic subgroup.
Let $\Sigma_{\mathfrak{F}_0}$ be a $\overline{\Gamma_{\mathfrak{F}_0}}$(or $\mathrm{GL}(g,\Z)$)-admissible polyhedral decomposition of
$C(\mathfrak{F}_0).$ Let $\{\Sigma_{\mathfrak{F}}\}_{\mathfrak{F}}$ be the  symmetric $\Gamma$-admissible family of polyhedral decompositions induced by $\Sigma_{\mathfrak{F}_0}$(
cf.Lemma \ref{family-decomposion from a minimal boundary}).
\begin{myenumi}
\item  For any cusp $\mathfrak{F}$ of $\mathfrak{H}_g,$
the induced $\overline{\Gamma_{\mathfrak{F}}}$-admissible polyhedral decomposition
$\Sigma_{\mathfrak{F}}$ of $C(\mathfrak{F})$ is $\Gamma$-separable.

\item For any subgroup $\Gamma^{'}$ of $\Gamma$ with finite index,
the decomposition $\Sigma_{\mathfrak{F}_0}$ is $\Gamma^{'}$-separable as a $\overline{\Gamma^{'}_{\mathfrak{F}_0}}$-admissible polyhedral decomposition.

\end{myenumi}
\end{lemma}

In general, given a toroidal
compactification $\overline{D/\Gamma}$ of a locally symmetric variety $D/\Gamma,$
the toroidal embedding $D/\Gamma\subset \overline{D/\Gamma}$  is not without self-intersection,
the main theorems II in \cite{AMRT} says that $D/\Gamma\subset \overline{D/\Gamma}$
is without monodromy  in sense that for each stratum $\mathcal{O}$ the branches of $\overline{D/\Gamma}\setminus D/\Gamma$ through $\mathcal{O}$ are not permuted by going around loops in $\mathcal{O}.$
For more efficient applications of toroidal compactifications in geometry,  we continue to
exploit the infinity boundary divisors explicitly.
\begin{theorem}\label{Infity-divisor-on-toroidal-compactification}
Let $\Gamma\subset \Sp(g,\Z)$ be an arithmetic subgroup and
 let $\Sigma_{\mathfrak{F}_0}:=\{\sigma_\alpha^{\mathfrak{F}_0}\}$ be a
 $\overline{\Gamma_{\mathfrak{F}_0}}$(or $\mathrm{GL}(g,\Z)$)-admissible polyhedral decomposition of $C(\mathfrak{F}_0)$
 where $\mathfrak{F}_0$ is the standard minimal cusp of $\mathfrak{H}_g.$
Let  $\overline{\sA}_{g,\Gamma}$ be the toroidal
compactification of $\sA_{g,\Gamma}$ constructed by $\Sigma_{\mathfrak{F}_0}$ and $D_\infty:=
\overline{\sA}_{g,\Gamma}\setminus\sA_{g,\Gamma}$
the boundary divisor.

Assume that the
decomposition $\Sigma_{\mathfrak{F}_0}$ is  regular with respect to $\Gamma.$

\begin{myenumi}
\item The number of irreducible components of $D_\infty$ is equal
to
$$[\Sp(g,\Z): \Gamma]+[\Sp(g,\Z): \Gamma]\times\#\{ \mbox{$\Gamma_{\mathfrak{F}_0}$-orbits in Interior-edge}\}. $$

\item The compactification
$\overline{\sA}_{g,\Gamma}$ is a smooth compact
analytic variety with simple normal crossing boundary divisor $D_\infty,$ if the group $\Gamma$ is neat and the decomposition  $\Sigma_{\mathfrak{F}_0}$
is $\Gamma$-separable.
\end{myenumi}
\end{theorem}

\begin{proof}
We define an  equivalent relation on the set of cusps :
$\mathfrak{F}\sim^{\Gamma}\mathfrak{F}^{'}$ if and only if there exists  an element
$\gamma \in \Gamma$ such that $\mathfrak{F}^{'}=\gamma\mathfrak{F}.$ The
equivalent class is denoted by $[\cdot].$

Let $\{\Sigma_{\mathfrak{F}}\}_{\mathfrak{F}}$ be the symmetric
$\Gamma$-admissible family of polyhedral decompositions induced by
the given decomposition $\Sigma_{\mathfrak{F}_0}.$ For any cusp
$\mathfrak{F},$ we also define an  equivalent relation on
$\Sigma_{\mathfrak{F}}$ : $\sigma^{\mathfrak{F}}_\alpha
\sim^{\Gamma_{\mathfrak{F}}} \sigma^{\mathfrak{F}}_\beta$ if and only if
there exists an element $\gamma \in \Gamma_{\mathfrak{F}}$ such that
$\sigma^{\mathfrak{F}}_\beta=\gamma(\sigma^{\mathfrak{F}}_\alpha).$ Denoted this
equivalent class  by $[\cdot]_{\mathfrak{F}}.$

For each cusp $\mathfrak{F},$ a basic fact is that the group
$\Gamma_{\mathfrak{F}}/\Gamma\cap U^{\mathfrak{F}}(\R)$ acts properly
discontinuously on $Z_{\mathfrak{F}}^{'}$(cf.Proposition 1 in \S6.3
Chap.III\cite{AMRT}), thus the morphism $\pi_{\mathfrak{F}}^{'}$
factors through a morphism $\pi_{\mathfrak{F}}:Z_{\mathfrak{F}} \to
\overline{\sA}_{g,\Gamma}$ so that there is a commutative diagram :
\begin{equation}\label{compactification-digram-2}
\begin{CDS}
Z_{\mathfrak{F}}^{'}  \>\mathrm{pr}_{\mathfrak{F}} >> Z_{\mathfrak{F}}\\
\novarr\SE \pi_{\mathfrak{F}}^{'} EE  \V V \pi_{\mathfrak{F}} V  \\
\novarr \novarr \overline{\sA}_{g,\Gamma},
\end{CDS}
\end{equation}
where $Z_{\mathfrak{F}}$ is  the quotient of $Z_{\mathfrak{F}}^{'}$ by
$\Gamma_{\mathfrak{F}}/\Gamma\cap U^{\mathfrak{F}}(\Q)$ and $\mathrm{pr}_{\mathfrak{F}}$ is a quotient morphism.

Let $\mathfrak{F}^1,\mathfrak{F}^2$ be two arbitrary cusps with
$\mathfrak{F}^2\prec\mathfrak{F}^1.$ We actually have
        $$\overline{\Gamma_{\mathfrak{F}^1}}-\mbox{orbit of
$\sigma_{\alpha}^{\mathfrak{F}^1}$ in
$\Sigma_{\mathfrak{F}^1}$}=(\overline{\Gamma_{\mathfrak{F}^2}}-\mbox{orbit
of $\sigma_{\alpha}^{\mathfrak{F}^1}$ in
$\Sigma_{\mathfrak{F}^2}$})\bigcap \Sigma_{\mathfrak{F}^1}\,\,\,\,
\forall \sigma_{\alpha}^{\mathfrak{F}^1}\in \Sigma_{\mathfrak{F}^1}$$
by the fact of $\Sigma_{\mathfrak{F}^1}=\{\sigma_{\beta}\in
\Sigma_{\mathfrak{F}^2} \,\,|\,\, \sigma_\beta\subset
\overline{C(\mathfrak{F}^1)}\}.$
Then, we get an induced morphism $\Pi_{\mathfrak{F}^1,\mathfrak{F}^2} : Z_{\mathfrak{F}^1}  \>  >>
Z_{\mathfrak{F}^2}.$ Moreover, $\Pi_{\mathfrak{F}^1,\mathfrak{F}^2}$ is a local isomorphism
satisfying  the following commutative diagram
$$
\begin{CDS}
Z_{\mathfrak{F}^1}^{'}  \> \Pi_{\mathfrak{F}^1,\mathfrak{F}^2}^{'}  >>
Z_{\mathfrak{F}^2}^{'}\\
\V\mathrm{pr}_{\mathfrak{F}^1} V V \novarr  \V V \mathrm{pr}_{\mathfrak{F}^2}  V  \\
Z_{\mathfrak{F}^1}  \> \Pi_{\mathfrak{F}^1,\mathfrak{F}^2} >> Z_{\mathfrak{F}^2}.
\end{CDS}.
$$

Recall the construction of a general toroidal compactification of
$\mathfrak{H}_g/\Gamma,$ the condition \ref{gluing-condition} of the
relation $R$ ensures that the following diagram is commutative :
\begin{equation}\label{1-proof-Infity-divisor-on-toroidal-compactification}
\begin{CDS}
Z_{\mathfrak{F}_2}  \> \Pi_{\mathfrak{F}_2,\mathfrak{F}_1} >> Z_{\mathfrak{F}_1}\\
\novarr\SE \pi_{\mathfrak{F}_2} EE  \V V \pi_{\mathfrak{F}_1} V  \\
\novarr \novarr \overline{\sA}_{g,\Gamma}
\end{CDS}
\end{equation}
for any two cusps $\mathfrak{F}_1, \mathfrak{F}_2$ with $\mathfrak{F}_1\prec
\mathfrak{F}_2.$

Let $\mathfrak{F}_{\min }$ be an arbitrary minimal cusp. Denote by $\mathfrak{U}_{[\mathfrak{F}_{\min }]}:=
\pi^{'}_{\mathfrak{F}_{\min }}(Z_{\mathfrak{F}_{\min }}^{'}).$ It is
well-defined as
$\pi^{'}_{\mathfrak{F}}(Z_{\mathfrak{F}}^{'})=\pi^{'}_{\gamma\mathfrak{F}}(Z_{\gamma\mathfrak{F}}^{'})$
for $\gamma\in \Gamma,\,\forall \mathfrak{F}.$ Since
$\pi^{'}_{\mathfrak{F}_{\min}}$ is an open morphism, the set
$\mathfrak{U}_{[\mathfrak{F}_{\min }]}$ is open in
$\overline{\sA}_{g,\Gamma}.$ There are useful properties F1-F4 :
\begin{description}
    \item[F1]  Let $\mathfrak{F}$ be an arbitrary cusp and $\gamma$ an arbitrary element in $\Gamma.$
It is obvious that there is an induced isomorphism
$\Pi_{\mathfrak{F},\gamma\mathfrak{F}} : Z_{\mathfrak{F}} \to
Z_{\gamma\mathfrak{F}}$ such that the following diagram is commutative
$$
\begin{CDS}
Z_{\mathfrak{F}}^{'}  \> \Pi_{\mathfrak{F},\gamma\mathfrak{F}}^{'}  >>
Z_{\gamma\mathfrak{F}}^{'}\\
\V\mathrm{pr}_{\mathfrak{F}} V V \novarr  \V V \mathrm{pr}_{\gamma\mathfrak{F}} V  \\
Z_{\mathfrak{F}}  \> \Pi_{\mathfrak{F},\gamma\mathfrak{F}} >>
Z_{\gamma\mathfrak{F}}.
\end{CDS}.
$$
Thus, it is easy to get
$\mathrm{pr}_{\mathfrak{F}}(\sS(\mathfrak{F}, \sigma^{\mathfrak{F}}_\alpha))
=\mathrm{pr}_{\mathfrak{F}}(\sS(\mathfrak{F}, \kappa(\sigma^{\mathfrak{F}}_\alpha)))\,\,
\forall \sigma^{\mathfrak{F}}_\alpha\in \Sigma_{\mathfrak{F}},\, \forall \kappa\in \Gamma_{\mathfrak{F}_0}.$
For any cone
$\sigma^{\mathfrak{F}}\in \Sigma_{\mathfrak{F}},$ we define
$\sY(\mathfrak{F},[\sigma^{\mathfrak{F}}]_{\mathfrak{F}}):=\mathrm{pr}_{\mathfrak{F}}(
\sS(\mathfrak{F}, \sigma^{\mathfrak{F}})).$  The collection
$\{\sY(\mathfrak{F},[\sigma^{\mathfrak{F}}_\alpha]_{\mathfrak{F}})\}_{\sigma^{\mathfrak{F}}_\alpha\in
\Sigma_{\mathfrak{F}}}$ is then a stratification of $Z_{\mathfrak{F}}.$

\item[F2] Let $\mathfrak{F}^1,\mathfrak{F}^2$ be two arbitrary cusps with
$\mathfrak{F}^2\prec\mathfrak{F}^1.$ By the lemma \ref{toroidal-embedding-partial-compactification},
we obtain that
\begin{equation}\label{2-proof-Infity-divisor-on-toroidal-compactification}
\Pi_{\mathfrak{F}^1,\mathfrak{F}^2}^{-1}(\sY(\mathfrak{F}^2,[\sigma_\alpha^{\mathfrak{F}^2}]_{\mathfrak{F}^2}))=
\left\{
  \begin{array}{ll}
    \sY(\mathfrak{F}^1,[\tau]_{\mathfrak{F}^1}) & \hbox{ if $\exists\tau \in [\sigma_\alpha^{\mathfrak{F}^2}]_{\mathfrak{F}^2}$ with $\tau\in  \Sigma_{\mathfrak{F}^1}$} \\
    \emptyset  & \hbox{ others}
  \end{array}
 \right.
\end{equation}
since $\Sigma_{\mathfrak{F}^1}=\{\sigma_{\beta}\in
\Sigma_{\mathfrak{F}^2} \,\,|\,\, \sigma_\beta\subset
\overline{C(\mathfrak{F}^1)}\}.$

   \item[F3] For  any cusp $\mathfrak{F}(W)$ and for any element $\gamma$ in $\Gamma,$
$\mathfrak{F}(W\cap \gamma(W))$ is the unique minimal one in the set of cusps
$\{\mathfrak{F}\,\,|\,\,\mathfrak{F}(W)\prec\mathfrak{F},\,\mathfrak{F}(\gamma(W))\prec\mathfrak{F}\},$
and so we can glue $Z_{\mathfrak{F}(W)}^{'}$ and $Z_{\mathfrak{F}(\gamma(W))}^{'}$ along
$Z_{\mathfrak{F}(W\cap \gamma(W))}^{'}.$

On the other hand, if we restrict the action of the relation $R$ on $Z_{\mathfrak{F}_{\min}}^{'}$ then this relation $R$ is reduced to the action of the $\Gamma_{\mathfrak{F}_{\min}}$ on
$Z_{\mathfrak{F}_{\min}}^{'}.$
We indeed obtain an analytic isomorphism
\begin{equation}\label{3-proof-Infity-divisor-on-toroidal-compactification}
\pi_{\mathfrak{F}_{\min}}: Z_{\mathfrak{F}_{\min}} \>\cong>>
\mathfrak{U}_{[\mathfrak{F}_{\min}]}.
\end{equation}
Therefore all $
\pi_{\mathfrak{F}}: Z_{\mathfrak{F}} \>>>
\overline{\sA}_{g,\Gamma}\,\,\forall \mathfrak{F}
$ are local isomorphism by the diagram
\ref{1-proof-Infity-divisor-on-toroidal-compactification}.

   \item[F4] Let $\mathfrak{F}$ be an arbitrary cusp. We define $$\sO(\mathfrak{F},
[\sigma^{\mathfrak{F}}]_{\mathfrak{F}}):=\pi_{\mathfrak{F}}(\sY(\mathfrak{F},
[\sigma^{\mathfrak{F}}]_{\mathfrak{F}}))=\pi_{\mathfrak{F}}^{'}(\sS(\mathfrak{F},
\sigma^{\mathfrak{F}}))\,\, \forall \sigma^{\mathfrak{F}}\in
\Sigma_{\mathfrak{F}}.$$ Since $\pi_{\mathfrak{F}}$ is a local
isomorphism,
 $\{ \sO(\mathfrak{F},
[\sigma^{\mathfrak{F}}_\alpha]_{\mathfrak{F}})\}_{\sigma^{\mathfrak{F}}_\alpha\in
\Sigma_{\mathfrak{F}}}$ is also a stratification of
$\pi_{\mathfrak{F}}(Z_{\mathfrak{F}}).$ In particular,  $\{
\sO(\mathfrak{F}_{\min},
[\sigma^{\mathfrak{F}_{\min}}_\alpha]_{\mathfrak{F}_{\min}})\}_{\sigma^{\mathfrak{F}_{\min}}_\alpha\in
\Sigma_{\mathfrak{F}_{\min}}}$ is a stratification of
$\mathfrak{U}_{\mathfrak{F}_{\min}}.$  Furthermore,  we have isomorphisms
\begin{equation}\label{4-proof-Infity-divisor-on-toroidal-compactification}
\pi_{\mathfrak{F}_{\min}}:
\sY(\mathfrak{F}_{\min},[\sigma_\alpha^{\mathfrak{F}_{\min}}]_{\mathfrak{F}_{\min}})\>\cong
>>\sO(\mathfrak{F}_{\min},[\sigma_\alpha^{\mathfrak{F}_{\min}}]_{\mathfrak{F}_{\min}})\,\,
 \forall \sigma_\alpha^{\mathfrak{F}_{\min}}\in
 \Sigma_{\mathfrak{F}_{\min}}.
\end{equation}
\end{description}

For any two cusps $\mathfrak{F}_{\min}$ and $\mathfrak{F}^{'}_{\min},$ we observe that $\mathfrak{U}_{[\mathfrak{F}_{\min }]}\neq
\mathfrak{U}_{[\mathfrak{F}^{'}_{\min }]}$ if and only if
$[\mathfrak{F}_{\min}]\neq [\mathfrak{F}^{'}_{\min}].$
Thus, the toroidal
compactification $\overline{\sA}_{g,\Gamma}$ is
covered by $[\Sp(g,\Z):\Gamma]$ different open sets, i.e.,
$\overline{\sA}_{g,\Gamma}=\bigcup\limits_{i=1}^{[\Sp(g,\Z):\Gamma]}\mathfrak{U}_{[\mathfrak{F}_{\min }^i]}$
where $\mathfrak{F}_{\min }^i $'s are minimal cusps such that $[\mathfrak{F}_{\min}^i]\neq [\mathfrak{F}^{j}_{\min}]$ if $i\neq j.$
Therefore, to study $D_\infty$ on
$\overline{\sA}_{g,\Gamma},$ it is sufficient to
study  all codimension-one strata on $\mathfrak{U}_{\mathfrak{F}_{0}}.$

We always fixed $\mathfrak{F}_{\min }^1$ as the standard minimal cusp $\mathfrak{F}_0.$ Let $\rho$ be an edge in
$\Sigma_{\mathfrak{F}_0}$ and
$\sO(\mathfrak{F}_{0},[\rho]_{\mathfrak{F}_{0}})$ the  associated stratum
in $\mathfrak{U}_{[\mathfrak{F}_{0}]}.$

\begin{myenumiii}
\item \textbf{Claim 1.} {\it Suppose that $[\mathfrak{F}_{\min}]\neq [\mathfrak{F}_0].$
That $\sO(\mathfrak{F}_{0},[\rho]_{\mathfrak{F}_{0}})\bigcap
    \mathfrak{U}_{[\mathfrak{F}_{\min}]}\neq \emptyset$ if and only if
     there is an element $\gamma\in\Gamma$  such that $$ \rho\subset
(\overline{C(\mathfrak{F}_0)}\setminus C(\mathfrak{F}_0))\bigcap
(\overline{C(\gamma\mathfrak{F}_{\min})}\setminus
C(\gamma\mathfrak{F}_{\min})).$$}

\noindent{Proof of Claim 1.}{\it  The "if" part : By the
lemma \ref{edge-to-boundary-component}, there is a cusp
$\mathfrak{F}_\rho$ of depth one
 such that $\mathfrak{F}_0\prec \mathfrak{F}_\rho$ and $\mathrm{Int}(\rho)=C(\mathfrak{F}_\rho).$
Clearly, $C(\mathfrak{F}_\rho)$ is also a rational boundary component
of $C(\gamma\mathfrak{F}_{\min}).$ Thus,
$\mathfrak{F}_0\prec\mathfrak{F}_\rho$ and
$\gamma\mathfrak{F}_{\min}\prec\mathfrak{F}_\rho.$ The gluing condition
\ref{gluing-condition}, together with
\ref{2-proof-Infity-divisor-on-toroidal-compactification} and
\ref{4-proof-Infity-divisor-on-toroidal-compactification} shows
that $\sO(\mathfrak{F}_{0},[\rho]_{\mathfrak{F}_{0}})\cap
    \mathfrak{U}_{[\mathfrak{F}_{\min}]}\neq \emptyset.$

 The "only if" part : We have a $z\in
 \sO(\mathfrak{F}_{0},[\rho]_{\mathfrak{F}_{0}})$ such that $z\in \mathfrak{U}_{[\mathfrak{F}_0]} \cap\mathfrak{U}_{[\mathfrak{F}_{\min}]}.$
By the gluing condition \ref{gluing-condition}, there exists a
$\gamma\in \Gamma,$ a cusp $\mathfrak{F}^{'}$ and a point in $x\in
Z_{\mathfrak{F}^{'}}^{'}$ such that $\mathfrak{F}_0\prec\mathfrak{F}^{'},
\gamma\mathfrak{F}_{\min}\prec\mathfrak{F}^{'}$ and
$\pi_{\mathfrak{F}^{'}}^{'}(x)=z.$ Thus the edge $\rho$ is in $
\Sigma_{\mathfrak{F}^{'}}$ by the lemma
\ref{toroidal-embedding-partial-compactification} and so $\rho\in
\Sigma_{\gamma\mathfrak{F}_{\min}}=\gamma(\Sigma_{\mathfrak{F}_{\min}}).$
The cusp $\mathfrak{F}^{'}$ can not
be $\mathfrak{F}_0$ by the  condition $[\mathfrak{F}_{\min}]\neq [\mathfrak{F}_0].$ Since $\rho\in \Sigma_{\mathfrak{F}^{'}} \subset \Sigma_{\mathfrak{F}_0}\cap \Sigma_{\gamma\mathfrak{F}_{\min}},$ we obtain
$\rho\subset(\overline{C(\mathfrak{F}_0)}\setminus C(\mathfrak{F}_0))\bigcap
(\overline{C(\gamma\mathfrak{F}_{\min})}\setminus C(\gamma\mathfrak{F}_{\min})).$}

\item We construct a global irreducible divisor $D_\rho$ in $\overline{\sA}_{g,\Gamma}$
by the edge $\rho$ as follows :
\begin{itemize}
\item Suppose that $\rho$ is in the set \rm{Interior-edge}. The claim 1 shows
that
$$\sO(\mathfrak{F}_{0},[\rho]_{\mathfrak{F}_{0}})\subset\overline{\sA}_{g,\Gamma}\setminus
\bigcup_{i=2}^{[\Sp(g,\Z):\Gamma]}\mathfrak{U}_{[\mathfrak{F}_{\min }^i]}.
$$
Let $D_\rho$ be the closure of
$\sO(\mathfrak{F}_{0},[\rho]_{\mathfrak{F}_{0}})$ in
$\overline{\sA}_{g,\Gamma}.$ $D_\rho$ is a
global divisor in $\overline{\sA}_{g,\Gamma},$
and
$$D_\rho\subset \overline{\sA}_{g,\Gamma}\setminus
\bigcup_{i=2}^{[\Sp(g,\Z):\Gamma]}\mathfrak{U}_{[\mathfrak{F}_{\min
}^i]}\subset\mathfrak{U}_{[\mathfrak{F}_{\min
}^1]}=\mathfrak{U}_{[\mathfrak{F}_{0}]}.$$

\item Suppose that $\rho$ is in the set \rm{boundary-edge}.
By the above case we also have that $D_{\rho'}\subset \mathfrak{U}_{[\mathfrak{F}_{0}]}$ if and only if $\rho^{'}$ is in the set \rm{Interior-edge}.
Thus we can rearrange the order of $\mathfrak{F}^i_{\min}$'s
and get an integer $l\geq 2$ such that
$$\left\{%
\begin{array}{ll}
  \sO(\mathfrak{F}_{0},[\rho]_{\mathfrak{F}_{0}})\cap
\mathfrak{U}_{[\mathfrak{F}_{\min }^i]} \neq \emptyset, & \hbox{for $i=1,\cdots l$;} \\
  \sO(\mathfrak{F}_{0},[\rho]_{\mathfrak{F}_{0}})\cap
\mathfrak{U}_{[\mathfrak{F}_{\min }^i]}  = \emptyset, & \hbox{for other $i.$} \\
\end{array}%
\right.    $$ Due to  the claim 1, we  let  $\rho\subset
\Sigma_{\mathfrak{F}_{\min }^i}$ and $\rho$ in
$\overline{C(\mathfrak{F}_{\min }^i)}\setminus C(\mathfrak{F}_{\min }^i)$
only for $i=1,\cdots, l.$ Then, there is a cusp $\mathfrak{F}_\rho$
of depth one such that $\rho\in \Sigma_{\mathfrak{F}_\rho}$ and $\mathfrak{F}_{\min}^{i}\prec \mathfrak{F}_\rho$ for all $i=1,\cdots, l$ by the lemma
\ref{edge-to-boundary-component}. For each integer $i$ in $[1,\cdots, l],$ the equality
\ref{2-proof-Infity-divisor-on-toroidal-compactification} says
that $\sO_\rho:=\pi_{\mathfrak{F}_{\rho}}( \sY(\mathfrak{F}_\rho,
[\rho]_{\mathfrak{F}_\rho}))$ is in the stratum
$\sO_{i}:=\sO(\mathfrak{F}_{\min}^{i}, [\rho]_{\mathfrak{F}_{\min}^{i}})$
of $\mathfrak{U}_{[\mathfrak{F}_{\min}^{i}]},$ and $\sO_\rho$ is an open subset in each $\sO_{i}$ since each $\pi_{\mathfrak{F}_\rho}$ is a local
isomorphism. We glue all
$\sO_1, \cdots, \sO_l$ together along $\sO_\rho$ to obtain an analytic
subspace $\sS_\rho.$ Therefore, the closure $D_\rho$ of $\sS_{\rho}$ is a global
divisor in $\overline{\sA}_{g,\Gamma}.$
\end{itemize}
\end{myenumiii}

Now we begin to prove the statements (1) and(2) in the theorem.

\begin{myenumi}

\item Define a set
$$\mbox{All-boundary-edge}:=\bigcup_{\rho \in\mbox{Boundary-edge}}\bigcup_{\gamma\in \Sp(g,\Z)}\gamma(\rho).$$
From the construction of the divisor by an edge in $\Sigma_{\mathfrak{F}_0},$ we immediately obtain :
\begin{eqnarray*}
   && \mbox{the number of irreducible components of $D_\infty$}\\
   &=&\# \{ \mbox{$\Gamma$-orbits in All-boundary-edge}\}+[\Sp(g,\Z): \Gamma]\times\#\{ \mbox{$\Gamma_{\mathfrak{F}_0}$-orbits in
       Interior-edge}\}\\
   &=&[\Sp(g,\Z): \Gamma]+[\Sp(g,\Z): \Gamma]\times\#\{ \mbox{$\Gamma_{\mathfrak{F}_0}$-orbits in Interior-edge}\}.
\end{eqnarray*}
The last equality is due to the lemma
\ref{edge-to-boundary-component} and the fact that every two cusps
of depth one are $\Sp(g,\Z)$-equivalent(cf.Remark(4.16) in $\S5$
\cite{Nam}).

\item Suppose that the decomposition of $\Sigma_{\mathfrak{F}_0}$
is $\Gamma$-separable.
Let $\mathfrak{F}$ be an arbitrary cusp. By the lemma \ref{lemma for non-selfintersection-2} the induced $\overline{\Gamma_{\mathfrak{F}}}$-admissible polyhedral decomposition
$\Sigma_{\mathfrak{F}}=\{\sigma^{\mathfrak{F}}\}$ of $C(\mathfrak{F})$ is $\Gamma$-separable.\\
\noindent\textbf{Claim 2.} {\it
For any $\gamma\in \Gamma_{\mathfrak{F}}$ and
nontrivial $\sigma^{\mathfrak{F}}\in\Sigma_{\mathfrak{F}},$ the following are
equivalent :
\begin{myenumiii}
\item $\overline{\sS(\mathfrak{F},
\sigma^{\mathfrak{F}})}^{\mathrm{cl}}\bigcap \overline{\sS(\mathfrak{F},
\gamma(\sigma^{\mathfrak{F}}))}^{\mathrm{cl}}\neq \emptyset,$

\item $\gamma$ acts as the identity on any cone in the set $\{\tau\in \Sigma_{\mathfrak{F}}\,\,|\,\,\tau\succeq\sigma^{\mathfrak{F}} \}.$
\end{myenumiii}}
\noindent{Proof of Claim 2.} {\it Suppose $\overline{\sS(\mathfrak{F},
\sigma^{\mathfrak{F}})}^{\mathrm{cl}}\bigcap \overline{\sS(\mathfrak{F},
\gamma(\sigma^{\mathfrak{F}}))}^{\mathrm{cl}}\neq \emptyset.$ We have
$$\overline{\sS(\mathfrak{F},\sigma^{\mathfrak{F}})}^{\mathrm{cl}}=\coprod_{\delta\in \Sigma_{\mathfrak{F}},\delta\succeq\sigma^{\mathfrak{F}}}
\sS(\mathfrak{F},\delta),\,\,\mbox{ and }\,\, \overline{\sS(\mathfrak{F},\gamma(\sigma^{\mathfrak{F}}))}^{\mathrm{cl}}=\coprod_{\delta\in \Sigma_{\mathfrak{F}},\delta\succeq\gamma(\sigma^{\mathfrak{F}})}
\sS(\mathfrak{F},\delta)$$
by the lemma \ref{toroidal-embedding-partial-compactification}.
 Because the collection $\{\sS(\mathfrak{F},\sigma)\}_{\sigma\in
\Sigma_\mathfrak{F}}$ is a stratification of $Z_{\mathfrak{F}}^{'},$ we a cone $\delta\in \Sigma_{\mathfrak{F}}$ such that  $\delta\succeq\sigma^{\mathfrak{F}}$ and $\delta\succeq\gamma(\sigma^{\mathfrak{F}}).$
Thus $\sigma^{\mathfrak{F}}\subset \delta\cap\gamma^{-1}(\delta).$ Since $\Sigma_{\mathfrak{F}}$ is $\Gamma$-separable,
$\gamma$ acts as the identity on any cone $\tau\in \Sigma_{\mathfrak{F}}$ containing $\sigma^{\mathfrak{F}}.$}

Using similar arguments in Claim 2, we also obtain :\\
\noindent\textbf{Claim 3.} {\it Let $\gamma\in \Gamma_{\mathfrak{F}}$ and let $\sigma^{\mathfrak{F}}\neq\{0\} $ be a cone in $\Sigma_{\mathfrak{F}}.$
If $\sS(\mathfrak{F},\sigma^{\mathfrak{F}})\cap S(\mathfrak{F},\gamma(\sigma^{\mathfrak{F}}))\neq\emptyset $ then
$\sS(\mathfrak{F},\sigma^{\mathfrak{F}})=S(\mathfrak{F},\gamma(\sigma^{\mathfrak{F}}))$ and
the restriction
$\Pi_{\mathfrak{F},\gamma\mathfrak{F}}^{'}|_{\sS(\mathfrak{F},\sigma^{\mathfrak{F}})}: \sS(\mathfrak{F},\sigma^{\mathfrak{F}}) \> >> S(\mathfrak{F},\gamma(\sigma^{\mathfrak{F}}))$  is just the identification on $\sS(\mathfrak{F},\sigma^{\mathfrak{F}}).$}

By the claims 2 and 3, we have
$\mathrm{pr}_{\mathfrak{F}}:\overline{\sS(\mathfrak{F},\sigma^{\mathfrak{F}})}^{\mathrm{cl}}\>\cong >>
\overline{\sY(\mathfrak{F},[\sigma^{\mathfrak{F}}]_{\mathfrak{F}})}^{\mathrm{cl}}\,\,\, \forall \{0\}\neq\sigma^{\mathfrak{F}}\in \Sigma_{\mathfrak{F}}.$
Moreover, The statement(1) of the corollary \ref{divisor-cone} guarantees that
$\overline{\sY(\mathfrak{F},[\sigma^{\mathfrak{F}}]_{\mathfrak{F}})}^{\mathrm{cl}}$
has only normal singularities for any nonzero cone $\sigma^{\mathfrak{F}}\in
\Sigma_{\mathfrak{F}}.$
By the symmetry, we only need to consider singularities in the
open set $\mathfrak{U}_{[\mathfrak{F}_0]}.$
Since $\pi_{\mathfrak{F}_{0}}: Z_{\mathfrak{F}_{0}} \>\cong>>
\mathfrak{U}_{[\mathfrak{F}_{0}]}$ is an isomorphism, there is an isomorphism
$
\pi_{\mathfrak{F}_{0}}:
\overline{\sY(\mathfrak{F}_{0},[\sigma_\alpha^{\mathfrak{F}_{0}}]_{\mathfrak{F}_{0}})}^{\mathrm{cl}}\>\cong
>>\overline{\sO(\mathfrak{F}_{0},[\sigma_\alpha^{\mathfrak{F}_{0}}]_{\mathfrak{F}_{0}})}^{\mathrm{cl}}
$
for any cone $\sigma_\alpha^{\mathfrak{F}_{0}}\in
 \Sigma_{\mathfrak{F}_{0}},$ where $\overline{\sO(\mathfrak{F}_{0},[\sigma_\alpha^{\mathfrak{F}_{0}}]_{\mathfrak{F}_{0}})}^{\mathrm{cl}}$
is the closure of $\sO(\mathfrak{F}_{0},[\sigma_\alpha^{\mathfrak{F}_{0}}]_{\mathfrak{F}_{0}})$
in $\mathfrak{U}_{[\mathfrak{F}_0]}.$

For each edge $\rho$ in $\Sigma_{\mathfrak{F}_0},$ the global divisor $D_\rho$ has that
$$D_\rho\cap \mathfrak{U}_{[\mathfrak{F}_0]}= \overline{\sO(\mathfrak{F}_{0},[\rho]_{\mathfrak{F}_{0}})}^{\mathrm{cl}}\cong \overline{\sY(\mathfrak{F}_{0},[\rho]_{\mathfrak{F}_{0}})}^{\mathrm{cl}}.$$
Thus $D_\rho$ is a normal variety. In particular $D_\rho$ has non self-intersections.

Since $\Sigma_{\mathfrak{F}_0}$ is  regular with respect to $\Gamma,$
    the statement (2) of the
    corollary \ref{divisor-cone} guarantees that
    $Z_{\mathfrak{F}_0}^{'}=\overline{\sS(\mathfrak{F},\{0\})}^{\mathrm{cl}}$
    is smooth and $
\overline{\sO(\mathfrak{F},[\sigma^{\mathfrak{F}}]_{\mathfrak{F}})}^{\mathrm{cl}}(
\cong \overline{\sS(\mathfrak{F},\sigma^{\mathfrak{F}})}^{\mathrm{cl}})$
is also smooth for any $\sigma^{\mathfrak{F}}\in \Sigma_{\mathfrak{F}}$
with $\sigma^{\mathfrak{F}}\neq \{0\}.$ Again by the statement (2) of the
 corollary \ref{divisor-cone}, we obtain that the irreducible
components of $D_\infty$ intersect transversely.

Now we suppose that $\Gamma$ is neat.
The fundamental group of $A_{g,\Gamma}$ is then
isomorphic to $\Gamma,$ and so
$\overline{\Gamma_{\mathfrak{F}}}/U^{\mathfrak{F}}\cap \Gamma$ acts freely
on $Z_{\mathfrak{F}}^{'}$ for any cusp $\mathfrak{F},$ thus the morphism
$\pi_{\mathfrak{F}}^{'}: Z_{\mathfrak{F}}^{'}\to
\overline{\sA}_{g,\Gamma}$ is \'etale.
Therefore, $\mathfrak{U}_{[\mathfrak{F}_0]}$ is smooth since that $Z_{\mathfrak{F}_0}^{'}$ is smooth.
\end{myenumi}
\end{proof}
\begin{myrem}Assume the condition that $\Gamma$ is neat and the decomposition $\Sigma_{\mathfrak{F}_0}$
is $\Gamma$-separable.
When an edge $\rho\in \Sigma_{\mathfrak{F}_0}$ is exactly in the set Boundary-edge, by using the argument in Theorem 2.2 of \cite{Wang93} we can assert that the associated
irreducible boundary divisor $D_\rho \subset \overline{\sA}_{g,\Gamma}\setminus\sA_{g,\Gamma}$ is actually a smooth toroidal compactification $\overline{X}$ of
some locally symmetric variety $X.$  The variety $X$ actually has a direct factor like a  Siegel variety $\sA_{g-1,\Gamma^{'}}$ for some arithmetic subgroup $\Gamma^{'}\subset \Sp(g-1,\Z)$ induced by $\Gamma.$ The toroidal compactification $\overline{X}$ is then constructed by a $\Gamma$-admissible family which is induced by the decomposition $\Sigma_{\mathfrak{F}_0}.$ Moreover, if two edges $\rho_2,\rho_2$ in $\Sigma_{\mathfrak{F}_0}$ are both in the set Boundary-edge then  $D_{\rho_1}\cong D_{\rho_2}.$\\
\end{myrem}

\begin{definition}\label{geometrically-fine-compactification}
Let $\Gamma\subset \Sp(g,\Z)$ be an arithmetic subgroup.
Let $\Sigma_{\mathfrak{F}_0}:=\{\sigma_\alpha^{\mathfrak{F}_0}\}$ be an
 $\overline{\Gamma_{\mathfrak{F}_0}}$(or $\mathrm{GL}(g,\Z)$)-admissible polyhedral decomposition of $C(\mathfrak{F}_0),$ and let  $\overline{\sA}_{g,\Gamma}$ be the symmetric toroidal
compactification of $\sA_{g,\Gamma}$ constructed by $\Sigma_{\mathfrak{F}_0}.$

With respect to the open morphism
$\pi_{\mathfrak{F}_0}^{'}:Z_{\mathfrak{F}_0}^{'} \to
\overline{\sA}_{g,\Gamma},$ we define :
\begin{myenumi}
\item A top-dimensional cone $\sigma_{\max}$ in $\Sigma_{\mathfrak{F}_0}$ is said to be \textbf{$\Gamma$-fine} if
the restriction $\pi_{\mathfrak{F}_0}^{'}|_{B_\rho}$ is an isomorphism onto its image for every edge $\rho$ of itself,
where $B_\rho$ is the divisor constructed by $\rho$ on $Z_{\mathfrak{F}_0}^{'}.$

\item The constructed symmetric toroidal compactification  $\overline{\sA}_{g,\Gamma}$ of $\sA_{g,\Gamma}$ is called
\textbf{geometrically $\Gamma$-fine} if the following condition is satisfied :
The restriction $\pi_{\mathfrak{F}_0}^{'}|_{B_\rho}$ is an isomorphism onto its image,
where $B_\rho$ is the  divisor on $Z_{\mathfrak{F}_0}^{'}$ constructed by $\rho,$ $\rho$ running over the edges of $\Sigma_{\mathfrak{F}_0}.$
\end{myenumi}
\end{definition}
The proof of the theorem \ref{Infity-divisor-on-toroidal-compactification} tells us that a  $\Gamma$-separable decomposition $\Sigma_{\mathfrak{F}_0}$
will induce a geometrically $\Gamma$-fine symmetric toroidal compactification  $\overline{\sA}_{g,\Gamma}$ of $\sA_{g,\Gamma}.$
\begin{theorem}\label{geometrically-fine-1-1-to non-self-intersections}
Let $\Gamma\subset \Sp(g,\Z)$ be a neat arithmetic subgroup and
let $\Sigma_{\mathfrak{F}_0}$ be a
 $\overline{\Gamma_{\mathfrak{F}_0}}$(or $\mathrm{GL}(g,\Z)$)-admissible polyhedral decomposition of $C(\mathfrak{F}_0).$
Let $\overline{\sA}_{g,\Gamma}$  be the toroidal compactification of $\sA_{g,\Gamma}$ constructed by $\Sigma_{\mathfrak{F}_0}.$

Assume that the decomposition $\Sigma_{\mathfrak{F}_0}$ is regular with respect to $\Gamma.$
The following four conditions are equivalent :
\begin{myenumiii}
\item Every irreducible component of $D_{\infty}= \overline{\sA}_{g,\Gamma}\setminus\sA_{g,\Gamma}$ has non self-intersections;
\item The compactification $\overline{\sA}_{g,\Gamma}$ is  geometrically $\Gamma$-fine;
\item the decomposition $\Sigma_{\mathfrak{F}_0}$
is $\Gamma$-separable;
\item The infinity boundary divisor $D_{\infty}= \overline{\sA}_{g,\Gamma}\setminus\sA_{g,\Gamma}$ is simple normal crossing.
\end{myenumiii}
\end{theorem}
\begin{proof}
By the corollary \ref{divisor-cone}, that (i)$\Longleftrightarrow$(iv) is obviously true.
The proof of the theorem \ref{Infity-divisor-on-toroidal-compactification} actually shows that (i)$\Longleftrightarrow$(ii) and
(iii)$\Longrightarrow$(ii).

We now begin to show that (ii)$\Longrightarrow$(iii).

Since $\mathfrak{F}_0$ is a minimal cusp of $\mathfrak{H}_g,$
$Z_{\mathfrak{F}_0}^{'}$ is isomorphic to the toroidal variety $X_{\mathfrak{F}_0}.$
Let $\sigma$ be an arbitrary cone in $\Sigma_{\mathfrak{F}_0}.$ Suppose that $\gamma\in \overline{\Gamma_{\mathfrak{F}}}$ is an element such that
$\gamma(\sigma)\cap \sigma\neq \{0\}.$

We know that $\tau:=\gamma(\sigma)\cap \sigma$ is a face of $\sigma.$ Let $\rho$ be an arbitrary edge of $\tau.$ Then, $\rho_1:=\gamma^{-1}(\rho)$ is also an edge of $\sigma.$ Let $B_\rho$(resp. $B_{\rho_1}$) be the divisor on $Z_{\mathfrak{F}_0}^{'}$ corresponding to the edge $\rho$(resp. $\rho_1$).
We must have $\pi_{\mathfrak{F}_0}^{'}(B_{\rho})=\pi_{\mathfrak{F}_0}^{'}(B_{\rho_1}).$

\noindent{\bf Claim}(*) $\gamma(\rho)=\rho$ : {\it Otherwise, there is a top-dimensional cone $\sigma_{\max}\in \Sigma_{\mathfrak{F}_0}$ containing $\sigma$ and so $B_{\rho}$ intersects with $B_{\rho_1}$ transversely by the corollary \ref{divisor-cone}. It contradicts the condition that $\overline{\sA}_{g,\Gamma}$ is  geometrically $\Gamma$-fine.}

Since $\gamma$ is an automorphism of the lattice $\Gamma\cap U^{\mathfrak{F}_0}(\Z),$ $\gamma$ acts  as the identity on the edge $\rho$ so that
$\gamma$ acts as the identity on the cone $\tau.$

\noindent{\bf Claim }(**) $\tau=\sigma$ : {\it Otherwise, $\gamma(\sigma)$ and $\sigma$ are two different cone in  $\Sigma_{\mathfrak{F}_0}.$
Let $\sO^{\sigma}$(resp. $\sO^{\gamma(\sigma)}$) be the orbit in  $Z_{\mathfrak{F}_0}^{'}$ corresponding to the cone $\sigma$(resp. $\gamma(\sigma)$).
Let $\delta$ be an edge of $\tau.$ Then $\sO^{\sigma}\cup \sO^{\gamma(\sigma)}\subset B_{\delta},\,\, \, \sO^{\sigma}\cap \sO^{\gamma(\sigma)}=\emptyset.$
On the other hand, we also have $\pi_{\mathfrak{F}_0}^{'}(\sO^{\sigma})=\pi_{\mathfrak{F}_0}^{'}(\sO^{\gamma(\sigma)}),$
so that the image $\pi_{\mathfrak{F}_0}^{'}(B_{\delta})$ must have self-intersections. It is a contradiction.}

Therefore the $\gamma$  acts as the identity on the cone $\sigma.$
\end{proof}

\begin{example}[Central cone decomposition]\label{central-cone-decomposition}
In \cite{Igu67} and \cite{Nam}, Igusa and Namikawa introduce a projective $\GL(g,\Z)$-admissible rational polyhedral
decomposition $\Sigma_{\mathrm{cent}}$(\textbf{central cone
decomposition}) of $C(\mathfrak{F}_0)$
containing \textbf{principal cone}
$$\sigma_0:=\{X=(x_{ij})\in\mathrm{Sym}_g(\R)\,\,|\,\, x_{ij}\leq0(i\neq j), \sum\limits_{j=1}^gx_{ij}\geq 0(\forall i)\},$$
which is top-dimensional regular cone
with respect to the lattice basis of $U^{\mathfrak{F}_0}(\Z).$
If $g\leq 3$ then the following  properties are satisfied:
\begin{itemize}
  \item That $\Sigma_{\mathrm{cent}}$ is  regular with respect to $\Sp(g,Z),$ and all edges of top-dimensional cones in the decomposition
  $\Sigma_{\mathrm{cent}}$ are on the boundary of $C(\mathfrak{F}_0);$
  \item  the principal cone $\sigma_0$ is the unique maximal cone in $\Sigma_{\mathrm{cent}}$ up to $\GL(g,\Z).$
\end{itemize}
Therefore, we obtain that {\it if the genus $g\leq 3$ then  the central cone decomposition $\Sigma_{\mathrm{cent}}$ can not be
$\Sp(g,\Z)$-separable so that the boundary divisor of the induced toroidal compactification is not normal crossing.}
\end{example}

\begin{corollary}\label{intersection-theory-boundary-divisor}
Let $\Gamma\subset \Sp(g,\Z)$ be a neat arithmetic subgroup and
let $\overline{\sA}_{g,\Gamma}$ be a geometrically $\Gamma$-fine  toroidal compactification of the Siegel variety $\sA_{g,\Gamma}:=\mathfrak{H}_g/\Gamma$ constructed by  a
 $\overline{\Gamma_{\mathfrak{F}_0}}$(or $\mathrm{GL}(g,\Z)$)-admissible polyhedral decomposition $\Sigma_{\mathfrak{F}_0}:=\{\sigma_\alpha^{\mathfrak{F}_0}\}$ of $C(\mathfrak{F}_0)$ regular with respect to $\Gamma,$  where $\mathfrak{F}_0$ is the standard minimal cusp of $\mathfrak{H}_g.$

Let $D_{1},\cdots, D_{d}$ be $d$ different irreducible components of the simple normal crossing boundary divisor $D_\infty=\overline{\sA}_{g,\Gamma}\setminus\sA_{g,\Gamma}.$
We have :
\begin{myenumi}
\item That $D_{1}\cap \cdots\cap D_{d}\neq \emptyset $
if and only if that $d\leq \dim_\C \sA_{g,\Gamma}$ and
there exists a minimal cusp $\mathfrak{F}_{\min}$ of $\mathfrak{H}_g$ and a top-dimensional cone $\sigma_{\max}$ in $\Sigma_{\mathfrak{F}_{\min}}$ with $d$ differential edges $\rho_i\,\, i=1,\cdots, d$ such that
$D_{i}=D_{\rho_i}\,\, i=1,\cdots, d,$
where  $D_\rho$ is the divisor constructed by an edge $\rho.$
\item Assume that $d=\dim_\C \sA_{g,\Gamma}.$ There are only two cases:
\begin{itemize}
  \item $D_{1}\cap \cdots\cap D_{d}=\emptyset$ and so $D_{1}\cdot D_{2}\cdots D_{d}=0.$
  \item $D_{1}\cap \cdots\cap D_{d}\neq \emptyset$ and the intersection number $D_{1}\cdot D_{2}\cdots D_{d}=1.$
\end{itemize}
\end{myenumi}
\end{corollary}
\begin{proof}
It is straightforward by the intersection theory on toric geometry(cf.\cite{Ful}).
\end{proof}



\section{Volume forms related to compactifications and constrained conditions of decompositions of cones from the viewpoint of K\"ahler-Einstein metric}

\vspace{0.5cm}

We still denote $\mathfrak{F}_k$ the cusp $\mathfrak{F}(V^{(g-k)})$ of $\mathfrak{H}_g$ for $1\leq k\leq g.$
We take coordinate system $\tau=(\tau_{ij})_{1\leq i,j\leq g}\in\mathfrak{H}_g$ of
the Siegel space
$\mathfrak{H}_g=\{\tau\in M_{g}(\C)\,\,|\,\, \tau=^t\tau, \mathrm{Im}(\tau)>0 \}.$
The Bergman metric on $\mathfrak{H}_g$
is $$d s^2=\sum_{1\leq i\leq j\leq g, 1\leq k\leq l\leq g}g_{ij, \overline{kl}}d \tau_{ij}d \overline{\tau_{kl}}:=\mathrm{Tr}(d\tau \mathrm{Im}(\tau)^{-1}d\overline{\tau} \mathrm{Im}(\tau)^{-1})$$
and its K\"ahler form is $\omega_{\mathrm{can}}=\frac{\sqrt{-1}}{2}\sum\limits_{1\leq i\leq j\leq g, 1\leq k\leq l\leq g}g_{ij, \overline{kl}}d \tau_{ij}\wedge d \overline{\tau_{kl}}.$
Then the volume form is
\begin{eqnarray*}
  \frac{1}{(\frac{g(g+1)}{2})!}\omega_{\mathrm{can}}^{\frac{g(g+1)}{2}} &=&(\frac{\sqrt{-1}}{2})^{\frac{g(g+1)}{2}}\det(g_{ij, \overline{kl}})dV_g  \\
   &=&(\frac{\sqrt{-1}}{2})^{\frac{g(g+1)}{2}}\frac{dV_g}{(\det \mathrm{Im}(\tau))^{g+1}}=:\Phi_g(\tau),
\end{eqnarray*}
where
\begin{eqnarray*}
  dV_g(\tau) &:=& 2^{\frac{g(g-1)}{2}}\bigwedge\limits_{1\leq i\leq j\leq g}d \tau_{ij}\wedge d \overline{\tau_{ij}} \\
   &=&  (-1)^{\frac{(g-1)g(g+1)(g+2)}{8}}2^{\frac{g(g-1)}{2}}(\bigwedge\limits_{1\leq i\leq j\leq g}d \tau_{ij})\wedge(\bigwedge\limits_{1\leq i\leq j\leq g}d \overline{\tau_{ij}})
\end{eqnarray*}
(cf.\cite{Sie43}). The Bergman metric is K\"ahler-Einstein, i.e.,
\begin{equation}\label{KE}
    \sqrt{-1}\partial\overline{\partial}\log \det(g_{ij, \overline{kl}})=-\sqrt{-1}\partial\overline{\partial}\log(\det \mathrm{Im}(\tau))^{g+1}=\frac{g+1}{2}\omega_{\mathrm{can}}.
\end{equation}

Let $\Gamma\subset \Sp(g,\Z)$ be a neat arithmetic subgroup.
Since($\mathfrak{H}_g,$ $ds^2$) is $\Sp(g,\R)$-invariant, it induces a canonical metric on the smooth Siegel variety  $\sA_{g,\Gamma}=\mathfrak{H}_g/\Gamma.$ The canonical metric is a complete K\"ahler-Einstein metric with negative Ricci curvature. The volume form $\Phi_g$ is also $\Sp(g,\R)$-invariant, and so we have an  induced volume form $\Phi_{g,\Gamma}$ on $\sA_{g,\Gamma}.$ It is known that $\Phi_{g,\Gamma}$ is singular  at the boundary divisor $D_{\infty}:=\overline{\sA}_{g,\Gamma}\setminus\sA_{g,\Gamma}$ for any smooth toroidal compactification $\overline{\sA}_{g,\Gamma}.$

\subsection{Volume forms of the Siegel space $\mathfrak{H}_g$ associated to cusps}  Associated to a cusp $\mathfrak{F}_{g-k},$  We  can write  the volume form $\Phi_g$ in the coordinate system explicitly.

Now we identify the Siegel Space $\mathfrak{H}_g$ with $\mathfrak{S}_g$ as in the proposition \ref{Borel-embedding}.
According to the embedding
$$\varphi : u^{\mathfrak{F}_{g-k}}_\C\times v^{\mathfrak{F}_{g-k}}_\R\times \mathfrak{F}_{g-k}\>\cong>> D(\mathfrak{F}_{g-k})\,\,\,(u_1+\sqrt{-1}u_2, v,F)\mapsto \exp(u_1+\sqrt{-1}u_2)\exp(v)(\check{F})$$
and the isomorphism
$ \varphi :(u^{\mathfrak{F}_{g-k}}+\sqrt{-1}C(\mathfrak{F}_{g-k}))\times v^{\mathfrak{F}_{g-k}}_\R\times \mathfrak{F}_{g-k}\>\cong>> \mathfrak{S}_g$
in the proposition \ref{embedding of the third type},we obtain that
\begin{eqnarray*}
 \mathfrak{S}_g  &=& \big\{ \left(
   \begin{array}{cccc}
     I_{g-k} & 0     & 0     & 0 \\
     0       & I_{k} & 0     & Z \\
     0       & 0     &I_{g-k}& 0 \\
     0       & 0     & 0     & I_{k}\\
   \end{array}
 \right)\cdot \left(
   \begin{array}{cccc}
     I_{g-k} & 0     & 0     & A \\
     -^tB     & I_{k} & ^tA  & \frac{^tAB-  ^tBA}{2} \\
     0       & 0     &I_{g-k}& B \\
     0       & 0     & 0     & I_{k}\\
   \end{array} \right)\check{F}\,\,\\
   &  &\,\,\,\, | \,\,Z=X+\sqrt{-1}Y\in(\mathrm{Sym}_k(\R)+\sqrt{-1}\mathrm{Sym}^{+}_k(\R)), A+\sqrt{-1}B\in M_{g-k,k}(\C), \check{F}\in \mathfrak{F}^\vee_k \big\}
\end{eqnarray*}
by the corollary \ref{Siegel domain of the third type}, where $\mathrm{Sym}^{+}_k(\R)=\{Y\in \mathrm{Sym}_k(\R)\,\,|\,\, Y>0\}.$
Suppose
$\tau^{'}\in\mathfrak{H}_{g-k} \mbox{ and }Z=X+\sqrt{-1}Y\in(\mathrm{Sym}_k(\R)+\sqrt{-1}\mathrm{Sym}^{+}_k(\R))$(=$\mathfrak{H}_k$) now.
We note that each $\check{F}\in \mathfrak{F}^\vee_k $ can be written as :
\begin{equation*}
(!)\hspace{0.5cm} \,\,  F_{\tau^{'}}=\mbox{ subspace of $V_\C$ spanned by the column vectors of }
  \left(
  \begin{array}{cc}
    \tau^{'} & 0 \\
    0 &  0_{k} \\
    I_{g-k} & 0 \\
    0 & I_k \\
  \end{array}
\right)  \,\,\mbox{ for } \tau^{'}\in \mathfrak{H}_{g-k}.
\end{equation*}
Thus, we get
\begin{eqnarray*}
\mathfrak{S}_g   &\ni& \left(
   \begin{array}{cccc}
     I_{g-k} & 0     & 0     & 0 \\
     0       & I_{k} & 0     & Z \\
     0       & 0     &I_{g-k}& 0 \\
     0       & 0     & 0     & I_{k}\\
   \end{array}
 \right)\cdot \left(
   \begin{array}{cccc}
     I_{g-k} & 0     & 0     & A \\
     -^tB     & I_{k} & ^tA  & \frac{^tAB-^tBA}{2} \\
     0       & 0     &I_{g-k}& B \\
     0       & 0     & 0     & I_{k}\\
   \end{array} \right)F_{\tau^{'}} \\
   &=& \mbox{ subspace of $V_\C$ spanned by the column vectors of }
  \left(
  \begin{array}{cc}
    \tau^{'} & A \\
    ^t(A-\tau^{'}B) &  Z+\frac{^tAB-^tBA}{2} \\
    I_{g-k} & B \\
    0 & I_k \\
  \end{array}
\right)  \\
   &=& \mbox{ subspace of $V_\C$ spanned by the column vectors of }\\
    && \,\,\, \left(
  \begin{array}{cc}
    \tau^{'} &(A-\tau^{'}B) \\
    ^t(A-\tau^{'}B) &  Z+^tB\tau^{'}B-\frac{(^tAB+^tBA)}{2} \\
    I_{g-k} & 0\\
    0 & I_k \\
  \end{array}
\right)\\
&=:& F^1_{\tau}.
\end{eqnarray*}
Thus this $F^1_{\tau}$ corresponds to a point $\tau:=\left(\begin{array}{cc}
    \tau^{'} &(A-\tau^{'}B) \\
    ^t(A-\tau^{'}B) &  Z+^tB\tau^{'}B-\frac{(^tAB+^tBA)}{2}
    \end{array}
\right)$ in $\mathfrak{H}_g$ as we describe in the proposition \ref{Borel-embedding}. We also have that
\begin{equation*}
\mathrm{Im}(\tau)=\left(\begin{array}{cc}
    \mathrm{Im}(\tau^{'}) & -\im(\tau^{'})B \\
    -^tB\im(\tau^{'}) &  \im(Z)+^tB\im(\tau^{'})B
    \end{array}
\right)
\end{equation*}
and $\det\im(\tau) =\det \im(\tau^{'})\det \im(Z).$
Write $\tau^{'}=(t_{ij}),\,\, A=(a_{ij}),\,\, B=(b_{ij}),\,\,Z=(c_{ij}),\,\,$
$S:=A+\sqrt{-1}B=(s_{ij}),$ and $ U:=A-\tau^{'}B=(u_{ij}).$ The
$((c_{ij}),(s_{ij}),(t_{i,j}))$ becomes a coordinate system of $\mathfrak{H}_g$ associated to the cusp $\mathfrak{F}_{g-k}.$  Then,
\begin{eqnarray*}
  d u_{ij} &=& da_{ij}+ d(\sum_{\alpha} t_{i\alpha}b_{\alpha j}) \\
           &=&  da_{ij}+\sum_{\alpha} t_{i\alpha}db_{\alpha j} + \mbox{ forms containing $dt$}\\
d\overline{u_{ij}} &=&  da_{ij}+\sum_{\alpha} \overline{t_{i\alpha}}db_{\alpha j} + \mbox{ forms containing $d\overline{t}$}\\
d u_{ij}\wedge d\overline{u_{ij}}&=&-2\sqrt{-1}\sum_{\alpha} \im(t_{i\alpha})da_{ij}\wedge db_{\alpha j}+\sum_{\alpha,\beta}t_{i\alpha}\overline{t_{i\beta}}db_{\alpha j}\wedge db_{\beta j}\\
&&\,\,\,\, + \mbox{ forms containing $dt$ or $d\overline{t}$}.
\end{eqnarray*}
Thus, we obtain that
\begin{eqnarray*}
   &&(\frac{\sqrt{-1}}{2})^{k(g-k)}\bigwedge_{i,j}d u_{ij}\wedge d \overline{u_{ij}}  \\
   &=& \det \im(t_{ij})(\bigwedge_{i,j}d a_{ij}\wedge db_{ij}) +\mbox{ forms containing $dt$ or $d\overline{t}$} \\
   &=&(\frac{\sqrt{-1}}{2})^{k(g-k)}\det \im(\tau^{'})(\bigwedge_{i,j}d s_{ij}\wedge d \overline{s_{ij}})+  \mbox{ forms containing $dt$ or $d\overline{t}$ }.
\end{eqnarray*}
Write $R:=Z+^tB\tau^{'}B-\frac{(^tAB+^tBA)}{2}=(r_{ij}),$ we calculate the  volume form in coordinate system $(c_{ij}, s_{ij}, t_{ij})$ of $\mathfrak{H}_g$ :
\begin{eqnarray*}
  dV_g(\tau) &=& 2^{\frac{g(g-1)}{2}}(\bigwedge_{1\leq  i\leq j\leq g-k} dt_{ij} \wedge d\overline{t_{ij}})\wedge(\bigwedge_{ij} du_{ij} \wedge d\overline{u_{ij}}) \wedge(\bigwedge_{1\leq i\leq j\leq k} dr_{ij}\wedge d\overline{r_{ij}}) \\
   &=&2^{\frac{g(g-1)}{2}}\det \im(\tau^{'})
 (\bigwedge_{1\leq  i\leq j\leq g-k} dt_{ij} \wedge d\overline{t_{ij}})\wedge(\bigwedge_{ij} ds_{ij}
  \wedge d\overline{s_{ij}}) \wedge(\bigwedge_{1\leq i\leq j\leq k} dc_{ij}\wedge d\overline{c_{ij}}).
\end{eqnarray*}
 Define
 $$d\mathrm{Vol}(\tau^{'}):=2^{\frac{(g-k)(g-k-1)}{2}}\bigwedge\limits_{1\leq  i\leq j\leq g-k} dt_{ij}\wedge d\overline{t_{ij}},\,\,d\mathrm{Vol}(Z):=2^{\frac{k(k-1)}{2}}\bigwedge\limits_{1\leq i\leq j \leq k} dc_{ij}\wedge d\overline{c_{ij}},$$ and  $d\mathrm{Vol}(S):=2^{k(g-k)}\bigwedge\limits_{1\leq i\leq g-k, 1\leq j \leq k} ds_{ij}\wedge d\overline{s_{ij}}.$
Since $d\mathrm{Vol}(\tau^{'})$ is just the standard Euclidian volume form $dV_{g-k}$ on $\mathfrak{H}_{g-k},$ we have that
\begin{eqnarray*}
  \Phi_g(\tau) &=&(\frac{\sqrt{-1}}{2})^{\frac{g(g+1)}{2}}\frac{dV_g(\tau)}{(\det \mathrm{Im}(\tau))^{g+1}} \\
   &=&(\frac{\sqrt{-1}}{2})^{\frac{g(g+1)}{2}}\frac{\det \im(\tau^{'})dV_{g-k}\wedge d\mathrm{Vol}(S) \wedge d\mathrm{Vol}(Z)}{(\det \im(\tau^{'})\det \im(Z))^{g+1}} \\
   &=&\Phi_{g-k}(\tau^{'})\wedge(\frac{\sqrt{-1}}{2})^{k(g-k)}\frac{d\mathrm{Vol}(S)}{(\det \im(\tau^{'}))^{k-1}}\wedge(\frac{\sqrt{-1}}{2})^{\frac{k(k+1)}{2}}\frac{d\mathrm{Vol}(Z)}{(\det \im(Z))^{g+1}}.\,\,
\end{eqnarray*}
\begin{proposition}\label{Calculation-volume-form-1}
Let $\mathfrak{F}=\mathfrak{F}(V^{(g-k)})$ be a $k$-th cusp of $\mathfrak{H}_g.$
\begin{myenumi}
\item The Siegel space can be written as
\begin{eqnarray*}
 \mathfrak{H}_g  &=&\big\{\tau:=\left(\begin{array}{cc}
    \tau^{'} &(A-\tau^{'}B) \\
    ^t(A-\tau^{'}B) &  Z+^tB\tau^{'}B-\frac{(^tAB+^tBA)}{2}
    \end{array}
\right)\in M_{g}(\C)\,\,\, \\
   && \,\, \mbox{   }|\,\,\, \tau^{'}=(t_{ij})\in \mathfrak{H}_{g-k},\,\, Z=(c_{ij})\in \mathfrak{H}_{k},\,\, S=(s_{ij}):=A+\sqrt{-1}B\in M_{g-k,k}(\C)\big\},
\end{eqnarray*}
and $((c_{ij}),(s_{ij}),(t_{ij}))$ becomes a coordinate system of $\mathfrak{H}_g$ associated to $\mathfrak{F}.$

\item We have the following formula of volume form :
$$\Phi_g(\tau)=\left\{
  \begin{array}{ll}
  \Phi_{g-k}(\tau^{'})\bigwedge(\frac{\sqrt{-1}}{2})^{k(g-k)}\frac{d\mathrm{Vol}(S)}{(\det \im(\tau^{'}))^{k-1}}\bigwedge(\frac{\sqrt{-1}}{2})^{\frac{k(k+1)}{2}}\frac{d\mathrm{Vol}(Z)}{(\det \im(Z))^{g+1}},   &  1\leq k<g\\
 (\frac{\sqrt{-1}}{2})^{\frac{g(g+1)}{2}}\frac{d\mathrm{Vol}(Z)}{(\det \im(Z))^{g+1}}, & k=g.
  \end{array}
\right.  \,\,
$$
\end{myenumi}
\end{proposition}

\subsection{Local volume forms of low-degree Siegel varieties}
For any two integers $1\leq i,j\leq g,$ let $E_{ij}=(a_{\alpha\beta})$ be a  $(g\times g)$-matrix of $a_{\alpha\beta}=\left\{
                                                                                          \begin{array}{ll}
                                                                                            1, & \hbox{$(\alpha,\beta)=(i,j)$;} \\
                                                                                            0, & \hbox{others.}
                                                                                          \end{array}
                                                                                        \right.
$

We compute volume forms of Siegel varieties $\sA_{g,n}$ of low genus $g$ with respect to
certain special compactifications.
Let $g=2$ or $3$ in this subsection. Let $\Sigma_{\mathrm{cent}}$ be the central cone
decomposition of $C(\mathfrak{F}_0)$ and $\sigma_{0}$ the principal cone in $\Sigma_{\mathrm{cent}}$ defined in the example
\ref{central-cone-decomposition}. We can write down the cone  $\sigma_0$ clearly(cf.\cite{Igu67},\cite{Nam}) :
      $\sigma_0=\{\sum\limits_{1\leq i\leq j\leq g} \lambda_{i,j}\zeta_{i,j} \,\,|\,\, \lambda_{i,j}\in\R_{\geq 0}\}$
      such that every  edge $\R_{\geq 0}\zeta_{i,j}$ is in
$\overline{C(\mathfrak{F}_0)}^{\mathrm{rc}}\setminus C(\mathfrak{F}_0),$ where
$
\left\{
  \begin{array}{ll}
   \zeta_{i,i}:= E_{i,i}, & 1 \leq i\leq g; \\
    \zeta_{i,j}:=-E_{i,j}-E_{j,i}+E_{i,i}+E_{j,j}, & 1\leq i<j \leq g.
  \end{array}
\right.$
Let $\overline{\sA}_{g,n}^{\mathrm{cent}}$ be the projective smooth
toroidal compactification constructed by the central cone
decomposition $\Sigma_{\mathrm{cent}}$ of
$\overline{C(\mathfrak{F}_0)}.$ The induced volume form $\Phi_{g,n}:=\Psi_{g,\Gamma(n)}$ on $\sA_{g,n}$
is singular at the boundary divisor $D_{\infty,n}:=\overline{\sA}_{g,n}^{\mathrm{cent}}\setminus\sA_{g,n}.$

We now calculate the volume form on Siegel space $\mathfrak{H}_g$ associated to the cusp $\mathfrak{F}_0.$
Let
$$
(\star) \hspace{2cm}\left\{
  \begin{array}{ll}
  \zeta_{i,i}^n:=nE_{i,i},&\hbox{ for }  1 \leq i\leq g;\\
  \zeta_{i,j}^n:=n(-E_{i,j}-E_{j,i}+E_{i,i}+E_{j,j}), & \hbox{ for } 1\leq i<j\leq g.
  \end{array}
\right.
$$
The $\{\zeta_{i,j}^n\}_{ 1\leq i\leq j\leq g}$ is a basis of
$\mathrm{Sym}_g(\R),$ and it also can  be regarded as a lattice basis of $\Gamma(n)\cap U^{\mathfrak{F}_0}(\Q).$
We note that $\R_{\geq 0}\zeta_{i,j}^n\, 1\leq i\leq j \leq g$ are all edges of the principal cone $\sigma_0$ in $\Sigma_{\mathrm{cent}}.$
Since the $Z=(c_{ij})_{1\leq i,j\leq g}$ in the proposition \ref{Calculation-volume-form-1} can be written as
$Z=\sum\limits_{1\leq i\leq j\leq g}z_{ij}\zeta_{i,j}^n,$
we get that
\begin{eqnarray*}
  Z&=&\sum_{1\leq i< j\leq g}z_{ij}\zeta_{i,j}^n+ \sum_{j=1}^gz_{jj}\zeta_{j,j}^n \\
   &=&n(\sum_{1\leq i< j\leq g}z_{ij}(-E_{i,j}-E_{j,i})+\sum_{j=1}^{g}(z_{jj}+\sum_{l=1}^{j-1}z_{lj}+\sum_{l=j+1}^gz_{jl})E_{j,j}).
\end{eqnarray*}

On the other hand, $Z= \sum\limits_{1\leq i<j\leq g}c_{ij}(E_{i,j}+E_{j,i})+\sum\limits_{l=1}^gc_{ll}E_{l,l}.$
Thus, we have that
$$
\left\{
  \begin{array}{ll}
    c_{ij}=c_{ji}=-nz_{ij}, & \hbox{ for } 1\leq i<j\leq g; \\
    c_{jj}=n(z_{jj}+\sum\limits_{l=1}^{j-1}z_{lj}+\sum\limits_{l=j+1}^gz_{jl})=:n(z_{jj}+m_j), & \hbox{ for } 1\leq j\leq g.
  \end{array}
\right.
$$
Let $\widetilde{Z}:= Z/n.$ We calculate the volume form $\Phi_g$ in the coordinate system $(z_{ij})$ :
$$dZ:=\bigwedge_{1\leq i\leq j \leq g} dc_{ij}= \pm n^{g(g+1)/2} \bigwedge_{1\leq i\leq j \leq g} dz_{ij},$$
\begin{equation}\label{volume in Siegel coordinates}
 \Phi_g(\tau)=(\frac{\sqrt{-1}}{2})^{\frac{g(g+1)}{2}} \frac{2^{\frac{g(g-1)}{2}}\bigwedge\limits_{1\leq i\leq j \leq g} dz_{ij}\wedge  d\overline{z_{ij}}}{(\det \im(\widetilde{Z}))^{g+1}}.
\end{equation}
where
$$\im(\widetilde{Z})=\im\left(
                       \begin{array}{ccccc}
                         z_{11}+m_1 &\cdots   &  -z_{1j} & \cdots &  -z_{1g}\\
                         \vdots     & \ddots  &  \vdots  & \vdots &  \vdots \\
                         -z_{1j}    & \cdots  &z_{jj}+m_j& \cdots &  -z_{jg}  \\
                          \vdots    &  \vdots &  \vdots  & \ddots &   \vdots \\
                          -z_{1g}   & \cdots  & -z_{jg}  & \cdots &   z_{gg}+m_g\\
                       \end{array}
                     \right)_{g\times g}.
$$
Recall the example \ref{local-coordinate-system},
we have a commutative diagram
$$
 \begin{CDS}
\mathfrak{H}_g \>\subset>>U^{\mathfrak{F}_0}(\C)\cong\C^{g(g+1)/2} \\
\V V  V \novarr \V V w_{ij}:=\exp(2\pi\sqrt{-1}z_{ij}) V  \\
\mathfrak{H}_{g}/\Gamma\cap U^{\mathfrak{F}}(\Q)\>\subset >>
(\C^*)^{g(g+1)/2}
\end{CDS}
$$
and the partial compactification
 \begin{eqnarray*}
   \widetilde{\Delta}_{\mathfrak{F}_0,\sigma_0}&=&\big\{ x=(w_{ij})_{1\leq i\leq j\leq g}\in \C^{g(g+1)/2} \,\,|\,\, \mbox{ there exists a neighborhood } \\
   &&\,\,\Delta_x  \mbox{ of $x$ such that } \Delta_x \cap(\C^{*})^{g(g+1)/2} \subset
 \frac{\mathfrak{H}_{g}}{\Gamma(l)\cap U^{\mathfrak{F}_0}(\Q)}\big\}.
 \end{eqnarray*}
Define $
       \widetilde{\Delta}_{\mathfrak{F}_0, \sigma_0}^*:=  \widetilde{\Delta}_{\mathfrak{F}_0,\sigma_0}-\bigcup\limits_{1\leq i\leq j\leq g} \big\{(w_{ij})\in \widetilde{\Delta}_{\mathfrak{F}_0,\sigma_0}
        \,\,|\,\, w_{ij}=0\big\}.
       $
The  volume on $\widetilde{\Delta}_{\mathfrak{F}_0,\sigma_0}^*$ becomes
\begin{equation}\label{local-volume}
 \Phi_{\sigma_0}(w_{ij})=(\frac{\sqrt{-1}}{2})^{\frac{g(g+1)}{2}}\frac{2^{\frac{g(g-1)}{2}}\bigwedge\limits_{1\leq i\leq j \leq g} dw_{ij}\wedge  d\overline{w_{ij}}}{(\prod\limits_{1\leq i\leq j \leq g}|w_{ij}|^2)(\det\log |W|)^{g+1}}
\end{equation}
where
\begin{eqnarray*}
   \log|W|&:=& \left(
                       \begin{array}{ccccc}
                         \log|w_{11}|+q_1 &\cdots   &-\log|w_{1j}|   & \cdots &-\log|w_{1g}|\\
                         \vdots           & \ddots  &  \vdots        & \vdots &  \vdots \\
                         -\log|w_{1j}|    & \cdots  &\log|w_{jj}|+q_j& \cdots &-\log|w_{jg}|  \\
                          \vdots          &  \vdots &  \vdots        & \ddots & \vdots \\
                         -\log|w_{1g}|    & \cdots  &-\log|w_{jg}|   & \cdots &\log|w_{gg}|+q_g\\
                       \end{array}
                     \right)_{g\times g},\\
  q_{j} &:=&-2\pi\im(m_j)=\sum_{l=1}^{j-1}\log|w_{lj}|+\sum_{l=j+1}^g \log |w_{jl}|   \,\,\,1\leq j\leq g.
\end{eqnarray*}
For genus $g=2,$ Wang has already obtained the volume form of this type in \cite{Wang93}.
However, we should be careful that if $g\leq 3$ then $(w_{ij})$ can not be a local coordinate system of $\overline{\sA}_{g,n}^{\mathrm{cent}}$
with respect to the central cone decomposition $\Sigma_{\mathrm{cent}}$ of $C(\mathfrak{F}_0)$ as we point out in the example \ref{central-cone-decomposition}.

\subsection{Global volume forms on Siegel varieties $\sA_{g,\Gamma}$}
Let $\Sigma_{\mathfrak{F}_0}:=\{\sigma_\alpha^{\mathfrak{F}_0}\}$ be a
 $\overline{\Gamma_{\mathfrak{F}_0}}$(or $\mathrm{GL}(g,\Z)$)-admissible polyhedral decomposition of $C(\mathfrak{F}_0)$ regular with respect to an arithmetic subgroup $\Gamma\subset\Sp(g,\Z).$
Let $\sigma_{\max}$ be a top-dimensional cone in $\Sigma_{\mathfrak{F}_0}.$ The cone $\sigma_{\max}$ is generated by a lattice basis of $U^{\mathfrak{F}_0}\cap\Gamma,$ and then we can endows a marking order on this lattice basis. We have $\sigma_{\max}=\sum\limits_{\mu=1}^{\frac{g(g+1)}{2}} \R_{\geq 0}\overrightarrow{l(\mu)}$
where $\{\overrightarrow{l(\mu)}\}_{\mu=1}^{\frac{g(g+1)}{2}}$ is the marking basis of $U^{\mathfrak{F}_0}\cap\Gamma.$
We write $\overrightarrow{l(\mu)}=\sum\limits_{1\leq i\leq j\leq g} l^{\mu}_{i,j} \delta_{i,j}$ for $\mu=1,\cdots, \frac{g(g+1)}{2},$
where
$\{\delta_{i,j}\}_{_{ 1\leq i\leq j\leq g}}$  is a $\Z$-basis of $\mathrm{Sym}_g(\Z)$ given by
$\left\{
  \begin{array}{ll}
  \delta_{i,i}=E_{i,i},& 1 \leq i\leq g;\\
  \delta_{i,j}=(E_{i,j}+E_{j,i}) & 1\leq i<j\leq g.
  \end{array}
\right.
$
Then, we have a ($\frac{g(g+1)}{2} \times \frac{g(g+1)}{2}$) integral matrix
\begin{equation}\label{matrix-associate to cone}
   L_{\Gamma}(\sigma_{\max},\{\overrightarrow{l(\mu)}\}):=\left(
                                              \begin{array}{ccccc}
                                               l^{1}_{1,1}& l^{1}_{1,2}&\cdots & l^{1}_{g-1,g}& l^{1}_{g,g} \\
                                                \vdots & \cdots &\cdots &\cdots& \vdots \\
                                             l^{\frac{g(g+1)}{2}}_{1,1} & l^{\frac{g(g+1)}{2}}_{1,2}&\ldots & l^{\frac{g(g+1)}{2}}_{g-1,g}& l^{\frac{g(g+1)}{2}}_{g,g}  \\
                                              \end{array}
                                            \right).
\end{equation}
We define lattice  volume  of the top-dimensional cone $\sigma_{\max}$ to be
\begin{equation}\label{volume-number-cone}
\mathrm{vol}_{\Gamma}(\sigma_{\max}):= |\det(L_{\Gamma}(\sigma_{\max}), \{\overrightarrow{l(\mu)}\})|,
\end{equation}
which is a positive integer independent of the marking order of the basis $\{\overrightarrow{l(\mu)}\}_{\mu=1}^{\frac{g(g+1)}{2}}.$

\begin{theorem}\label{global-volume-form-on Siegel varieties}
Let $\Gamma\subset \Sp(g,\Z)$ be a neat arithmetic subgroup and $\Sigma_{\mathfrak{F}_0}$ a
$\overline{\Gamma_{\mathfrak{F}_0}}$(or $\mathrm{GL}(g,\Z)$)-admissible polyhedral decomposition of $C(\mathfrak{F}_0)$ regular with respect to $\Gamma.$
Let $\overline{\sA}_{g,\Gamma}$ be the toroidal compactification  of $\sA_{g,\Gamma}:=\mathfrak{H}_{g}/\Gamma$ constructed by  $\Sigma_{\mathfrak{F}_0}.$

Assume that the boundary divisor $D_{\infty}:=\overline{\sA}_{g,\Gamma}\setminus\sA_{g,\Gamma}$ is  simple normal crossing.
For each irreducible component $D_i$ of $D_\infty=\bigcup\limits_{i}D_i,$ let $s_i$ be the global section of the line bundle $[D_i]$ defining $D_i.$
Let $\sigma_{\max}$ be an arbitrary top-dimensional cone in $\Sigma_{\mathfrak{F}_0}$ and renumber all components $D_i$'s of $D_\infty$ such that $D_1,\cdots, D_{\frac{g(g+1}{2})}$ corresponds to the edges of $\sigma_{\max}$ with marking order.
 \begin{myenumi}
\item The  volume $\Phi_{g,\Gamma}$ on $\sA_{g,\Gamma}$ can be represented by
 \begin{equation}\label{volume-form}
 \Phi_{g,\Gamma}=\frac{2^{\frac{g(g-1)}{2}}\mathrm{vol}_{\Gamma}(\sigma_{\max})^2d\sV_g}{(\prod_{j=1}^{\frac{g(g+1)}{2}} ||s_i||_i^2)F^{g+1}_{\sigma_{\max}}(\log||s_1||_1,\cdots, \log||s_{\frac{g(g+1)}{2}}||_{\frac{g(g+1)}{2}})},
 \end{equation}
where $d\sV_g$ is a continuous volume form on a partial compactification $\mathcal{U}_{\sigma_{\max}}\subset \overline{\sA}_{g,\Gamma}$ of $\sA_{g,\Gamma},$ each $||\cdot||_{i}$ is a suitable Hermitian metric of the line bundle $[D_i]$  on $\overline{\sA}_{g,\Gamma}$($1\leq i \leq g(g+1)/2$) and $F_{\sigma_{\max}}\in \Z[x_1,\cdots, x_{g(g+1)/2}]$
is a homogenous polynomial of degree $g.$ Moreover, the
 coefficients of $F_{\sigma_{\max}}$ only depend on both $\Gamma$ and $\sigma_{\max}$ with marking order of edges.

\item Moreover, the polynomial $F_{\sigma_{\max}}(x_1,\cdots, x_{\frac{g(g+1)}{2}})$  satisfies the following equation
\begin{eqnarray*}
  && \det\big( F_{\sigma_{\max}}(\frac{\partial^2 F_{\sigma_{\max}}}{\partial x_i\partial x_j})_{i,j}-\left(
                               \begin{array}{c}
                                 \frac{\partial F_{\sigma_{\max}}}{\partial x_1} \\
                                 \vdots \\
                                \frac{\partial F_{\sigma_{\max}}}{\partial x_{\frac{g(g+1)}{2}}} \\
                               \end{array}
                             \right)
 \left(
\begin{array}{ccccc}
\frac{\partial F_{\sigma_{\max}}}{\partial x_1}, & \cdots,&\frac{\partial F_{\sigma_{\max}}}{\partial x_{\frac{g(g+1)}{2}}} \\
\end{array}
\right)\big) \\
  &=&  (-1)^{\frac{g(g+1)}{2}}2^{\frac{g(g-1)}{2}}\mathrm{vol}_{\Gamma}(\sigma_{\max})^2F^{(g+1)(g-1)}_{\sigma_{\max}}.
\end{eqnarray*}
\end{myenumi}
\end{theorem}
\begin{remark}
Let $H_{\sigma_{\max}}(x_1,\cdots, x_{\frac{g(g+1)}{2}})=-\log F_{\sigma_{\max}}(-x_1,\cdots, -x_{\frac{g(g+1)}{2}}).$ The equation in the statement (3) becomes
\begin{equation}\label{Monge-Ampere-eq}
   \det(\frac{\partial^2 H_{\sigma_{\max}}}{\partial x_i\partial x_j})_{i,j}=2^{\frac{g(g-1)}{2}}\mathrm{vol}_{\Gamma}(\sigma_{\max})^2 \exp ((g+1)H_{\sigma_{\max}}).
\end{equation}
By the formula \ref{dd-volume-equality} in the next section, it is a real Monge-Amp\'ere of elliptic type on the domain
$\{(x_1,\cdots, x_{\frac{g(g+1)}{2}})\in \R^{\frac{g(g+1)}{2}}\,\,|\,\,x_i \geq C \forall i\}$ for some positive number $C.$
\end{remark}
\begin{proof}[Proof of the theorem \ref{global-volume-form-on Siegel varieties}]
Let $N=\frac{g(g+1)}{2}.$ Let $\mathfrak{F}_{\min}$ be an arbitrary minimal cusp of $\mathfrak{H}_g,$ and let
$\sigma$ be any  top-dimensional cone in the decomposition $\Sigma_{\mathfrak{F}_{\min}}$ induced by
$\Sigma_{\mathfrak{F}_0}$(cf.Lemma \ref{family-decomposion from a minimal boundary}).
Recall the local chart $(\widetilde{\Delta}_{\mathfrak{F}_{\min},\sigma},(w_1^\sigma, \cdots, w_N^\sigma))$ in \ref{neighborhood-toiroidal-embedding} and \ref{local-coordinate-system},
we have
$$
 \begin{CDS}
\frac{\mathfrak{H}_g}{\Gamma\cap U^{\mathfrak{F}_{\min}}(\Q)} \>\subset>> \widetilde{\Delta}_{\mathfrak{F}_{\min},\sigma}\\
\V V  V \novarr \V   \pi_{\mathfrak{F}_{\min}}^{'} V\mbox{\'etale}V  \\
 \mathfrak{H}_{g}/\Gamma\>\subset >> \overline{\sA}_{g,\Gamma}
\end{CDS}
$$
with a toroidal  embedding $\frac{\mathfrak{H}_g}{\Gamma\cap U^{\mathfrak{F}_{\min}}(\Q)}\>\subset >>
\widetilde{\Delta}_{\mathfrak{F}_{\min},\sigma}$(cf.Lemma \ref{toroidal-embedding-partial-compactification}).
There are  facts :
\begin{myenumiii}
\item The morphism  $\pi_{\mathfrak{F}_{\min}}^{'}: \frac{\mathfrak{H}_g}{\Gamma\cap U^{\mathfrak{F}_{\min}}(\Q)}\to \mathfrak{H}_{g}/\Gamma$ is surjective.
\item Define $W_{\mathfrak{F}_{\min},\sigma}:= \pi_{\mathfrak{F}_{\min}}^{'}( \widetilde{\Delta}_{\mathfrak{F}_{\min},\sigma}).$
Since $\overline{\sA}_{g,\Gamma}$ is geometrically fine,  the restriction map $\pi_{\mathfrak{F}_{\min}}^{'}|_{\{ w_i^\sigma=0\}}$ is an isomorphism
onto its image for each $w_i^\sigma.$
   Thus,  $(W_{\mathfrak{F}_{\min},\sigma},(w_1^\sigma,\cdots, w_N^\sigma))$ becomes a coordinate neighborhood of $\overline{\sA}_{g,\Gamma}.$

\item That $W_{\mathfrak{F}_{\min},\sigma}^*= W_{\mathfrak{F}_{\min},\sigma}\setminus D_\infty=\mathfrak{H}_{g}/\Gamma$ where $W_{\mathfrak{F}_{\min},\sigma}^*:= W_{\mathfrak{F}_{\min},\sigma}\setminus\bigcup\limits_{i=1}^N\{ w_i^\sigma=0 \}.$

\item  The compactification $\overline{\sA}_{g,\Gamma}$ is covered by finitely many open sets of the form $W_{\mathfrak{F},\delta},$ where $\mathfrak{F}$ is a minimal cusp of $\mathfrak{H}_g$ and $\delta$ is a top-dimensional cone in the decomposition $\Sigma_{\mathfrak{F}}.$\\
\end{myenumiii}

Now, we begin to prove the statements $1-3$ :

Let $\sigma_{\max}$  be an arbitrary top-dimensional cone in the decomposition $\Sigma_{\mathfrak{F}_0}.$
We take a coordinate chart $(W_{\mathfrak{F}_{0},\sigma_{\max}}^{*},(w_1,\cdots, w_N))$ on $\sA_{g,\Gamma}$ constructed by $\sigma_{\max}$ as above such that $D_i\cap W_{\mathfrak{F}_{0},\sigma_{\max}}=\{w_i=0\}$ for any integer $i\in [1,\frac{g(g+1)}{2}].$
\begin{myenumi}
\item By Theorem 4.1 in \cite{Wang93} or by similar calculations as in \ref{local-volume}, the volume form $\Phi_{g,\Gamma}$ on the chart $(W_{\mathfrak{F}_{0},\sigma_{\max}}^{*},(w_1,\cdots, w_N))$ can be written as
\begin{equation}\label{Wang-volume-form}
    \Phi_{\sigma_{\max}}=\frac{(\frac{\sqrt{-1}}{2})^N2^{\frac{g(g-1)}{2}}\mathrm{vol}_{\Gamma}(\sigma_{\max})^2\bigwedge\limits_{1\leq i \leq N} dw_{i}\wedge d\overline{w_{i}}}{(\prod\limits_{1\leq i \leq N}|w_{i}|^2)(F_{\sigma_{\max}}(\log|w_1|,\cdots,\log|w_N|))^{g+1}}
\end{equation}
where $F_{\sigma_{\max}}\in \Z[x_1,\cdots, x_{N}]$ is a homogenous polynomial of degree $g.$ It is obvious that the coefficients of $F_{\sigma_{\max}}$  only depend on $\Gamma$ and $\sigma_{\max}$ with marking order of edges.

Let $\mathcal{U}_{\sigma_{\max}}:=W_{\mathfrak{F}_{0},\sigma_{\max}}-\bigcup_{i\neq j}D_i\cap D_j.$
The $\mathcal{U}_{\sigma_{\max}}$ is a partial compactification of $\sA_{g,\Gamma}$ satisfying  that $\sA_{g,\Gamma}\subset \mathcal{U}_{\sigma_{\max}}\subset \overline{\sA}_{g,\Gamma}.$ We can choose a Hermitian metrics $||\cdot||_{i}$ of  line bundle $[D_i]$ on $\overline{\sA}_{g,n}$ by setting
$$||s_i||_{i}^2=\rho_i|w_i|^2\,\,  \mbox{ on }\,\,\mathcal{U}_{\sigma_{\max}}$$
for $1\leq i\leq N$ such that $u:=\frac{F_{\sigma_{\max}}(\log||s_1||_1,\cdots, \log||s_{N}||_{N})}{F_{\sigma_{\max}}(\log|w_1|,\cdots, \log|w_N|_{N})}$ is a positive function on $\sA_{g,n}$ by the Lemma \ref{lemma-on-degree-of-polynomial}. Again by the Lemma \ref{lemma-on-degree-of-polynomial},
we can  choose the following continuous volume form $d\sV_g$ on $\mathcal{U}_{\sigma_{\max}}$ given by
$$d\sV_g=(\frac{\sqrt{-1}}{2})^N (u^{g+1} \prod_{i=1}^N\rho_i)\bigwedge\limits_{1\leq i \leq N}dw_{i}\wedge d\overline{w_{i}}$$ which is smooth on $\sA_{g,n}.$
Then, we obtain that the form $\Phi_{g,\Gamma}$ on $\sA_{g,\Gamma}$  can be represented by
 $$\Phi_{g,\Gamma}=\frac{2^{\frac{g(g-1)}{2}}\mathrm{vol}_{\Gamma}(\sigma_{\max})^2d\sV_g}{(\prod\limits_{j=1}^{N} ||s_i||_i^2)F^{g+1}_{\sigma_{\max}}(\log||s_1||_1,\cdots, \log||s_{N}||_{N})}.$$

\item With respect to the coordinate chart $(W_{\mathfrak{F}_{0},\sigma_{\max}}^{*},(w_1,\cdots, w_N))$ on $\sA_{g,\Gamma},$
we define $G(w_1,\cdots, w_N):=F_{\sigma_{\max}}(\log|w_1|,\cdots, \log|w_N|).$ By the K\"ahler-Einstein metric, we get that
\begin{eqnarray*}
  &&(\frac{-g-1}{2})^N2^{-N}N! 2^{\frac{g(g-1)}{2}}\mathrm{vol}_{\Gamma}(\sigma_{\max})^2\bigwedge\limits_{1\leq i \leq N} dw_{i}\wedge d\overline{w_{i}} \\
  &=& G^{g+1}( \prod_{j=1}^{N} |w_i|^2)(\partial\overline{\partial}\log G^{g+1})^N\\
  &=&(g+1)^NG^{g+1}(\prod_{j=1}^{N} |w_i|^2)(\frac{\partial\overline{\partial}G }{G}- \frac{\partial G\wedge \overline{\partial}G}{G^2})^N \\
  &=&(g+1)^NG^{g+1}(\prod_{j=1}^{N} |w_i|^2)(\sum_{i,j}\frac{GG_{w_i\overline{w_j}}-G_{w_i}G_{\overline{w_j}}}{G^2}dw_i\wedge d\overline{w_j} )^N.
\end{eqnarray*}
Let $F_\alpha:=\frac{\partial F_{\sigma_{\max}}(x_1,\cdots,x_N)}{\partial x_\alpha}$ and $F_{\alpha\beta}:=\frac{\partial^2 F_{\sigma_{\max}}(x_1,\cdots,x_N)}{\partial x_\alpha \partial x_\beta}$ for all $1\leq \alpha,\beta\leq N.$
On the Siegel variety $\sA_{g,\Gamma},$ we have :
\begin{eqnarray*}
 G_{w_i}(w_1,\cdots, w_N)              &=& \frac{F_i(\log|w_1|,\cdots, \log|w_N|)}{2w_i}, \\
 G_{\overline{w_i}}(w_1,\cdots, w_N)   &=&  \frac{F_i(\log|w_1|,\cdots, \log|w_N|)}{2\overline{w_i}},\\
 G_{w_i\overline{w_j}}(w_1,\cdots, w_N)&=&  \frac{F_{ij}(\log|w_1|,\cdots, \log|w_N|)}{4w_i\overline{w_j}},
\end{eqnarray*}
$$\big(\det(F_{\sigma_{\max}}F_{ij}-F_iF_j)_{i,j}-(-1)^{N}2^{\frac{g(g-1)}{2}}\mathrm{vol}_{\Gamma}(\sigma_{\max})^2F^{(g+1)(g-1)}_{\sigma_{\max}}\big)(\log|w_1|,\cdots, \log|w_N|)=0.$$
We also have
$$
(\frac{\partial^2 \log F_{\sigma_{\max}}}{\partial x_i\partial x_j})_{i,j}
  =(F_{\sigma_{\max}}F_{ij}-F_iF_j)_{i,j}
= F_{\sigma_{\max}}(F_{ij})_{i,j}-\left(
                               \begin{array}{c}
                                 F_1 \\
                                 \vdots \\
                                F_{N} \\
                               \end{array}
                             \right)\left(
                                      \begin{array}{ccccc}
                                        F_1, & \cdots,& F_{N} \\
                                      \end{array}
                                    \right).
$$
\end{myenumi}
\end{proof}
\begin{example}Let $\Gamma=\Gamma(n)$ for some fixed integer $n\geq 3.$ We consider volume forms on the genus two Siegel variety $\sA_{2,n}$ with same notations in the theorem \ref{global-volume-form-on Siegel varieties}.
Let $\sigma_{\max}=\{\sum\limits_{i=1}^3 \lambda_i\overrightarrow{l_i} | \lambda_i\in \R_{\geq 0},\, \,\, i=1,\cdots, 3\}
$ be a top-dimensional cone in $\Sigma_{\mathfrak{F}_0},$ where  $\overrightarrow{l_1}, \overrightarrow{l_2},\overrightarrow{l_3}$ satisfies that
\begin{itemize}
  \item  $\overrightarrow{l_i}=n\big(a_{i1}\left(
                       \begin{array}{cc}
                         1 & 0 \\
                         0 & 0 \\
                       \end{array}
                     \right)
  + a_{i2}\left(
                       \begin{array}{cc}
                         0 & 1 \\
                         1 & 0 \\
                       \end{array}
                     \right)
                     + a_{i3}\left(
                       \begin{array}{cc}
                         0 & 0 \\
                         0 & 1 \\
                       \end{array}
                     \right)\big)=n\left(
                                     \begin{array}{cc}
                                       a_{i1} & a_{i2} \\
                                       a_{i2} & a_{i3} \\
                                     \end{array}
                                   \right)
                     $ for each $i,$
  \item all $a_{ij}$ are integers and $D:=\sqrt[3]{\mathrm{vol}_{\Gamma}(\sigma_{\max})^2}=\sqrt[3]{(\det(a_{ij})_{3\times3})^{2}}>0.$
\end{itemize}

Each $Z\in\mathfrak{H}_g=\mathrm{Sym}_k(\R)+\sqrt{-1}\mathrm{Sym}^{+}_k(\R)$ can be written as
\begin{eqnarray*}
  Z &=& \sum_{i=1}^3 z_i\overrightarrow{l_i} \\
    &=& n\left(
           \begin{array}{cc}
            z_1a_{11}+ z_2a_{21}+z_3a_{31} & z_1a_{12}+z_2a_{22}+z_3a_{32}\\
            z_1a_{12}+z_2a_{22}+z_3a_{32} & z_1a_{13}+z_2a_{23}+z_3a_{33}\\
           \end{array}
         \right).
\end{eqnarray*}
Then, the symplectic volume becomes
$$ \Phi_2=(\frac{\sqrt{-1}}{2})^{3} \frac{2D^3dz_1\wedge dz_2\wedge dz_3\wedge d\overline{z_1}\wedge d\overline{z_2}\wedge d\overline{z_3}}{(\det \im(\widetilde{Z}))^{3}}$$
where $\widetilde{Z}=\frac{1}{n}Z= \left(
           \begin{array}{cc}
            z_1a_{11}+ z_2a_{21}+z_3a_{31} & z_1a_{12}+z_2a_{22}+z_3a_{32}\\
            z_1a_{12}+z_2a_{22}+z_3a_{32} & z_1a_{13}+z_2a_{23}+z_3a_{33}\\
           \end{array}
         \right).$
For each integer $i$ in $[1,3],$ let  $D_i$ be the smooth divisor on $\overline{\sA}_{2,l}$ constructed by  $\overrightarrow{l_i}$ and let
$s_i$ be the global section of $[D_i]$ defining $D_i.$
Then,  the symplectic volume form $\Phi_{2,l}$ on $\sA_{2,l}$ can be represented by
 $$\Phi_{2}=\frac{2\mathrm{vol}_{\Gamma}(\sigma_{\max})^2d\sV_g}{(||s_1||_1||s_2||_2||s_3||_3)^2F^{3}_{\sigma_{\max}}(\log||s_1||_1,\log||s_2||_2, \log||s_{3}||_{3})}.$$
Here,the polynomial $F_{\sigma_{\max}}$ is
$F_{\sigma_{\max}}(x,y,z)=Ax^2+By^2+Cz^2+Lxy+Mxz+Nyz$
where $A=a_{11}a_{13}-a_{12}^2, B=a_{21}a_{23}-a_{22}^2,C=a_{31}a_{33}-a_{32}^2,$
and $L=a_{11}a_{23}+a_{21}a_{13} -2a_{12}a_{22},$ $ M=a_{11}a_{33}+a_{13}a_{31} -2a_{12}a_{32}, N=a_{21}a_{33}+a_{23}a_{31} -2a_{22}a_{32}.$
Define variables $y_1:=x,y_2:=y, y_3:=z.$ and let $F_i=\frac{\partial F_{\sigma_{\max}}}{\partial y_i},F_{ij}=\frac{\partial^2 F_{\sigma_{\max}}}{\partial y_i\partial y_j}\forall i,j.$ Let
$$G:=\det(F_{\sigma_{\max}}F_{ij}-F_iF_j)_{1\leq i,j\leq 3}-(-1)^{3}2(DF_{\sigma_{\max}})^3.$$
The coefficient of the term $y_1^6$ in $G$ is $A^3P(a_{ij}),$
where $P\in \Q[(x_{ij})_{3\times3}]$ is a homogenous polynomial such that
$P(a_{ij})= 2BM^2+2CL^2+2AN^2-8ABC-2LMN +2D^3.$
 In the polynomial $P,$ the coefficient of the term $x_{12}^2x_{22}^2x_{33}^2$ is  $-2.$ Thus $P$ is a nonzero homogenous polynomial of degree $6.$
\end{example}

\subsection{Constrained combinational conditions of decompositions of cones}

Let $N=g(g+1)/2,$ and let $S_N$ be the group of permutations of the set $\{1,\cdots, N\}.$

For any integer $i\in [1, N],$
let $Y(i)$ be a  $g\times g$ symmetric matrix $\left(
                                                 \begin{array}{ccc}
                                                   y_{11}(i)& \cdots & y_{1g}(i) \\
                                                  \vdots& \cdots & \vdots \\
                                                   y_{1g}(i)  & \cdots & y_{gg}(i) \\
                                                 \end{array}
                                               \right)
$ with $N$-variables. Each $Y(i)$ can be identified with a $1\times N$ matrix as
$$\widetilde{Y}(i):=(y_{11}(i), \cdots,y_{1g}(i), y_{22}(i), \cdots,y_{2g}(i),\cdots, y_{jj}(i),\cdots, y_{jg}(i),\cdots, y_{gg}(i)).$$
Define $Y:=\left(
              \begin{array}{c}
                \widetilde{Y}(1) \\
                \vdots \\
                \widetilde{Y}(i)\\
                \vdots \\
                \widetilde{Y}(N) \\
              \end{array}
            \right).
$
We know that $Y$ is a $N\times N$ matrix with $N^2$-variables. Define $D(Y):=\det(Y),$ it is a homogenous polynomial of degree $N$ in
$\Z[Y]:=\Z[(y_{kl}(i))_{1\leq k\leq l\leq g, 1\leq i\leq N}].$
For any $\varsigma\in S_N,$ we define
$\varsigma(Y):=\left(
              \begin{array}{c}
                \widetilde{Y}(\varsigma(1)) \\
                \vdots \\
                \widetilde{ Y}(\varsigma(i))\\
                \vdots \\
                \widetilde{Y}(\varsigma(N)) \\
              \end{array}
            \right).\\
$

We begin to show that there is a characteristic variety $\mathfrak{Q}_g$ by the unique group-invariant K\"ahler-Einstein metric on $\mathfrak{H}_g.$
Define $F=\det(\sum_{i=1}^Nx_iY(i)).$
We have $$F=\sum_{i_1+\cdots +i_N=g, i_k\geq0 } t_{i_1\cdots i_N}(Y)x_1^{i_1}\cdots x_{N}^{i_N}$$
and each $t_{i_1\cdots i_g}(Y)\in\Z[Y]$ is a homogenous polynomial of degree $g.$
Let
\begin{eqnarray*}
  \mathcal{C}_1&:=& \det\big(F(\frac{\partial^2 F}{\partial x_i\partial x_i})_{i,j}-\left(
                               \begin{array}{c}
                                 \frac{\partial F}{\partial x_1} \\
                                 \vdots \\
                                \frac{\partial F}{\partial x_{N}} \\
                               \end{array}
                             \right)
 \left(
\begin{array}{ccccc}
\frac{\partial F}{\partial x_1}, & \cdots,&\frac{\partial F}{\partial x_{N}} \\
\end{array}
\right)\big), \\
\mathcal{C}_2 &:=&(-1)^{\frac{g(g+1)}{2}}2^{\frac{g(g-1)}{2}}F^{(g+1)(g-1)}D(Y)^2.
\end{eqnarray*}
We then  write $\mathcal{C}_1:=\mathcal{C}_1-\mathcal{C}_2$ as
\begin{equation}\label{defining-equation}
 \mathcal{C}=\sum_{j_1+\cdots +j_N=g(g^2-1), j_k\geq 0 } C_{j_1\cdots j_N}(Y)x_1^{j_1}\cdots x_{N}^{j_N}
\end{equation}
such that each $C_{j_1\cdots j_N}(Y)\in \Q[Y]$ is a homogenous polynomial of degree $g^2(g+1).$
\begin{lemma}\label{symmetric-relation-polynomials}
For any tuple $(j_1,\cdots, j_N)$ of non-negative integers with $\sum\limits_{\alpha=1}^Nj_\alpha=g(g^2-1),$
let $C_{j_1\cdots j_N}(Y)\in \Q[Y]$ be homogenous polynomials defined in \ref{defining-equation}. We have
$$C_{j_1\cdots j_N}(\varsigma(Y))=C_{j_{\varsigma^{-1}(1)}\cdots j_{\varsigma^{-1}(N)}}(Y) \,\,\forall \varsigma\in S_N.$$
\end{lemma}
\begin{proof}Let $\varsigma$ be an arbitrary element in the group $S_N.$ We get
$$\det(\sum\limits_{i=1}^Nx_iY(\varsigma(i)))=\sum\limits_{i_1+\cdots +i_N=g, i_k\geq0 } t_{i_1\cdots i_N}(\varsigma(Y))x_1^{i_1}\cdots x_{N}^{i_N}.$$
On the other hand,
\begin{eqnarray*}
  \det(\sum_{i=1}^Nx_iY(\varsigma(i))) &=& \det(\sum_{i=1}^Nx_{\varsigma^{-1}(i)}Y(i))\\
                                       &=&  \sum_{i_1+\cdots +i_N=g, i_k\geq0 } t_{i_1\cdots i_N}(Y)x_{\varsigma^{-1}(1)}^{i_1}\cdots x_{\varsigma^{-1}(N)}^{i_N}\\
                                       &=&\sum_{i_{\varsigma(1)}+\cdots +i_{\varsigma(N)}=g, i_k\geq0 } t_{i_{1}\cdots i_{N}}(Y)x_{1}^{i_{\varsigma(1)}}\cdots x_{N}^{i_{\varsigma(N)}}\\
                                       &=&\sum_{j_{1}+\cdots +j_{N}=g, j_k\geq0 } t_{j_{\varsigma^{-1}(1)}\cdots j_{\varsigma^{-1}(N)}}(Y)x_{1}^{j_{1}}\cdots x_{N}^{j_{N}}.
\end{eqnarray*}
Therefore, we obtain
$t_{i_1\cdots i_N}(\varsigma(Y))=t_{i_{\varsigma^{-1}(1)}\cdots i_{\varsigma^{-1}(N)}}(Y)\,\,\,\, \forall \, i_1+\cdots +i_N=g $ with  $i_k\geq0.$
Since $D(\varsigma(Y))^2\equiv D(Y)^2,$ we prove the statement as well.\\
\end{proof}

The group $S_N$ has a natural action on the set of $(N\times N)$-matrices $M_{N\times N}(\C)$  as
$$\varsigma(B):=\left(
              \begin{array}{c}
                \widetilde{B}(\varsigma(1)) \\
                \vdots \\
                \widetilde{B}(\varsigma(i))\\
                \vdots \\
                \widetilde{B}(\varsigma(N)) \\
              \end{array}
            \right)\, \mbox{ for }\varsigma\in S_N, B=( \widetilde{B}(1)^T,\cdots  \widetilde{B}(i)^T,\cdots, \widetilde{B}(N)^T)\in M_{N\times N}(\C).
$$
Since the $S_N$ acts freely on
 $\mathrm{GL}(g,\C)=\{A\in M_{N\times N}(\C)\,\,|\,\,  \det A\neq 0 \},$
 the quotient $\mathfrak{P}_g:=  \mathrm{GL}(g,\C)/S_N$ is a smooth affine variety.
\begin{lemma}\label{matrix-associate to cone-lemma}
Let $\Sigma_{\mathfrak{F}_0}:=\{\sigma_\alpha^{\mathfrak{F}_0}\}$ be a
$\overline{\Gamma_{\mathfrak{F}_0}}$(or $\mathrm{GL}(g,\Z)$)-admissible polyhedral decomposition of $C(\mathfrak{F}_0)$ regular with respect to an arithmetic subgroup $\Gamma\subset \Sp(g,\Z).$
There is an injective map of sets
$\nu_{\Gamma}:\{\mbox{top-dimensional cones in }\Sigma_{\mathfrak{F}_0}\}\> \hookrightarrow>> \mathfrak{P}_g(\Z),$
where $\mathfrak{P}_g(\Z)$ is the set of all integral points of the variety $\mathfrak{P}_g.$
\end{lemma}
\begin{proof}
Let $\sigma_{\max}$ be an arbitrary top-dimensional cone in $\Sigma_{\mathfrak{F}_0}.$
Let $\{\overrightarrow{l(\mu)}\}_{\mu=1}^{N}$ be any marking basis of $U^{\mathfrak{F}_0}(\Z)\cap \Gamma$ such that $\sigma_{\max}=\sum\limits_{\mu=1}^{N} \R_{\geq 0}\overrightarrow{l(\mu)}.$
The $L_{\Gamma}(\sigma_{\max},\{\overrightarrow{l(\mu)}\})$ in \ref{matrix-associate to cone} is an element in
$\mathrm{GL}(g,\Z),$ and its projective image $[L_{\Gamma}(\sigma_{\max},\{\overrightarrow{l(\mu)}\})]$ in $\mathrm{GL}(g,\Z)/S_N$ is independent of the marking order of the basis $\{\overrightarrow{l(\mu)}\}_{\mu=1}^{N}.$
Therefore we can define an  injective map
$$\nu_{\Gamma}:\{\mbox{top-dimensional cones in }\Sigma_{\mathfrak{F}_0}\}\>>> \mathfrak{P}_g(\Z)\,\,\, $$ by sending $\sigma_{\max}$ to
the equivalent class of $L_{\Gamma}(\sigma_{\max},\{\overrightarrow{l(\mu)}\})$ in $\mathfrak{P}_g(\Z).$
\end{proof}
\begin{lemma}
Define $$\mathfrak{A}_g=\{Z\in \mathrm{GL}(g,\C)\,\,|\,\,  C_{j_1\cdots j_N}(Z)=0\,\,\forall j_1+\cdots +j_N=g(g^2-1)\mbox{ with }j_k\geq 0\},$$
where $C_{j_1\cdots j_N}$'s are polynomials defined in \ref{defining-equation}.
The permutation group $S_N$ acts freely on the affine variety $\mathfrak{A}_g.$
\end{lemma}
\begin{proof}
By the lemma \ref{symmetric-relation-polynomials}, we obtain that if $Z\in \mathfrak{A}_g$ then $\varsigma(Z)\in \mathfrak{A}_g$ for all $\varsigma\in S_N.$ Thus the group $S_N$ has a free action on $\mathfrak{A}_g.$
\end{proof}

Define
\begin{equation}\label{characteristic-variety}
  \mathfrak{Q}_g:=\mathfrak{A}_g/S_N.
\end{equation}
It is obvious that $\mathfrak{Q}_g$ is an affine variety defined over $\Q$ dependent only on $\mathfrak{H}_g.$
We call $\mathfrak{Q}_g$ the $g$-\textbf{KE-characteristic variety}.

\begin{theorem}\label{combinatorial-condition-1}
Let $\Gamma\subset \Sp(g,\Z)$ be a neat arithmetic subgroup. Let $\Sigma_{\mathfrak{F}_0}:=\{\sigma_\alpha^{\mathfrak{F}_0}\}$ be a
$\overline{\Gamma_{\mathfrak{F}_0}}$(or $\mathrm{GL}(g,\Z)$)-admissible polyhedral decomposition of $C(\mathfrak{F}_0)$ regular with respect to $\Gamma,$ where $\mathfrak{F}_0$ is the standard minimal cusp of the Siegel space $\mathfrak{H}_g.$
Let $\overline{\sA}_{g,\Gamma}$ be the toroidal compactification  of $\sA_{g,\Gamma}:=\mathfrak{H}_{g}/\Gamma$ constructed by  $\Sigma_{\mathfrak{F}_0}.$

Assume that the boundary divisor $D_{\infty}:=\overline{\sA}_{g,\Gamma}\setminus\sA_{g,\Gamma}$ is
simple normal crossing.
There is an injective map of sets
$$\nu_{\Gamma}:\{\mbox{top-dimensional cones in }\Sigma_{\mathfrak{F}_0}\}\> \hookrightarrow>> \mathfrak{Q}_g(\Z),$$
where  $\mathfrak{Q}_g(\Z)$ is the set of all integral points of the $g$-KE-characteristic variety $\mathfrak{Q}_g.$
\end{theorem}
\begin{proof}
It is straightforward by the theorem \ref{global-volume-form-on Siegel varieties} and the lemma \ref{matrix-associate to cone-lemma}.
\end{proof}
\begin{remark}\label{Remark-mumford-volume}
Actually, the assumption of normal crossing $D_{\infty}:=\overline{\sA}_{g,\Gamma}\setminus\sA_{g,\Gamma}$ in the theorem \ref{combinatorial-condition-1} is
not necessary. The theorem \ref{combinatorial-condition-1} is true for all smooth toroidal compactifications.
Consider the partial compactification given by the diagram
$$
 \begin{CDS}
\mathfrak{H}_g \>\subset>>U^{\mathfrak{F}_0}(\C)\cong\C^{\frac{g(g+1)}{2}} \\
\V V  V \novarr \V V w_{i}:=\exp(2\pi\sqrt{-1}z_{i}) V  \\
\mathfrak{H}_{g}/\Gamma\cap U^{\mathfrak{F}}(\Q)\>\subset >>
(\C^*)^{g(g+1)/2}
\end{CDS}
$$
with respect to an arbitrary regular top-dimensional cone $\sigma\in \Sigma_{\mathfrak{F}_0},$
the $(w_1,\cdots, w_{\frac{g(g+1)}{2}})$ is always a local coordinate system of the partial compactification
even though it can not be  a local coordinate system of
$\overline{\sA}_{g,\Gamma},$ the quotient manifold $\mathfrak{H}_{g}/\Gamma\cap U^{\mathfrak{F}}(\Q)$ also has an induced K\"ahler-Einstein metric with
 volume form \ref{Wang-volume-form}.
Therefore, the function $H_{\sigma}(x_1,\cdots, x_{\frac{g(g+1)}{2}}):=-\log F_{\sigma}(-x_1,\cdots, -x_{\frac{g(g+1)}{2}})$ must satisfy the elliptic real  Monge-Amp\'ere equation \ref{Monge-Ampere-eq}.
\end{remark}



\section{Asymptotic behaviours of logarithmical canonical line bundles}

\vspace{0.5cm}

Let $N=g(g+1)/2.$ For any positive integer $n,$ we define a constant $C_n=(\frac{\sqrt{-1}}{2\pi})^n.$

In this section, we fix a neat subgroup $\Gamma\subset \Sp(g,\Z)$ and a
$\overline{\Gamma_{\mathfrak{F}_0}}$(or $\mathrm{GL}(g,\Z)$)-admissible polyhedral decomposition $\Sigma_{\mathfrak{F}_0}:=\{\sigma_\alpha^{\mathfrak{F}_0}\}$ of $C(\mathfrak{F}_0)$ regular with respect to $\Gamma$ such that the constructed symmetric toroidal compactification $\overline{\sA}_{g,\Gamma}$ of $\sA_{g,\Gamma}:=\mathfrak{H}_{g}/\Gamma$ is geometrically $\Gamma$-fine, i.e., $D_\infty:=\overline{\sA}_{g,\Gamma}\setminus\sA_{g,\Gamma}$ is a simple normal crossing divisor.
Let $K_{\sA_{g,\Gamma}}$ be the canonical divisor on $\sA_{g,\Gamma}$ and $h_B$ the metric  on the canonical line bundle $\sO_{\sA_{g,\Gamma}}(K_{\sA_{g,\Gamma}})$ induced by the Bergman metric $\omega_{\mathrm{can}}$ of $\sA_{g,\Gamma}.$

We define
$D_\infty(\epsilon)$ to be tube neighborhood of $D_\infty$ with radius $\epsilon$ for suitable real number $\epsilon>0.$
For every irreducible component $Y$ of $D_\infty,$ we define
$$Y_{\infty}:=\bigcup\limits_{D_j\neq Y}(D_j\cap Y)\,\,\mbox{ and }\,\,Y^*:=Y\setminus Y_{\infty}.$$
Then $Y_{\infty}$ is a simple normal crossing divisor of $Y.$
In the theorem \ref{global-volume-form-on Siegel varieties},  we show that there is a system $\{(U_{\alpha},(w_1^{\alpha},\cdots, w_N^{\alpha}))\}_{\alpha}$ of finitely many coordinate charts of the compactification $\overline{\sA}_{g,\Gamma}$
such that
\begin{equation}\label{local-chart-toroidal-compactification}
    U_{\alpha}^*:=U_{\alpha}\setminus D_\infty=\sA_{g,\Gamma}\,\, \mbox{  and  } \,\, U_{\alpha}\cap D_\infty=\bigcup_{i=1}^N\{ w_i=0 \}.
\end{equation}
On any  such coordinate chart $(U_{\alpha}^*,(w_1^{\alpha},\cdots, w_N^{\alpha})),$ the volume form $\Phi_{g,\Gamma}$ becomes
$$
 \Phi_\alpha=\frac{(\frac{\sqrt{-1}}{2})^N2^{\frac{g(g-1)}{2}}\mathrm{vol}_{\Gamma}(\sigma_{\max})^2\bigwedge\limits_{1\leq i \leq N} dw_{i}^\alpha\wedge d\overline{w_{i}^\alpha}}{(\prod_{1\leq i \leq N}|w_{i}^{\alpha}|^2)(F^{\alpha}(\log|w_1^{\alpha}|,\cdots,\log|w_N^{\alpha}|))^{g+1}}.
$$
where $F^{\alpha}\in \R[x_1,\cdots, x_{N}]$ is a homogenous polynomial in of degree $g.$
We call this $F^{\alpha}$ the \textbf{local volume function} with respect to the local chart $U_{\alpha}^*.$
Write $$F^{\alpha}_i:=\frac{\partial F^{\alpha}(x_1,\cdots,x_N)}{\partial x_i},F^{\alpha}_{ij}:=\frac{\partial^2 F^{\alpha}(x_1,\cdots,x_N)}{\partial x_i \partial x_j}\,\, 1\leq i,j\leq N.$$
Define $$T^{\alpha}_{i,j}:=F^{\alpha}F^{\alpha}_{ij}-F^{\alpha}_{i}F^{\alpha}_{j}\,\,\, 1\leq i,j\leq N.$$
We have a $N\times N$ matrix $T^\alpha:=(T^{\alpha}_{i,j})$ such that each  $T^{\alpha}_{i,j}$ is a homogenous polynomial of degree $2g-2$ in $\R[x_1,\cdots, x_{N}].$

Now we begin to compute $\partial\overline{\partial}\log \Phi_\alpha$  as a distribution form on $U_\alpha$ :

\begin{eqnarray*}
 && \partial\overline{\partial}\log \Phi_\alpha\\
    &=& -\partial\overline{\partial}\log\prod_{1\leq i \leq N}|w_{i}^{\alpha}|^2-(g+1)\partial\overline{\partial}\log F^{\alpha}(\log|w_1^\alpha|,\cdots,\log|w_N^\alpha|)\\
   &=&- \partial\overline{\partial}\log\prod_{1\leq i \leq N}|w_{i}^{\alpha}|^2\\
   &&-(g+1)\{\frac{\partial\overline{\partial}F^{\alpha}(\log|w_1^\alpha|,\cdots,\log|w_N^\alpha|)}{F^{\alpha}(\log|w_1^\alpha|,\cdots,\log|w_N^\alpha|)} -\frac{(\partial F^{\alpha}\wedge \overline{\partial} F^{\alpha})(\log|w_1^\alpha|,\cdots,\log|w_N^\alpha|)}{F^{\alpha}(\log|w_1^\alpha|,\cdots,\log|w_N^\alpha|)^2}\}\\
   &=& -\sum_{i=1}^N(2+(g+1)\frac{F^{\alpha}_i}{F^{\alpha}}(\log|w_1^\alpha|,\cdots,\log|w_N^\alpha|))
   \partial\overline{\partial}\log|w_i^\alpha|  \\
   &&
   -(g+1)\sum_{1\leq i,j\leq N}\frac{T^{\alpha}_{i,j}}{(F^{\alpha})^2}(\log|w_1^\alpha|,\cdots,\log|w_N^\alpha|)\partial\log |w_i^\alpha|\wedge \overline{\partial}\log |w_j^\alpha|.\\
\end{eqnarray*}
Particularly,
\begin{equation}\label{dd-volume-equality}
    \partial\overline{\partial}\log \Phi_\alpha=\sum_{1\leq i,j\leq N}K_{ij}(w_1^{\alpha},\cdots,w_N^\alpha)dw_i^{\alpha}\wedge d\overline{w_j^\alpha}\,\,\, \mbox{     on } U^*_\alpha.
\end{equation}
is a smooth form on $ U^*_\alpha,$ where $K=(K_{i,j})$ is a new $N\times N$ matrix given by
\begin{equation*}
    K_{i,j}(w_1^{\alpha},\cdots,w_N^\alpha):=
    \frac{-(g+1)}{4}\frac{ T^{\alpha}_{i,j}(\log|w_1^\alpha|,\cdots,\log|w_N^\alpha|)}{w_i^\alpha\overline{w_j^{\alpha}}(F^{\alpha})^2(\log|w_1^\alpha|,\cdots,\log|w_N^\alpha|)}\,\, \forall i,j.
\end{equation*}

In general, with respect to an arbitrary smooth toroidal compactification $\overline{D/\Gamma}$ of any locally symmetric manifold $D/\Gamma$
with normal crossing boundary divisor, Mumford has shown
that any group-invariant Hermitian metric on the homogeneous holomorphic cotangent bundle $\Omega_{D/\Gamma}^1$ is good on $\overline{D/\Gamma}$ such that
the good extension of $\Omega_{D/\Gamma}^1$ to  $\overline{D/\Gamma}$ is just the logarithmical cotangent bundle on  $\overline{D/\Gamma}$
(cf. Main Theorem 3.1 and Proposition 3.4  in Section 1 of \cite{Mum77}).
\begin{lemma}\label{Poincare-growth-metric}For any positive integer $p,$
the smooth $(p,p)$-form $(\partial\overline{\partial}\log \Phi_{g,\Gamma})^p$ on $U_\alpha^*$ has Poincar\'e growth on $D_\infty\cap U_{\alpha}$(cf.\cite{Mum77} for the definition), and $(\sqrt{-1}\partial\overline{\partial}\log \Phi_{g,\Gamma})^p$ is a positive current on $\overline{\sA}_{g,\Gamma}.$
\end{lemma}
\begin{proof}
The Main Theorem 3.1 and Proposition 3.4 in \cite{Mum77} guarantee that $[K_{\overline{\sA}_{g,\Gamma}}+D_\infty]$
is the unique extension of $\sO_{\sA_{g,\Gamma}}(K_{\sA_{g,\Gamma}})$ to $\overline{\sA}_{g,\Gamma}$ such that $h_B$ is singular metric good on $\overline{\sA}_{g,\Gamma}.$ Thus $c_1(\sO_{\sA_{g,\Gamma}}(K_{\sA_{g,\Gamma}}),h_B)$ has Poincar\'e growth on $D_\infty.$
By the K\"ahler-Einstein equality \ref{KE}, we get that
$$(\frac{g+1}{2}\omega_{\mathrm{can}})^p =(\sqrt{-1}\partial\overline{\partial}\log \Phi_{g,\Gamma})^p
    =(2\pi c_1(\sO_{\sA_{g,\Gamma}}(K_{\sA_{g,\Gamma}}),h_B))^p.$$
\end{proof}
We can make an improvement on the above lemma : By the generalized Schwarz lemma(cf.\cite{Yau78-2},\cite{CCL79} and \cite{Roy80}), the lemma is true for not only smooth toroidal compactifications but also a general compactification with normal crossings boundary divisor.

\begin{proposition}\label{key-lemma-on-recurrence}
Let $\Gamma\subset \Sp(g,\Z)$ be a neat arithmetic subgroup and $\Sigma_{\mathfrak{F}_0}$ a
$\overline{\Gamma_{\mathfrak{F}_0}}$(or $\mathrm{GL}(g,\Z)$)-admissible polyhedral decomposition of $C(\mathfrak{F}_0)$ regular with respect to $\Gamma.$
Let $\overline{\sA}_{g,\Gamma}$ be a toroidal compactification  of $\sA_{g,\Gamma}:=\mathfrak{H}_{g}/\Gamma$ constructed by  $\Sigma_{\mathfrak{F}_0}.$

Assume that the boundary divisor $D_{\infty}:=\overline{\sA}_{g,\Gamma}\setminus\sA_{g,\Gamma}$ is  simple normal crossing. Let $\Phi_{g,\Gamma}$ be the standard volume form on $\sA_{g,\Gamma}.$
Let $p$ be a positive integer in $[1, \frac{g(g+1)}{2}]$ and let $\Psi:=(\partial\overline{\partial}\log\Phi_{g,\Gamma})^p.$
Write $D_\infty=\bigcup\limits_{j=1}^l D_j,$ we have :
\begin{myenumi}
\item Regard $\Psi$ as a singular form $\widetilde{\Psi}$ on the compactification $\overline{\sA}_{g,n},$
we define the restriction of $\Psi$ to $D_i$ as follows in sense of limit :
\begin{equation}\label{restrict-rule}
 \mathrm{Res}_{D_i}((\partial\overline{\partial}\log\Phi_{g,\Gamma})^p):=\widetilde{\Psi}|_{D_i^*}
\end{equation}
for each irreducible component $D_i$ of $D_\infty.$ For each $D_i,$ the form $\mathrm{Res}_{D_i}((\partial\overline{\partial}\log\Phi_{g,\Gamma})^p)$ becomes a smooth form on each $D_i.$

\item For each irreducible component $D_i$ of $D_\infty,$ the form $\mathrm{Res}_{D_i}((\partial\overline{\partial}\log\Phi_{g,\Gamma})^p)$ has Poincar\'e growth on the simple normal crossing divisor $D_{i,\infty}$ of $D_i.$
 That $\mathrm{Res}_{D_i}((\partial\overline{\partial}\log\Phi_{g,\Gamma})^p)$ becomes a current on $D_i$ in sense that the following integral
\begin{equation}\label{Poincare-growth-integral}
  \int_{D_i}\mathrm{Res}_{D_i}((\partial\overline{\partial}\log\Phi_{g,\Gamma})^p)\wedge \alpha:=\lim_{\varepsilon\to 0} \int_{D_i\setminus T_i(\varepsilon)} \mathrm{Res}_{D_i}((\partial\overline{\partial}\log\Phi_{g,\Gamma})^p)\wedge \alpha\,\,\,\, \,
\end{equation}
is finite for each smooth $(g(g+1)-2p-2)$-form  $\alpha$ on $D_i,$ where $T_i(\varepsilon)$ is a tube neighborhood of $D_{i,\infty}$ with radius $\varepsilon.$
\item For each irreducible component $D_i$ of $D_\infty,$ the form $\mathrm{Res}_{D_i}((\partial\overline{\partial}\log\Phi_{g,\Gamma})^p)$ is closed on $D_i^*$ and $(\frac{\sqrt{-1}}{2\pi})^p\mathrm{Res}_{D_i}((\partial\overline{\partial}\log\Phi_{g,\Gamma})^p)$ is a positive closed current on $D_i.$
\end{myenumi}
\end{proposition}
\begin{proof}
\begin{myenumi}
\item The statement (1) is a consequence of (4) of the lemma \ref{lemma-on-degree-of-polynomial} and  the lemma \ref{degree-determine-polynomal} in the next subsection.
\item By symmetry, we prove the statement (2) for case $D_1$ only.  It is a local problem and it is sufficient to prove this statement for $p=1.$

Taking a local chart $(U_\alpha,(w^\alpha_1,w^\alpha_2,\cdots, w^\alpha_N))$ of $\overline{\sA}_{g,\Gamma}$ as in \ref{local-chart-toroidal-compactification}, we have the smooth form
$$
 \partial\overline{\partial}\log \Phi_\alpha
   =-(g+1)\sum_{1\leq i,j\leq N} \frac{T^{\alpha}_{i,j}}{(F^{\alpha})^2}(\log|w_1^\alpha|,\cdots,\log|w_N^\alpha|)\frac{dw_i^{\alpha}\wedge d\overline{w_j^\alpha}}{4w_i^\alpha\overline{w_j^\alpha}}\,\,\, \mbox{     on } U^*_\alpha.\\
$$
Thus, on $D_1^*\cap U_\alpha,$ the $\mathrm{Res}_{D_1}(\partial\overline{\partial}\log\Phi_{g,\Gamma})$ can be written
as
\begin{eqnarray*}
  \mathrm{Res}_{D_1}(\partial\overline{\partial}\log\Phi_{g,\Gamma})
   &=&  -(g+1)\sum_{2\leq i,j\leq N} \frac{T^{\alpha}_{i,j}}{(F^{\alpha})^2}(\log|w_1^\alpha|,\cdots,\log|w_N^\alpha|)|_{D_1^*\cap U_\alpha}\frac{dw_i^{\alpha}\wedge d\overline{w_j^\alpha}}{4w_i^\alpha\overline{w_j^\alpha}}\\
&=:& \sum_{2\leq i,j\leq N} a_{i,j}dw_i^{\alpha}\wedge d\overline{w_j^\alpha}.
\end{eqnarray*}

Let $V_1\subset D_1$ be a small neighborhood in containing the origin point and $U_1$ a small neighborhood in $U_\alpha$ such that $U_1\cap D_1=V_1.$
Let  $i,j$ be two arbitrary integers with $2\leq i,j\leq N.$
Since $\partial\overline{\partial}\log\Phi_{g,\Gamma}$ has Pincar\'e growth on $U_\alpha \cap D_\infty$
we have
$$|\frac{T^{\alpha}_{i,j}}{(F^{\alpha})^2}(\log|w_1^\alpha|,\cdots,\log|w_N^\alpha|)
\frac{1}{4w_i^\alpha\overline{w_j^\alpha}}|\leq \frac{C}{|w_i^\alpha w_j^\alpha||\log |w_i^\alpha|\log |w_j^\alpha||}\,\, \mbox{ on } \,\,U_1\cap U^*_\alpha $$
for a suitable constant. Let $w_1^\alpha \to 0,$  we get
$$|a_{ij}|\leq \frac{C}{|w_i^\alpha w_j^\alpha||\log |w_i^\alpha|\log |w_j^\alpha||}\,\, \mbox{ on } \,\,V_1\cap D_1^*\cap U^*_\alpha. $$
Therefore,
$\mathrm{Res}_{D_1}(\partial\overline{\partial}\log\Phi_{g,\Gamma})$ has Poincar\'e growth on $D_{1,\infty}\cap U_\alpha,$
and the integral \ref{Poincare-growth-integral} is finite.

\item It is sufficient to  prove the statement(3) in case of $D_1$ for $p=1.$ It is also a local problem.
Take a local chart $U_\alpha$ of $\overline{\sA}_{g,\Gamma}.$
Let $V$ be an open neighborhood in $D_1^*\cap U_\alpha.$ For a sufficiently small $\varepsilon,$
we define a sub-complex manifold  in $D^*$
$$V_\varepsilon:=\{(\varepsilon, w^\alpha_2,\cdots, w^\alpha_N)\,\,|\,\,(0, w^\alpha_2,\cdots, w^\alpha_N)
\in V \}.$$

On $V_\varepsilon,$ we have
\begin{eqnarray*}
&& d(\widetilde{\Psi}|_{V_\varepsilon})(\varepsilon, w^\alpha_2,\cdots, w^\alpha_N) \\
&=& d(\Psi|_{V_\varepsilon})(\varepsilon, w^\alpha_2,\cdots, w^\alpha_N) \\
   &=& \sum_{2\leq k,i,j\leq N}Q_{k,i,j}(\log|\varepsilon|, \log |w^\alpha_2|,\cdots,\log |w^\alpha_N|)\frac{(\overline{w^\alpha_k} d w^\alpha_k+w^\alpha_k d\overline{ w^\alpha_k})\wedge
dw^\alpha_i\wedge d\overline{w^\alpha_j}}{|w^\alpha_k|^2w^\alpha_i\overline{w^\alpha_j}},
\end{eqnarray*}
where $Q_{k,i,j}(x_1,x_2,\cdots, x_N)$'s are rational functions.

On $V,$ we also have
\begin{eqnarray*}
&&d\mathrm{Res}_{D_1}(\partial\overline{\partial}\log\Phi_{g,\Gamma})(w^\alpha_2,\cdots, w^\alpha_N)\\
&=& \sum_{2\leq k,i,j\leq N}P_{k,i,j}(\log |w^\alpha_2|,\cdots,\log |w^\alpha_N|)\frac{(\overline{w^\alpha_k} d w^\alpha_k+w^\alpha_k d\overline{ w^\alpha_k})\wedge
dw^\alpha_i\wedge d\overline{w^\alpha_j}}{|w^\alpha_k|^2w^\alpha_i\overline{w^\alpha_j}},
\end{eqnarray*}
where $P_{k,i,j}(x_2,\cdots, x_N)$'s are rational functions.

Let $(k,i,j)$ be an arbitrary triple  with $2\leq k,i,j\leq N.$
Let $(z^\alpha_2,\cdots,z^\alpha_N)$ be an arbitrary point on $V.$
By directly calculating, we get
$$P_{k,i,j}(\log |z^\alpha_2|,\cdots,\log |z^\alpha_N|)= \lim_{\varepsilon\to 0} Q_{k,i,j}(\log|\varepsilon|,\log |z^\alpha_2|,\cdots,\log |z^\alpha_N|).$$
On the other hand,
$$d(\Psi|_{V_\varepsilon})=(d\Psi)|_{V_\varepsilon}\equiv 0|_{V_\varepsilon}\equiv0.$$
Thus  $ Q_{k,i,j}(\log|\varepsilon|,\log |z^\alpha_2|,\cdots,\log |z^\alpha_N|)=0$
for any sufficiently small $\varepsilon.$
Therefore, $$P_{k,i,j}(\log |w^\alpha_2|,\cdots,\log |w^\alpha_N|)\equiv 0  \,\,\, \mbox{ on } V$$
and so $d\mathrm{Res}_{D_1}(\partial\overline{\partial}\log\Phi_{g,\Gamma})=0.$
\end{myenumi}
\end{proof}
\subsection{Some lemmas on local volume functions}
For any polynomial  $T$ in $\R[x_1,\cdots, x_{N}],$  let $\deg_iT$ be the degree of $T$ with respect to $x_i.$
For example, we write $T=a_l x_1^l+a_{l-1}x_1^{l-1}+\cdots + a_0 \,\,\mbox{ with } a_l\neq 0$
where each $a_i$ is a  polynomial in $\R[x_2,\cdots, x_{N}],$ then  $\deg_1(T)=l.$

\begin{lemma}\label{degree-determine-polynomal}
Let $F$ be an arbitrary local volume function. We have :
\begin{myenumi}
\item That $$\deg_{k}T^{\alpha}_{i,j}\left\{
  \begin{array}{ll}
   =2\deg_k F^{\alpha}-2, & \hbox{$i=j=k$} \\
   \leq 2\deg_k F^{\alpha}-1, & \hbox{$i=k,\, j\neq k$ or $i\neq k,\, j= k$} \\
    \leq 2\deg_k F^{\alpha}, & \hbox{$i\neq k,\,j\neq k$}
  \end{array}
\right.
$$
and
$$\deg_k\det(T^{\alpha}_{i,j})_{1\leq i,j\leq N}\leq 2N\deg_k F^{\alpha}-2.$$

\item That $\deg_i F^{\alpha} \geq 1$ for  all $i=1,\cdots N.$
\end{myenumi}
\end{lemma}
\begin{proof}
The (1) and (2) are obvious. We just prove the (3).
Otherwise, $\deg_k F^{\alpha} =0$ for some $k.$ Then $T_{k,j}^\alpha$ and $T_{j,k}^\alpha$ are zero polynomials for all integers $j\in[1,N],$ so that $\det(T^{\alpha}_{i,j})_{1\leq i,j\leq N}$ is a zero polynomials. But on $U^*_\alpha,$
$$\det(T^{\alpha}_{i,j})(\log|w_1^\alpha|,\cdots,\log|w_N^\alpha|)=
(-1)^{\frac{g(g+1)}{2}}(F^{\alpha})^{(g+1)(g-1)}(\log|w_1^\alpha|,\cdots,\log|w_N^\alpha|)\neq 0$$
by the theorem \ref{global-volume-form-on Siegel varieties}. It is a contradiction.
\end{proof}

For convenience, we now allow any function to take $\pm\infty$ value.
Let $$\nu_0=-\log 0^+:=\lim\limits_{r\to 0+}\log\frac{1}{r}.$$
We have the following reasonable definitions and rules :
\begin{eqnarray*}
  0\times\nu_0:=0, &&(\nu_0)^0=1, \\
  \alpha\nu_0:=\lim\limits_{r\to 0+}\log(\frac{1}{r})^{\alpha} \,\, (\alpha\in\R),  && \alpha\nu_0+\beta\nu_0:=(\alpha+\beta)\nu_0 \,\, (\alpha,\beta\in\R),\\
 (\nu_0)^{\alpha}:=\lim\limits_{r\to 0+}(\log\frac{1}{r})^\alpha \,\,(\alpha\in \R_{>0}),&&(\nu_0)^{-\alpha}:=\frac{1}{(\nu_0)^{\alpha}} \,\, (\alpha\in \R_{>0}), \\
  \nu_0^{\alpha}\times \nu_0^{\beta}:=\nu_0^{\alpha+\beta}\,\, (\alpha,\beta\in\R).&& 
\end{eqnarray*}
In this paper, the addition of $\alpha\nu_0^a$ and $\beta \nu_0^b$ is formally written as $\alpha\nu_0^a+\beta \nu_0^b$ for any two nonzero different real  numbers $a, b$(we particularly prohibit to use the rule that $\alpha\nu_0^a+\beta \nu_0^b=\alpha\nu_0^a$ if $a>b$);
we set the rule of  multiplication by $(\sum\limits_{i=1}^n\alpha_i\nu_0^{a_i})(\sum\limits_{j=1}^m\beta_j\nu_0^{b_j}):=\sum\limits_{1\leq i\leq n,1\leq j\leq m}\alpha_i\beta_j\nu_0^{a_i+b_j}.$
We also allow that any coefficient of a matrix take $+\infty$ value. A $n\times n$ real matrix $M$ is said to be a \textbf{$\infty$-positive} if
$M$ can be written as $M=M_1+(\nu_0-c)M_2 $
for some positive $n\times n$ matrix $M_1,$ some semi-positive $n\times n$ matrix $M_2$ and some non-negative real number $c.$ Certainly,
$M_1,M_2$ and $c$ are not unique for any $\infty$-positive matrix  $M.$
\begin{lemma}
The determine of a $n\times n$ $\infty$-positive matrix $M=M_1+(\nu_0-c)M_2$ is never zero. Moreover,
$\det(M)=\sum\limits_{i=0}^{\rank(M_2)}c_i\nu_0^i $ with some finite number  $c_{\rank(M_2)}>0.$
\end{lemma}
\begin{lemma}\label{lemma-on-degree-of-polynomial}
Let $n,g$ be  positive integers such that $n\geq g(g+1)/2.$
Let $B$ be an  open set in $(\C^n,(z_1,\cdots, z_n))$ containing  $(0,\cdots,0)$ and let
$B^*:=B\setminus\bigcup\limits_{i=1}^n B_i$ where $B_{i}:=\{z_i=0\}.$

Let $M(x_1,\cdots, x_n)$ be a $g\times g$ \textbf{logarithmical positive matrix function} on $B^*$ with $n$ real variables $x_1,\cdots, x_n,$ i.e.,
$M(x_1,\cdots, x_n)=\sum\limits_{i=1}^n x_jE_j$ for  $n$ nonzero semi-positive symmetric real $g\times g$ matrices $E_1,\cdots,E_{n}$ and
$M(-\log|w_1|,\cdots, -\log|w_n|)$ is a positive matrix at every point $(w_1,\cdots, w_n)$ in $B^*.$

Let $S(x_1,\cdots, x_n):=\det(M(x_1,\cdots, x_n)).$ For any integer $i\in [1,n],$
we have :
\begin{myenumi}
\item $M(-\log|w_1|,\cdots, -\log|w_n|)$ is a $\infty$-positive matrix at any point  $(w_1,\cdots, w_n)\in B.$

\item $\deg_iS= \rank E_i,$
and
$$S(x_1,\cdots, x_n)=S_i(x_1,\cdots \widehat{x_{i}}\cdots, x_n)x_i^{\deg_i S}+\mbox{ terms with lower degree of $x_i$ }$$
where $S_i(x_1,\cdots \widehat{x_{i}}\cdots, x_n)$ is a homogenous polynomial of degree $(g-\deg_i S)$ with $n-1$ variables
$x_1,\cdots, x_{i-1},x_{i+1},\cdots, x_n.$

\item There exists a  positive number $\alpha$ and a $(g-\rank E_i)\times(g-\rank E_i)$ logarithmical positive matrix function
$M_i(x_1,\cdots \widehat{x_{i}}\cdots, x_n)$
on $B^*_{i}:=B_i
\setminus \bigcup\limits_{m\neq i}(B_i\cap B_m)$ such that
$$S_i(x_1,\cdots \widehat{x_{i}}\cdots, x_n)=\alpha\det(M^{(i)}(x_1,\cdots \widehat{x_{i}}\cdots, x_n)).$$
Moreover, $M^{(i)}(-\log|w_1|,\cdots \widehat{-\log|w_i|}\cdots, -\log|w_n|)$ is a $\infty$-positive matrix at any point  $(w_1,\cdots, w_n)\in B.$
In particular, $S_i(-\log|w_1|,\cdots \widehat{-\log|w_i|}\cdots, -\log|w_n|)$ is never zero at any point  $(w_1,\cdots, w_n)\in B.$

\item Let $i\in[1,n]$ be an integer and let $B_i^*:= B_i\setminus\bigcup_{j\neq i}B_j.$
Let $Q\in \R[x_1,\cdots, x_n]$ be a homogenous polynomial with $\deg_{i}Q\leq \deg_iS$ such that
$$Q=Q_i(x_1,\cdots \widehat{x_{i}}\cdots, x_n)x_i^{\deg_i Q}+\mbox{ terms with lower degree of $x_i$ }.$$
where $Q_i(x_1,\cdots \widehat{x_{i}}\cdots, x_n)$ is a homogenous polynomial with $n-1$ variables
$x_1,\cdots, x_{i-1},$ $x_{i+1},\cdots, x_n.$
Define a function  $A(z_1,\cdots, z_n):=(Q/S)(-\log|z_1|,\cdots, -\log|z_n|)$ on $B^*.$
\begin{myenumii}
\item That $J(w_1,\cdots, \underbrace{0}_{i},\cdots,w_n):=\lim\limits_{t_i\to 0}(Q/S)(-\log|w_1|,\cdots, -\log|t_i|, \cdots,-\log|w_n|)$
exists as a finite real number for any point $(w_1,\cdots, \underbrace{0}_{i},\cdots,w_n)\in B_i^*;$
\item the function $A$ can be extended to a continuous function $\widetilde{A}$ on $B^*\cup B_i^*,$ where
\begin{eqnarray*}
   && \widetilde{A}(w_1,\cdots, w_n ) \\
   &:=&\left\{
    \begin{array}{ll}
      A(w_1,\cdots, w_n) & \hbox{$(w_1,\cdots, w_n )\in B^*,$} \\
 J(w_1,\cdots,w_{i-1}, 0,w_{i+1},\cdots,w_n) & \hbox{$(w_1,\cdots,w_{i-1}, 0,w_{i+1},\cdots,w_n)\in B_i^*;$}
    \end{array}
  \right.
\end{eqnarray*}

\item $\mathrm{Res}_i(A):=\widetilde{A}|_{B_i^*}$ is a smooth function on $B_i^*.$
\end{myenumii}

\end{myenumi}
\end{lemma}
\begin{proof}
\begin{myenumi}
\item It is sufficient to  show that $M(-\log|w_1|,\cdots, -\log|w_n|)$ is a $\infty$-positive matrix at any point $(w_1,\cdots, w_n)\in \bigcup\limits_{i=1}^{n} \{w_i=0\}.$

Let $\Lambda$ be a subset of $\{1,\cdots, n\}.$

Let $(w_1,\cdots, w_n)$ be a point in $\bigcup\limits_{i=1}^{n} \{w_i=0\}$ such that
$\left\{
        \begin{array}{ll}
          w_i=0, & i\in \Lambda \\
          w_i\neq 0, & i\notin \Lambda
        \end{array}
      \right..
$
We define a system of points $\{(w_1(\epsilon),\cdots, w_n(\epsilon))\}_{\epsilon\in \R_+}$ given by
$ \left\{
                    \begin{array}{ll}
                      w_i(\epsilon):=\epsilon, & i\in \Lambda \\
                      w_i(\epsilon):=w_i, & i\notin \Lambda
                    \end{array}
                  \right..
$
It is easy to check that there is a positive $\epsilon_0<1$ such that
$(w_1(\epsilon),\cdots, w_n(\epsilon))\in B^*$ for $\forall \epsilon\in(0, \epsilon_0].$

We always have
\begin{eqnarray*}
 M(-\log|w_1(r)|,\cdots, -\log|w_n(r)|)  &=& M(-\log|w_1(\epsilon_0)|,\cdots, -\log|w_n(\epsilon_0)|) \\
   && +(-\log r +\log \varepsilon_0)(\sum_{i\in \Lambda}E_i)
\end{eqnarray*}
for any sufficient small real positive number $r.$
Since $M(-\log|w_1(\epsilon_0)|,\cdots, -\log|w_n(\epsilon_0)|)$ is a positive matrix as $(w_1(\epsilon_0),\cdots, w_n(\epsilon_0))\in B^*$ and $\sum_{i\in \Lambda}E_i$ is a semi-positive matrix, we obtain
$M(-\log|w_1|,\cdots, -\log|w_n|)$ is a $\infty$-positive matrix and
\begin{eqnarray*}
   && M(-\log|w_1|,\cdots, -\log|w_n|) \\
   &=&M(-\log|w_1(\epsilon_0)|,\cdots, -\log|w_n(\epsilon_0)|)+(\nu_0-\log\frac{1}{\epsilon_0})(\sum_{i\in \Lambda}E_i).
\end{eqnarray*}

\item  We prove the second and the third statements in this lemma together.
 By the symmetry, it is sufficient to prove all statements in case of $i=1.$

Since $E_1$ is semi-positive symmetric and nonzero, $E_1$ is diagonalized by an orthogonal matrix $O$ such that
$$O^T E_1O=\left(
             \begin{array}{cccc}
               \lambda_1 & 0      & \cdots & 0 \\
               0         & \ddots & \vdots & \vdots\\
               \vdots    &  0     & \ddots & 0 \\
               0         & 0      &\cdots  & \lambda_g \\
             \end{array}
           \right)
$$with
$$\lambda_1\geq \lambda_2\geq \cdots \geq \lambda_{k_1}>0=\lambda_{k_1+1}=\lambda_{k_1+2}=\cdots$$
 where $k_1=\rank E_1.$

Then, we have
$O^TMO = x_1\mathrm{diag}[\lambda_1,\cdots,\lambda_{k_1},0,\cdots,0]+\sum_{i=2}^n x_iO^T E_iO$
and
\begin{eqnarray*}
  S &=& \det\big(x_1\mathrm{diag}[\lambda_1,\cdots,\lambda_{k_1},0,\cdots,0]+\sum_{i=2}^n x_iO^T E_iO \big) \\
   &=& P_{k_1}(x_{2},\cdots, x_n)(\prod_{i=1}^{k_1}\lambda_i)x_1^{k_1} + \sum_{i=1}^{k_1}P_{k_1-i}(x_{2},\cdots, x_n)a_{k_1-i}(\lambda_1,\cdots, \lambda_{k_1})x_1^{k_1-i},
\end{eqnarray*}
where each $P_i(x_{2},\cdots, x_n)$ is a homogenous polynomial of degree $g-i$ and each $a_i(y_{1},\cdots, y_{k_1})\in \R[y_{1},\cdots, y_{k_1}]$ is a homogenous polynomial of degree $i.$

For any $2\leq j\leq n,$  we define $E_j^{(1)}$ to be the $(g-k_1,g-k_1)$ matrix by deleting rows $1,\cdots,k_1$ and columns $1,\cdots,k_1$ of the matrix $O^T E_jO.$ So all $E_i^{(1)}$ are semi-positive. Let $M^{(1)}(x_2,\cdots, x_n):=\sum\limits_{j=2}^n x_jE_j^{(1)}.$
Then, $M^{(1)}(x_2,\cdots, x_n)$ is the $(g-k_1,g-k_1)$ matrix by deleting rows $1,\cdots,k_1$ and columns $1,\cdots,k_1$ of the matrix $O^T M(x_1,\cdots,x_n)O,$
and $P_{k_1}(x_{2},\cdots, x_n)=\det(M^{(1)}(x_2,\cdots, x_n)).$

At any point $(w_1,\cdots, w_n)\in B,$
$M^{(1)}(-\log|w_2|,\cdots, -\log|w_n|)$ is a $\infty$-positive $(g-k_1,g-k_1)$ matrix since $O^T M(-\log|w_1|,\cdots, -\log|w_n|)O$ is a $\infty$-positive $(g,g)$ matrix by (1) of this lemma.
Let $(0, w_2^{'},\cdots, w_n^{'})\in B_1^*$ be an arbitrary point.
The matrix $O^T M(-\log|\epsilon|,-\log|w_1^{'}|\cdots, -\log|w_n^{'}|)O$ is positive as $(\epsilon,w_2,\cdots, w_n)\in B^*$
for any nonzero sufficiently small real number  $\epsilon,$ and so $M^{(1)}(-\log|w_2^{'}|,\cdots, -\log|w_n^{'}|)$ is a positive $(g-\rank E_i)\times(g-\rank E_i)$ matrix. Thus  $M^{(1)}(x_2,\cdots, x_n)$ is a $(g-\rank E_i)\times(g-\rank E_i)$ logarithmical positive matrix function on $B_1^*.$
Therefore, $P_{k_1}(-\log|w_2|,\cdots, -\log|w_n|)$ is non zero at any point $(w_1,\cdots, w_n)\in B$ by the statement (1) of this lemma. In particular, we have that $P_{k_1}(x_2,\cdots,x_n)$ is a nonzero polynomial and $\deg_1 S=\rank E_1.$

\item See the proof of the statement(2).

\item Let $(w_1,\cdots, \underbrace{0}_{i},\cdots,w_n)\in B_i^*$ be an arbitrary point.
Then $(w_1,\cdots, t_i,\cdots,w_n)$ is in $B^*$ for any $t_i\in \C^*$ with sufficiently small $|t_i|.$
We have that
\begin{eqnarray*}
  A(w_1,\cdots, t_i \cdots,w_n) &=& \frac{(-\log|t_i|)^{\deg_i S}}{S(-\log|w_1|,\cdots, -\log|t_i| \cdots,-\log|w_n|) } \\
   &&\times \frac{Q(-\log|w_1|,\cdots, -\log|t_i| \cdots,-\log|w_n|)}{(-\log|t_i|)^{\deg_i S}}
\end{eqnarray*}
for any $t_i\in \C^*$ near zero point. Since both
$$\lim_{t_i\to 0}\frac{(-\log|t_i|)^{\deg_i S}}{S(-\log|w_1|,\cdots, -\log|t_i| \cdots,-\log|w_n|)}<\infty$$
and $$\lim_{t_i\to 0}\frac{Q(-\log|w_1|,\cdots, -\log|t_i| \cdots,-\log|w_n|)}{(-\log|t_i|)^{\deg_i S}}<\infty$$ exist,
$J(z_1,\cdots, \underbrace{0}_{i},\cdots,z_n)$ is well-defined on $B_i^*.$
It is easy to check that for any $p\in B_i^*,$ there is a neighborhood $V_p$ of $p$ in $B^*\cup B_i^*$
such that $\widetilde{A}(z_1,\cdots,z_n)$ is continues on  $V_p.$
Moreover, we have
\begin{eqnarray*}
   && \mathrm{Res}_i(A)(w_1,\cdots, \underbrace{0}_{i},\cdots,w_n) \\
   &=&\left\{
                                                                \begin{array}{ll}
                                                                  0, &  \deg_i Q < \deg_i S;\\
                        \frac{Q_i}{S_i}(-\log|w_1|,\cdots \widehat{-\log|w_i|}\cdots, -\log|w_n|), & \deg_i Q = \deg_i S.
                                                                \end{array}
                                                              \right.
\end{eqnarray*}
Thus, $\mathrm{Res}_i(A)$ is a smooth function on $B_i^*.$
\end{myenumi}
\end{proof}

\subsection{Behaviors of logarithmical canonical line bundles}
Let $\widetilde{t}_{i,j}=(-1)^{i+j}\det(\widetilde{T}_{i,j})$ be the $(i,j)$-th cofactor of $(T^{\alpha}_{i,j})$
where $\widetilde{T}_{i,j}$ is a $(N-1)\times(N-1)$ the matrix that from deleting row $i$ and column $j$ of the $N\times N$ matrix $(T^{\alpha}_{i,j}).$ \\

Let $\widetilde{K}_{i,j}$ be the $(N-1)\times(N-1)$ the matrix that from deleting row $i$ and column $j$ of the $N\times N$ matrix $(K_{i,j}).$ Let $\widetilde{k}_{i,j}=(-1)^{i+j}\det(\widetilde{K}_{i,j})$ be the $(i,j)$-th cofactor of the matrix $(K_{i,j}).$

For all pair $(i,j)$ with $1\leq i,j\leq N,$ we define  $(N-1,N-1)$ simple forms
$$\nu_{i,j}=\left\{
              \begin{array}{ll}
                dw_i^{\alpha}\wedge d\overline{w_j^{\alpha}}\bigwedge\limits_{k\neq i,k\neq j}^{1\leq k\leq N}dw_k^{\alpha}\wedge d\overline{w_k^{\alpha}}, & \hbox{$i\neq j$;} \\
                \bigwedge\limits_{k\neq i}^{1\leq k\leq N}dw_k^{\alpha}\wedge d\overline{w_k^{\alpha}}, & \hbox{$i=j$.}
              \end{array}
            \right.
$$
By \ref{dd-volume-equality},
we obtain
$(\partial\overline{\partial}\log \Phi_\alpha)^{N-1}=(N-1)!(\sum_{i=1}^N \widetilde{k}_{i,i}\nu_{ii}
-\sum_{i\neq j}^{1\leq i,j\leq N}\widetilde{k}_{j,i}\nu_{i,j}).$

\begin{lemma}\label{chern-class-current}
Let $\widetilde{h}$ be  an arbitrary  smooth Hermitian metric on the logarithmical cotangent bundle $[K_{\overline{\sA}_{g,\Gamma}}+D_\infty].$
Then, there holds that
$$\int_{\overline{\sA}_{g,\Gamma}}c_1([K_{\overline{\sA}_{g,\Gamma}}+D_\infty], \widetilde{h })^{N-k}\wedge \eta
=\int_{\sA_{g,\Gamma}}c_1(\sO_{\sA_{g,\Gamma}}(K_{\sA_{g,\Gamma}}), h_B)^{N-k}\wedge \eta$$
for any $d$-closed smooth $2k$-form($1\leq k\leq N$) $\eta$ on $\overline{\sA}_{g,\Gamma}.$
\end{lemma}
\begin{proof}
 According to the lemma \ref{Poincare-growth-metric}, the statement can be obtained directly by Mumford's argument in Theorem 1.4 of \cite{Mum77} and Koll\'ar argument of 5.18 in \cite{Kol87}.
\end{proof}

\begin{theorem}\label{non-ample-logarithmical-cotangent-bundle}
Let $\Gamma\subset \Sp(g,\Z)$ be a neat arithmetic subgroup and $\Sigma_{\mathfrak{F}_0}$ a
$\overline{\Gamma_{\mathfrak{F}_0}}$(or $\mathrm{GL}(g,\Z)$)-admissible polyhedral decomposition of $C(\mathfrak{F}_0)$ regular with respect to $\Gamma.$
Let $\overline{\sA}_{g,\Gamma}$ be the toroidal compactification  of $\sA_{g,\Gamma}:=\mathfrak{H}_{g}/\Gamma$ constructed by  $\Sigma_{\mathfrak{F}_0}.$

Assume that the boundary divisor $D_{\infty}:=\overline{\sA}_{g,\Gamma}\setminus\sA_{g,\Gamma}$ is  simple normal crossing. Let $D_i$ be an arbitrary irreducible component of $D_\infty.$
The intersection number
$$ D_i\cdot(K_{\overline{\sA}_{g,\Gamma}}+D_\infty)^{\dim_\C \sA_{g,\Gamma} -1}=0$$
if one of the following conditions is satisfied :
(i)\, $g=2,$ \,\,(ii)\, $D_i$ is constructed from an edge $\rho_i$ in $\Sigma_{\mathfrak{F}_{\min}}$ for some minimal cusp $\mathfrak{F}_{\min}$ such that
$\mathrm{Int}(\rho_i)\subset C(\mathfrak{F}_{\min}).$
\end{theorem}
\begin{proof}
Let $||\cdot||_i$ be an arbitrary Hermitian metric on the line bundle $[D_i]$ and  $s_i$  the global section
of $[D_i]$ defining $D_i.$
By  the lemma \ref{chern-class-current}, we have :
\begin{equation}\label{intersection-chern-form}
    (K_{\overline{\sA}_{g,\Gamma}}+D_\infty)^{\dim_\C \sA_{g,\Gamma} -1}\cdot D_i
   =\lim_{\epsilon\to 0}\int_{\overline{\sA}_{g,\Gamma}\setminus D_\infty(\epsilon)}(\frac{\sqrt{-1}}{2\pi}\partial\overline{\partial}\log\Phi_{g,\Gamma})^{N-1}\wedge c_1([D_i],||\cdot||_i)
\end{equation}
where $D_\infty(\epsilon)$ is a tube of radius $\epsilon$ around $D_\infty.$ Thus
\begin{eqnarray*}
    &&(K_{\overline{\sA}_{g,\Gamma}}+D_\infty)^{\dim_\C \sA_{g,\Gamma} -1}\cdot D_i \\
   &=&\lim_{\epsilon\to 0}-(\frac{\sqrt{-1}}{2\pi})^N\int_{\overline{\sA}_{g,\Gamma}\setminus D_\infty(\epsilon)}(\partial\overline{\partial}\log\Phi_{g,\Gamma})^{N-1}\wedge \partial\overline{\partial}\log ||s_i||_i\\
   &=&\lim_{\epsilon\to 0}\frac{1}{2}(\frac{\sqrt{-1}}{2\pi})^N \int_{\overline{\sA}_{g,\Gamma}\setminus D_\infty(\epsilon)}(\partial\overline{\partial}\log\Phi_{g,\Gamma})^{N-1}\wedge d(\partial-\overline{\partial})\log ||s_i||_i\\
   &=&-\lim_{\epsilon\to 0}\frac{\sqrt{-1}}{4\pi}C_{N-1} \int_{\partial D_\infty(\epsilon)}(\partial\overline{\partial}\log\Phi_{g,\Gamma})^{N-1}\wedge(\partial-\overline{\partial})\log ||s_i||_i.
\end{eqnarray*}

Let $(U_\alpha,(w_1^{\alpha},\cdots, w_N^{\alpha}))$ be a local chart  as in \ref{local-chart-toroidal-compactification} such that  $D_i\cap U_{\alpha}=\{w_i^\alpha=0\}.$
We write $||s_i||_i=h_\alpha(w)|w_i^\alpha|^2$ on $U_\alpha.$
Then,
\begin{eqnarray*}
   && \int_{\partial D_\infty(\epsilon)\cap U_\alpha}C_{N-1}(\partial\overline{\partial}\log\Phi_{\alpha})^{N-1}\wedge(\partial-\overline{\partial})\log(h_\alpha(w)|w_i^\alpha|^2) \\
   &=& \int_{\partial D_\infty(\epsilon)\cap U_\alpha}C_{N-1}(\partial\overline{\partial}\log\Phi_{\alpha})^{N-1}\wedge
   \big\{2\sqrt{-1}\mathrm{Im}(\partial\log |w_i^\alpha|)+(\partial-\overline{\partial})\log h_\alpha(w)\big\}\\
   &=& 2\sqrt{-1}\mathrm{Im}\big(C_{N-1}\int_{\partial D_\infty(\epsilon)\cap U_\alpha}\partial\log |w_i^\alpha|\wedge(\partial\overline{\partial}\log\Phi_{\alpha})^{N-1} \big)\\
   &&+ C_{N-1}\int_{\partial D_\infty(\epsilon)\cap U_\alpha}(\partial\overline{\partial}\log\Phi_{\alpha})^{N-1}\wedge(\partial-\overline{\partial})\log h_\alpha(w).
\end{eqnarray*}

Using similar calculation as Proposition 1.2 in \cite{Mum77}, we get
\begin{equation}\label{vanishing-result-1}
  \lim_{\epsilon\to 0} \int_{\partial D_\infty(\epsilon)\cap U_\alpha}(\partial\overline{\partial}\log\Phi_{\alpha})^{N-1}\wedge(\partial-\overline{\partial})\log h_\alpha(w)=0.
\end{equation}

On $U_\alpha^*=U_\alpha\setminus D_\infty,$ we have
\begin{eqnarray*}
   && \frac{1}{(N-1)!} \partial\log |w_i^\alpha|\wedge(\partial\overline{\partial}\log \Phi_\alpha)^{N-1} \\
   &=&\widetilde{k}_{i,i}\frac{dw_i^\alpha}{w_i^\alpha}\wedge \bigwedge_{1\leq m\leq N,m\neq i}dw_m^{\alpha}\wedge d\overline{w_m^{\alpha}}\\
   &&+\sum_{1\leq m\leq N,m\neq i}\widetilde{k}_{i,m}\frac{dw_m^{\alpha}}{w_i^\alpha}\wedge \bigwedge_{1\leq l\leq N,l\neq m}dw_l^{\alpha}\wedge d\overline{w_l^{\alpha}}
\end{eqnarray*}

Let $T_j(\epsilon)$ be the tube neighborhood of $\{w_j=0\}$ in $U_\alpha $ for $j=1,\cdots,N.$
Due to $dw_k^{\alpha}\wedge d\overline{w_k^{\alpha}}=0 \,\, \mbox{ on } \partial T_k(\epsilon):=\{|w_k^\alpha|=\epsilon\},$ we get
\begin{eqnarray*}
   &&  \int_{\partial D_\infty(\epsilon)\cap U_\alpha}\partial\log |w_i^\alpha|\wedge(\partial\overline{\partial}\log \Phi_\alpha)^{N-1}\\
   &=& \sum_{j=1}^N\int_{\partial T_j(\epsilon)\cap\partial D_\infty(\epsilon) }\partial\log |w_i^\alpha|\wedge(\partial\overline{\partial}\log \Phi_\alpha)^{N-1}\\
   &=&(N-1)!\int_{\partial T_i(\epsilon)\cap\partial D_\infty(\epsilon)}\frac{\widetilde{k}_{i,i}}{w_i^\alpha}dw_i^\alpha\wedge \bigwedge_{1\leq m\leq N,m\neq i}dw_m^{\alpha}\wedge d\overline{w_m^{\alpha}}\\
   &&+(N-1)!\sum_{1\leq j\leq N, j\neq i}\int_{\partial T_j(\epsilon)\cap\partial D_\infty(\epsilon) }\frac{\widetilde{k}_{i,j}}{w_i^\alpha}dw_j^{\alpha}\wedge \bigwedge_{1\leq l\leq N, l\neq j}dw_l^{\alpha}\wedge d\overline{w_l^{\alpha}}.
\end{eqnarray*}
Since $\sqrt{-1}\partial\overline{\partial}\log\Phi_{\alpha}$ has Poincar\'e growth on boundary, the above integral is bounded uniformly on $\epsilon.$

For any integer $j\in[1,N]$ with $j\neq i,$ we have
\begin{eqnarray*}
 \frac{\widetilde{k}_{i,j}}{w_i^\alpha}  &=&(-1)^{N-1}(\frac{g+1}{4})^{N-1} \\
   && \times \frac{\widetilde{t}_{i,j}(\log|w_1^\alpha|,\cdots,\log|w_N^\alpha|) }{w_j^\alpha(\prod_{1\leq l\leq N,l\neq j}|w_l^\alpha|^2)(F^{\alpha})^{2(N-1)}(\log|w_1^\alpha|,\cdots,\log|w_N^\alpha|)} \,\, \mbox{ on }\partial T_j(\epsilon)\cap\partial D_\infty(\epsilon)
\end{eqnarray*}
where  $\widetilde{t}_{i,j}=(-1)^{i+j}\det(\widetilde{T}_{i,j})$ is the $(i,j)$-th cofactor of the matrix $(T_{l,m})_{1\leq l,m\leq N}.$

The lemma \ref{degree-determine-polynomal} says
$$\deg_j \widetilde{t}_{i,j}\leq 2(N-1)\deg_jF-1 \,\,\forall j\neq i,$$
we then get
\begin{equation}\label{vanishing-result-2}
\lim_{\epsilon\to 0} \int_{\partial T_j(\epsilon)\cap\partial D_\infty(\epsilon) }\frac{\widetilde{k}_{i,j}}{w_i^\alpha}dw_j^{\alpha}\wedge \bigwedge_{1\leq l\leq N, l\neq j}dw_l^{\alpha}\wedge d\overline{w_l^{\alpha}}=0\,\,\,\,  \forall j\neq i.\\
\end{equation}
by using the generalized Cauchy integral formula and the Poincar\'e growth of $(\partial\overline{\partial}\log\Phi_{\alpha})^{N-1}.$

Also, we have
\begin{eqnarray*}
 \frac{\widetilde{k}_{i,i}}{w_i^\alpha}  &=&(-1)^{N-1}(\frac{g+1}{4})^{N-1} \\
   &&\times \frac{\widetilde{t}_{i,i}(\log|w_1^\alpha|,\cdots,\log|w_N^\alpha|) }{w_i^\alpha(\prod_{1\leq l\leq N,l\neq i}|w_l^\alpha|^2)(F^{\alpha})^{2(N-1)}(\log|w_1^\alpha|,\cdots,\log|w_N^\alpha|)} \,\, \mbox{ on }\partial T_i(\epsilon)\cap\partial D_\infty(\epsilon)
\end{eqnarray*}
where  $\widetilde{t}_{i,i}=\det(\widetilde{T}_{i,i})$ is the $(i,i)$-th cofactor of the matrix $(T_{l,m})_{1\leq l,m\leq N}.$

Let $d_i=\deg_iF^{\alpha}[x_1,\cdots, x_N].$ We can write
$$F^{\alpha}=P(x_1,\cdots, \widehat{x_i},\cdots, x_N)x_i^{d_i}+ \cdots,$$
where $P$ is a homogenous polynomial in $\R[x_1,\cdots, \widehat{x_i},\cdots, x_N]$ of degree $g-d_i.$
For any integers $l,m $ in $[1,N]\setminus \{i\},$ we define
$A_{l,m}:=PP_{lm}- P_lP_m,$
where $P_{lm}:=\frac{\partial^2 P}{\partial x_l\partial x_m}, \,\, P_l:=\frac{\partial P}{\partial x_l}.$
We  get  a $(N-1)\times(N-1)$ matrix $(A_{l,m})_{l,m\in [1,N]\setminus \{i\}}.$
Then, the coefficient of the term $x_i^{2(N-1)d_i}$ in $\widetilde{t}_{i,i}$ is just
$\det(A_{l,m}).$ Again using  the generalized Cauchy integral formula and the Poincar\'e growth of $(\partial\overline{\partial}\log\Phi_{\alpha})^{N-1},$
we have
\begin{eqnarray*}
   && \lim_{\epsilon\to 0}\frac{1}{2\pi\sqrt{-1} } \int_{\partial T_i(\epsilon)\cap\partial D_\infty(\epsilon) }\frac{\widetilde{k}_{i,i}}{w_i^\alpha}
   dw_i^\alpha\wedge \bigwedge_{1\leq m\leq N,m\neq i}dw_m^{\alpha}\wedge d\overline{w_m^{\alpha}} \\
   &=&(-1)^{N-1}(\frac{g+1}{4})^{N-1}\\
   &&\times\int_{\{w_i^{\alpha}=0\}}\frac{\det(A_{l,m})(\log |w_1^{\alpha}|,\cdots, \widehat{\log |w_i^{\alpha}|},\cdots, \log |w_N^{\alpha}|)\bigwedge\limits_{1\leq m\leq N,m\neq i}dw_m^{\alpha}\wedge d\overline{w_m^{\alpha}} }{(\prod_{1\leq l\leq N,l\neq i}|w_l^\alpha|^2)P^{2(N-1)}(\log|w_1^\alpha|,\cdots, \widehat{\log |w_i^{\alpha}|},\cdots,,\log|w_N^\alpha|)}.
\end{eqnarray*}

Therefore, we obtain
\begin{eqnarray*}
   &&\lim_{\epsilon\to 0}\frac{1}{2}(\frac{\sqrt{-1}}{2\pi})^N \int_{\partial D_\infty(\epsilon)\cap U_\alpha}(\partial\overline{\partial}\log\Phi_{g,\Gamma})^{N-1}\wedge(\partial-\overline{\partial})\log ||s_i||_i\\
   &=& \frac{\sqrt{-1}}{4\pi}\times \lim_{\epsilon\to 0} \int_{\partial D_\infty(\epsilon)\cap U_\alpha}(\frac{\sqrt{-1}}{2\pi}\partial\overline{\partial}\log\Phi_{\alpha})^{N-1}\wedge(\partial-\overline{\partial})\log(h_\alpha(w)|w_i^\alpha|^2)\\
   &=& \frac{-1}{2\pi}\mathrm{Im}\big(\lim_{\epsilon\to 0}\int_{\partial D_\infty(\epsilon)\cap U_\alpha}\partial\log |w_i^\alpha|\wedge(\frac{\sqrt{-1}}{2\pi}\partial\overline{\partial}\log\Phi_{\alpha})^{N-1} \big) \\
   &=& \frac{-(N-1)!}{2\pi}\mathrm{Im}\big(\lim_{\epsilon\to 0}(\frac{\sqrt{-1}}{2\pi})^{N-1}\int_{\partial D_\infty(\epsilon)\cap T_i(\epsilon)}
   \frac{\widetilde{k}_{i,i}}{w_i^\alpha}
   dw_i^\alpha\wedge \bigwedge_{1\leq m\leq N,m\neq i}dw_m^{\alpha}\wedge d\overline{w_m^{\alpha}} \big) \\
   &=&(-1)^N(N-1)!(\frac{g+1}{4})^{N-1} \\
   &&\times  \int_{\{w_i^{\alpha}=0\}}\frac{\det(PP_{lm}- P_lP_m)}{P^{2(N-1)}}(\log|w_1^\alpha|,\cdots,\widehat{_i},\cdots,\log|w_N^\alpha|)\frac{\bigwedge\limits_{m\neq i}^{1\leq m\leq N}(\frac{\sqrt{-1}}{2\pi}dw_m^{\alpha}\wedge d\overline{w_m^{\alpha}})}{\prod_{1\leq l\leq N,l\neq i}|w_l^\alpha|^2}.
\end{eqnarray*}
In all conditions(i),(ii),
the polynomial $P$ never takes  zero value and  $\det(PP_{lm}- P_lP_m)$ is always zero by
the lemma \ref{lemma-on-degree-of-polynomial}.
\end{proof}

\subsection{Intersection theory for infinity divisor boundaries and non ampleness of logarithmical canonical bundles}
Let $d$ be an integer  with  $1\leq d\leq N-1.$
Let $D_1,\cdots, D_{d}$ be $d$ irreducible components of the boundary divisor $D_\infty.$  For each integer $i\in[1, d],$ let $||\cdot||_i$ be an arbitrary Hermitian metric on the line bundle $[D_i]$ and let $s_i$ be a global section of $[D_i]$ defining $D_i.$

Now we study the intersection number $(K_{\overline{\sA}_{g,\Gamma}}+D_\infty)^{N-d}\cdot D_1\cdots D_{d}.$
By  the lemma \ref{chern-class-current}, we have :
\begin{eqnarray*}
    &&(K_{\overline{\sA}_{g,\Gamma}}+D_\infty)^{N-d}\cdot D_1\cdots D_{d} \\
   &=&\lim_{\epsilon\to 0}\int_{\overline{\sA}_{g,\Gamma}-D_\infty(\epsilon)}(\frac{\sqrt{-1}}{2\pi}\partial\overline{\partial}\log\Phi_{g,\Gamma})^{N-d}\wedge \bigwedge_{i=1}^{d}c_1([D_i],||\cdot||_i)\\
   &=&\lim_{\epsilon\to 0}-C_{N-d+1}\int_{\overline{\sA}_{g,\Gamma}-D_\infty(\epsilon)}(\partial\overline{\partial}\log\Phi_{g,\Gamma})^{N-d}\wedge(\bigwedge_{i=2}^{d}c_1([D_i],||\cdot||_i))\wedge\partial\overline{\partial}\log ||s_1||_1\\
   &=&\lim_{\epsilon\to 0}\frac{C_{N-d+1}}{2}\int_{\overline{\sA}_{g,\Gamma}-D_\infty(\epsilon)}(\partial\overline{\partial}\log\Phi_{g,\Gamma})^{N-d}\wedge(\bigwedge_{i=2}^{d}c_1([D_i],||\cdot||_i))\wedge d(\partial-\overline{\partial})\log ||s_1||_1\\
   &=&-\lim_{\epsilon\to 0}\frac{\sqrt{-1}}{4\pi} C_{N-d}  \int_{\partial D_\infty(\epsilon)}(\partial\overline{\partial}\log\Phi_{g,\Gamma})^{N-d}\wedge(\bigwedge_{i=2}^{d}c_1([D_i],||\cdot||_i) )\wedge(\partial-\overline{\partial})\log ||s_1||_1.
\end{eqnarray*}
Let $(U_{\alpha},(w_1^{\alpha},\cdots, w_N^{\alpha}))$  be a local coordinate chart  as in \ref{local-chart-toroidal-compactification} such that
$D_1\cap U_{\alpha}=\{w_1^\alpha=0\}\,\mbox{ and }\,U_\alpha^*= U_{\alpha}\setminus D_\infty.$
We write $||s_1||_1=h_\alpha(w)|w_1^\alpha|^2$ on $U_\alpha.$
Since $\omega_{\mathrm{can}}$ has Poincar\'e growth on $D_\infty,$ we get
\begin{eqnarray*}
  &&  \lim_{\epsilon \to 0}\frac{\sqrt{-1}}{4\pi} C_{N-d}\int_{\partial D_\infty(\epsilon)\cap U_\alpha}(\partial\overline{\partial}\log\Phi_{g,\Gamma})^{N-d}\wedge(\bigwedge_{i=2}^{d}c_1([D_i],||\cdot||_i) )\wedge(\partial-\overline{\partial})\log ||s_1||_1.\\
  &=&\frac{-1}{2\pi}\mathrm{Im}\big(C_{N-d}\lim_{\epsilon \to 0}\int_{\partial D_\infty(\epsilon)\cap U_\alpha}\partial\log |w_1^\alpha|\wedge(\partial\overline{\partial}\log\Phi_{\alpha})^{N-d}\wedge(\bigwedge_{i=2}^{d}c_1([D_i],||\cdot||_i))\big).\\
\end{eqnarray*}

We now use $[i_1, \cdots, i_l]$ to mean a $l$-tuple $(i_1,\cdots,i_l)$ with $1\leq i_1<i_2<\cdots <i_l\leq N,$
and we say  $[i_1, \cdots, i_l]=[j_1, \cdots, j_l]$ if and only if $i_k=j_k\,\,\forall k=1, \cdots, l.$
For any $[i_1, i_2,\cdots, i_l],$  let $j_1, \cdots, j_{N-l}$ be integers in $\{1,\cdots, N\}\setminus\{i_1,\cdots, i_l \}$ satisfying $1\leq j_1<j_2<\cdots <j_{N-l}\leq N,$ define
$[i_1,\cdots, i_l]^{\circ}:=[j_1, \cdots, j_{N-l}].$
So $[i_1,\cdots, i_l]^{\circ}=[j_1,\cdots, j_l]^{\circ}$ if and only if $[i_1,\cdots, i_l]=[j_1,\cdots, j_l].$
For any $l$-tuple $[i_i, \cdots, i_l],$ we define a simple $(l,l)$-form on $U_\alpha$
$$\nu^{[i_i, \cdots, i_l]}_{[j_i, \cdots, j_l]}:=(dw^\alpha_{i_1}\wedge\cdots\wedge dw^\alpha_{i_l})\bigwedge(d\overline{w^\alpha_{j_1}}\wedge\cdots\wedge d\overline{w^\alpha_{j_l}}).$$
For a $N\times N$ matrix $A=(A_{ij}),$ we use $A^{[i_i, \cdots, i_d]}_{[j_i, \cdots, j_d]}$ to  mean a $(N-d, N-d)$ matrix by
deleting rows $i_1,\cdots, i_d$ and column $j_1,\cdots, j_d$ of the matrix $A=(A_{i,j}),$ and we define
$$\widetilde{A}^{[i_i, \cdots, i_d]}_{[j_i, \cdots, j_d]}=(-1)^{\sum_{k=1}^d(i_k+j_k)}\det(A^{[i_i, \cdots, i_d]}_{[j_i, \cdots, j_d]}).$$
Consider the $N\times N$ matrix $K=(K_{i,j})$ related to the equality \ref{dd-volume-equality},
we have
\begin{eqnarray*}
  && \frac{1}{(N-d)!}(\partial\overline{\partial}\log\Phi_{\alpha})^{N-d} \\
  &=&\sum_{[i_1, \cdots, i_d]}\widetilde{K}^{[i_1, \cdots, i_d]}_{[i_1, \cdots, i_d]}\bigwedge_{1\leq l\leq N-d}^{[j_1, \cdots, j_{N-d}]=[i_1,\cdots, i_d]^{\circ} }(dw^\alpha_{j_l}\wedge d\overline{w^\alpha_{j_l}})+
\sum_{[i_1, \cdots, i_d],[j_1, \cdots, j_d]}^{[i_1, \cdots, i_d]\neq [j_1, \cdots, j_d]} \pm \widetilde{K}^{[i_1, \cdots, i_d]}_{[j_1, \cdots, j_d]}\nu^{[i_1, \cdots, i_d]^{\circ}}_{[j_1, \cdots, j_d]^{\circ}}
\end{eqnarray*}
and we get the following equality on $\partial D_\infty(\epsilon)\cap U_\alpha$ :
\begin{eqnarray*}
  \frac{\partial\log |w_1^\alpha|\wedge(\partial\overline{\partial}\log\Phi_{\alpha})^{N-d}}{(N-d)!} &=& \underbrace{\sum_{[1,i_2, \cdots, i_d]}\frac{\widetilde{K}^{[1, i_2,\cdots, i_d]}_{[1, i_2\cdots, i_d]}}{2w_1^\alpha}dw^\alpha_1\wedge\bigwedge_{1\leq l \leq N-d}^{[j_1, \cdots, j_{N-d}]=[1,i_2,\cdots, i_d]^{\circ}}(dw^\alpha_{j_l}\wedge d\overline{w^\alpha_{j_l}})}_{I} \\
   &&+ \underbrace{\sum_{[1, i_2,\cdots, i_d],[1,j_2, \cdots, j_d]}^{[1, i_2,\cdots, i_d]\neq [1,j_2, \cdots, j_d]} \pm \frac{\widetilde{K}^{[1, i_2,\cdots, i_d]}_{[1,j_2, \cdots, j_d]}}{2w_1^\alpha}dw^\alpha_1\wedge\nu^{[1,i_2,\cdots, i_d]^{\circ}}_{[1,j_2, \cdots, j_d]^{\circ}}}_{II}\\
&&+ \underbrace{\sum_{[1, i_2,\cdots, i_d],[j_1, \cdots, j_d]}^{j_1\neq 1} \pm \frac{\widetilde{K}^{[1, i_2,\cdots, i_d]}_{[j_1, \cdots, j_d]}}{2w_1^\alpha}dw^\alpha_1\wedge\nu^{[1,i_2,\cdots, i_d]^{\circ}}_{[j_1, \cdots, j_d]^{\circ}}}_{III}\\
\end{eqnarray*}
Here $I,$ $II$ and $III$  are smooth forms on $\partial D_\infty(\epsilon)\cap U_\alpha.$
There is
$$\partial D_\infty(\epsilon)\cap U_\alpha=\bigcup_{j=1}^N\partial D_\infty(\epsilon)\cap \partial T_j(\epsilon),$$
where each  $T_j(\epsilon)$ is the tube neighborhood of $\{w_j=0\}$ in $U_\alpha $ with radius $\varepsilon.$
Certainly, it is not necessary that $D_j\cap U_\alpha=\{w_k=0\}$ for some $k$ if $j\geq 2.$
\begin{lemma}\label{lemma-intersection-calculation-1}
Let $\eta$ be an arbitrary smooth $(d-1,d-1)$-form on $\overline{\sA}_{g,\Gamma}.$ We have
$$\lim_{\epsilon\to 0}\int_{\partial D_\infty(\epsilon)\cap \partial T_k(\epsilon)} I\wedge \eta=0$$
for any integer $k\in [2,N].$
\end{lemma}
\begin{proof}
Since
$dw_k^{\alpha}\wedge d\overline{w_k^{\alpha}}=0 \,\, \mbox{ on } \partial T_k(\epsilon)(=\{|w_k^\alpha|=\epsilon\}),$
we get
$$I= \sum\limits_{[1,i_2, \cdots, i_d],k\in \{i_2, \cdots, i_d \}}\theta([1,i_2,\cdots i_d])
\,\,\,\mbox{   on   }\partial T_k(\epsilon),$$
where
$$\theta([1,i_2,\cdots i_d]):= \widetilde{K}^{[1,i_2, \cdots, i_d]}_{[1, i_2,\cdots, i_d]}\frac{dw^\alpha_1}{2w_1^\alpha}\wedge\bigwedge_{l=1}^{N-d}
(dw^\alpha_{j_l}\wedge d\overline{w^\alpha_{j_l}})\,\,\,\mbox{   on   }\partial T_k(\epsilon)$$
with
$[j_1, \cdots, j_{N-d}]=[1,\cdots, i_l]^{\circ}.$
On the coordinate chart $(U_\alpha,(w^\alpha_1,\cdots, w^\alpha_N)),$ we write $\eta=\sum_{\beta}c_{\beta}\eta_{\beta}$ where each $c_\beta$ is a  smooth function on $U_\alpha$ and
each $\eta_\beta$ is a simple $(d-1,d-1)$ form  given by the wedge product of some $dw^\alpha_l$'s and some $d\overline{w^\alpha_m}$'s with coefficient $1.$
It is sufficient to prove the equality
$$\lim\limits_{\epsilon\to 0}\int_{\partial D_\infty(\epsilon)\cap \partial T_k(\epsilon)} \theta([1,i_2,\cdots i_d])\wedge c_\beta\eta_\beta=0$$ for
any  $d$-tuple $[1,i_2,\cdots i_d]$ satisfying $k\in  \{i_2, \cdots, i_d \},$ any simple $(d-1,d-1)$ form $\eta_{\beta}$ with coefficient $1$ and any smooth function $c_\beta$ on  $U_\alpha.$

Let $\eta_{\beta}$ be an arbitrary simple $(d-1,d-1)$ form with coefficient $1$ and $c_\beta$ an arbitrary  smooth function  on  $U_\alpha.$ We may require that $\eta_\beta$ does not contain  the factor $dw^\alpha_1$ and contains neither a factor like $dw^\alpha_{j_l}$ nor a factor like $d\overline{w^\alpha_{j_m}}$(or else $\theta([1,i_2,\cdots i_d])\wedge \eta_\beta=0$). Since $\theta\wedge \eta_{\beta}$ is a simple $(N, N-1)$-form with coefficient $\frac{\widetilde{K}^{[1, i_2,\cdots, i_d]}_{[1, i_2,\cdots, i_d]}}{2w_1^\alpha},$ $\eta_\beta$ must contain $dw^\alpha_k.$
We also require that $\eta_\beta$ does not contain the factor $d\overline{w^\alpha_k}$(or else $\eta_\beta=0$ on $\partial T_k(\epsilon)$).
Then, $\eta_\beta$ contains the factor $d\overline{w^\alpha_1}$ and
\begin{eqnarray*}
   && \theta([1,i_2,\cdots i_d])\wedge \eta_{\beta} \\
   &=&  \pm\widetilde{K}^{[1, i_2,\cdots, i_d]}_{[1,i_2, \cdots, i_d]}dw^\alpha_k \wedge \frac{dw^\alpha_1\wedge d\overline{w^\alpha_1}}{2w_1^\alpha}\wedge \bigwedge_{j=2}^{k-1}(dw^\alpha_j\wedge d\overline{w^\alpha_j}) \wedge \bigwedge_{j=k+1}^{N}(dw^\alpha_j\wedge d\overline{w^\alpha_j}) \,\,\mbox{ on }\,\,\partial T_k(\epsilon).
\end{eqnarray*}

Since $(\partial\overline{\partial}\log\Phi_{\alpha})^{N-d}$ is a  $(N-d,N-d)$ form of Poincar\'e growth,
there is a small neighborhood $V_0$ of the origin point $0\in U_\alpha$ such that
 $$  |\widetilde{K}^{[1, i_2,\cdots, i_d]}_{[1,i_2, \cdots, i_d]}|\leq \frac{ C }{\prod_{m=1}^{N-d} |w^\alpha_{j_m}|^2(\log |w^\alpha_{j_m}|)^2}   \mbox{ on }  V_0\cap U_\alpha^*$$
where $C$ is a constant depending on $V_0.$
It is well-known that there is
$$\lim_{\varepsilon\to 0^+}\varepsilon(\int_{\varepsilon}^A\frac{dr}{r(\log r)^2} )^n=0$$
for any real numbers  $A,n>0.$
Therefore, we have
$$\lim_{\epsilon\to 0}\int_{\partial D_\infty(\epsilon)\cap \partial T_k(\epsilon)} \theta([1,i_2,\cdots i_d])\wedge c_\beta\eta_\beta=0.$$
\end{proof}

\begin{lemma}\label{lemma-intersection-calculation-2}
Let $\eta$ be an arbitrary smooth $(d-1,d-1)$-form on $\overline{\sA}_{g,\Gamma}.$ We have
$$\lim_{\epsilon\to 0}\int_{\partial D_\infty(\epsilon)\cap \partial T_k(\epsilon)} II\wedge \eta=0$$
for any integer $k\in [2,N].$
\end{lemma}
\begin{proof}
We can write
$II=\sum\limits_{[1, i_2,\cdots, i_d],[1,j_2, \cdots, j_d]}^{[1, i_2,\cdots, i_d]\neq [1,j_2, \cdots, j_d]} \pm \theta^{[1, i_2,\cdots, i_d]}_{[1,j_2, \cdots, j_d]},$
where
$$\theta^{[1, i_2,\cdots, i_d]}_{[1,j_2, \cdots, j_d]}:= \widetilde{K}^{[1, i_2,\cdots, i_d]}_{[1,j_2, \cdots, j_d]}\frac{dw^\alpha_1}{2w_1^\alpha}\wedge\nu^{[1,i_2,\cdots, i_d]^{\circ}}_{[1,j_2, \cdots, j_d]^{\circ}}.$$
Since
$dw_k^{\alpha}\wedge d\overline{w_k^{\alpha}}=0 \,\, \mbox{ on } \partial T_k(\epsilon),$
it is sufficient to show
$$\lim_{\epsilon\to 0}\int_{\partial D_\infty(\epsilon)\cap \partial T_k(\epsilon)} \theta^{[1, i_2,\cdots, i_d]}_{[1,j_2, \cdots, j_d]}\wedge c_\beta\eta_\beta=0$$
for any  simple $(d-1,d-1)$ form $\eta_{\beta}$ with coefficient $1,$ any smooth function $c_\beta$ on  $U_\alpha$ and any two $d$-tuples
$[1, i_2,\cdots, i_d], [1,j_2, \cdots, j_d]$ satisfying
$$(!)\,\,\, \,\,\,\, [1, i_2,\cdots, i_d]\neq [1,j_2, \cdots, j_d]\,\,\,\mbox{ and } \,\, k\in \{i_2,\cdots, i_d\}\bigcup \{j_2, \cdots, j_d\}.$$

Let $\eta_{\beta}$ be an arbitrary simple $(d-1,d-1)$ form with coefficient $1$ and $c_\beta$ an arbitrary  smooth function  on  $U_\alpha.$
Let $[1, i_2,\cdots, i_d],[1,j_2, \cdots, j_d]$ be two arbitrary $d$-tuples satisfying the condition (!),
and let $[i_{d+1},\cdots, i_{N}]=[1, i_2,\cdots, i_d]^{\circ}$ and $[j_{d+1},\cdots, j_{N}]=[1,j_2,\cdots, j_{d}]^{\circ}.$ So $i_{d+1}\geq 2, j_{d+1}\geq 2$ and
$\theta^{[1, i_2,\cdots, i_d]}_{[1,j_2, \cdots, j_d]}= \widetilde{K}^{[1, i_2,\cdots, i_d]}_{[1,j_2, \cdots, j_d]}\frac{dw^\alpha_1}{2w_1^\alpha}\wedge\nu^{[i_{d+1},\cdots, i_{N}]}_{[j_{d+1},\cdots, j_{N}]}.$

We can require that the simple form  $\eta_\beta$  contains no factor in the set
$$\{dw^\alpha_1\}\cup\{dw^\alpha_{i_{d+1}},\cdots,dw^\alpha_{i_{N}}\}\cup \{ d\overline{w^\alpha_{1}},d\overline{w^\alpha_{j_{d+2}}},\cdots,d\overline{w^\alpha_{j_N}}\}$$
(or else  $\theta^{[1, i_2,\cdots, i_d]}_{[j_1, \cdots, j_d]}\wedge \eta_\beta =0$). There are three cases :
\begin{itemize}
  \item $k\in \{i_2,\cdots, i_d\}\cap \{j_2, \cdots, j_d\}$ : Then  $k\notin \{i_{d+1},\cdots, i_{N}\}\cup \{j_{d+1}, \cdots, j_N\}.$
Suppose that $\theta^{[1, i_2,\cdots, i_d]}_{[j_1, \cdots, j_d]}\wedge \eta_\beta$ is a nonzero $(N,N-1)$ simple form.
Then $\eta_\beta$ contains the factor $dw^\alpha_k$ by that $k\notin \{i_{d+1},\cdots, i_{N}\}.$
 We may require that $\eta_\beta$ does not contain the factor $d\overline{w^\alpha_k}$(otherwise $\eta=0$ on $\partial T_k(\epsilon)$), and then $\eta_\beta$ contains the factor $d\overline{w^\alpha_1}.$
We have that
 $$\theta^{[1, i_2,\cdots, i_d]}_{[1,j_2, \cdots, j_d]}\wedge \eta_\beta=\pm\widetilde{K}^{[1, i_2,\cdots, i_d]}_{[1,j_2, \cdots, j_d]}  dw^\alpha_k\wedge\frac{dw^\alpha_1\wedge d\overline{w^\alpha_1}}{2w_1^\alpha}\wedge(\bigwedge_{l=2}^{k-1}dw^\alpha_l\wedge d\overline{w^\alpha_{l}}) \wedge(\bigwedge_{k+1}^{N}dw^\alpha_l\wedge d\overline{w^\alpha_{l}}).$$
by that $k\notin \{j_{d+1},\cdots, j_{N}\}.$ Since that $k\notin \{i_{d+1},\cdots, i_{N}\}\cup\{j_{d+1},\cdots, j_{N}\},$
we obtain $$\lim_{\epsilon\to 0}\int_{\partial D_\infty(\epsilon)\cap \partial T_k(\epsilon)}\theta^{[1, i_2,\cdots, i_d]}_{[j_1, \cdots, j_d]}\wedge c_\beta\eta_\beta=0$$
by  the Poincar\'e growth  of the form $(\partial\overline{\partial}\log\Phi_{\alpha})^{N-d}.$

  \item $k\in \{i_2,\cdots, i_d\}$ but $k\notin\{j_2, \cdots, j_d\}$ : Then, $k\in\{ j_{d+1}, \cdots, j_{N}\},$ and so
$$\theta^{[1, i_2,\cdots, i_d]}_{[j_1, \cdots, j_d]}\wedge \eta_\beta=0\,\, \mbox{ on } \,\,   \partial T_k(\epsilon).$$

  \item $k\in \{j_2, \cdots, j_d\}\setminus\{i_2,\cdots, i_d\}$ : Then $k\in \{i_{d+1},\cdots, i_N\}$ but $k\notin\{j_{d+1},\cdots, j_N\}.$
  We require that $\eta_\beta$ does not have the factor $d\overline{w_{k}^\alpha}$(or else $\theta^{[1, i_2,\cdots, i_d]}_{[j_1, \cdots, j_d]}\wedge \eta_\beta=0$ on $\partial T_k(\epsilon)$ by $k\in \{i_{d+1},\cdots, i_N\}$), and so $\eta_\beta$ has the factor $d\overline{w_{1}^\alpha}$ as $j_{d+1}\geq 2.$ Thus, we get
  \begin{eqnarray*}
       && \theta^{[1, i_2,\cdots, i_d]}_{[1,j_2, \cdots, j_d]}\wedge \eta_\beta \\
       &=& \pm\frac{(-g-1)^{N-d}\det(T^{[1, i_2,\cdots, i_d]}_{[1,j_2, \cdots, j_d]})}{4^{N-d}(\prod\limits_{d+1\leq l\leq N}^{i_l\neq k} w_{i_l}^\alpha)(\prod\limits_{l=d+1}^N \overline{w_{j_l}} )(F^{\alpha})^{2(N-d)}(\log|w_1^\alpha|,\cdots,\log|w_N^\alpha|)} \\
       && \times \frac{dw_k^\alpha}{w_k^\alpha} \wedge \frac{dw^\alpha_1\wedge d\overline{w^\alpha_1}}{2w_1^\alpha}\wedge(\bigwedge_{i=2}^{k-1}  dw_i^\alpha d\wedge\overline{w_i^\alpha})\wedge(\bigwedge_{i=k+1}^{N}  dw_i^\alpha\wedge d\overline{w_i^\alpha}) \mbox{ on }    \partial T_k(\epsilon).
\end{eqnarray*}
Since $k\in\{ j_{d+1}, \cdots, j_{N}\}$ but $k\notin \{ i_{d+1}, \cdots, i_{N}\},$ we have
  $$\deg_k \det(T^{[1, i_2,\cdots, i_d]}_{[1,j_2, \cdots, j_d]})\leq 2(N-d)\deg_k F^\alpha-1.$$
by the lemma \ref{degree-determine-polynomal}.
Therefore, we  obtain
$$\lim_{\epsilon\to 0}\int_{\partial D_\infty(\epsilon)\cap \partial T_k(\epsilon)}\theta^{[1, i_2,\cdots, i_d]}_{[1,j_2, \cdots, j_d]}\wedge c_\beta\eta_\beta=\int_{\{w^\alpha_k=0\}}0=0.$$
by using the generalized Cauchy integral formula and using the Poincar\'e growth of $(\partial\overline{\partial}\log\Phi_{\alpha})^{N-d}.$
\end{itemize}
\end{proof}

\begin{lemma}\label{lemma-intersection-calculation-3}Let $\eta$ be an arbitrary smooth $(d-1,d-1)$-form  on $\overline{\sA}_{g,\Gamma}.$
We have that
$$\lim_{\epsilon\to 0}\int_{\partial D_\infty(\epsilon)\cap \partial T_k(\epsilon)} III\wedge \eta=0$$
for any  integer $k\in [1,N].$
\end{lemma}
\begin{proof}
Since
$dw_k^{\alpha}\wedge d\overline{w_k^{\alpha}}=0 \,\, \mbox{ on } \partial T_k(\epsilon),$
$III= \sum\limits_{[1, i_2,\cdots, i_d], [j_1, \cdots, j_d], j_1\neq 1}^{k\in \{1, i_2,\cdots, i_d\}\cup\{j_1, \cdots, j_d\}} \pm \theta^{[1, i_2,\cdots, i_d]}_{[j_1, \cdots, j_d]}   \,\,\,\mbox{   on   }\partial T_k(\epsilon)$
where $$\theta^{[1, i_2,\cdots, i_d]}_{[j_1, \cdots, j_d]}  := \widetilde{K}^{[1, i_2,\cdots, i_d]}_{[j_1, \cdots, j_d]}\frac{dw^\alpha_1}{2w_1^\alpha}\wedge\nu^{[1,\cdots, i_d]^{\circ}}_{[j_1, \cdots, j_d]^{\circ}}.$$
As $j_1\neq 1,$ we must have $[j_{1},\cdots, j_{d}]^{\circ}=[1,\cdots ]$ and so
$$III=0 \mbox{ on } \partial T_1(\epsilon).$$

Now we suppose that $k>1.$ It is sufficient to show
$$\lim_{\epsilon\to 0}\int_{\partial D_\infty(\epsilon)\cap \partial T_k(\epsilon)} \theta^{[1, i_2,\cdots, i_d]}_{[j_1, \cdots, j_d]}\wedge c_\beta\eta_\beta=0$$
for any  simple $(d-1,d-1)$ form $\eta_{\beta}$ with coefficient $1,$ any smooth function $c_\beta$ on  $U_\alpha$ and any two $d$-tuples
$[1, i_2,\cdots, i_d], [j_1, \cdots, j_d]$ satisfying
$$(!!)\,\,\,\,\,\, \, j_1\neq 1  \,\,\mbox{   and  }\,\,k\in \{1, i_2,\cdots, i_d\}\cup\{j_1, \cdots, j_d\}.$$
Let $\eta_{\beta}$ be an arbitrary simple $(d-1,d-1)$ form with coefficient $1$ and $c_\beta$ an arbitrary  smooth function  on  $U_\alpha.$
Let $[1, i_2,\cdots, i_d],[j_1, \cdots, j_d]$ be two arbitrary $d$-tuples satisfying the condition (!!), and
let $[i_{d+1},\cdots, i_{N}]=[1, i_2,\cdots, i_d]^{\circ}$ and $[1,j_{d+2},\cdots, j_{N}]=[j_{1},\cdots, j_{d}]^{\circ}.$
Then, we get
$$\theta^{[1, i_2,\cdots, i_d]}_{[j_1, \cdots, j_d]}  = \widetilde{K}^{[1, i_2,\cdots, i_d]}_{[j_1, \cdots, j_d]}
\frac{dw^\alpha_1\wedge d\overline{w^\alpha_1}}{2w_1^\alpha}\wedge \bigwedge_{l=d+1}^N dw^\alpha_{i_l}\wedge \bigwedge_{l=d+2}^Nd\overline{w^\alpha_{j_l}}.$$
We can require that the simple form $\eta_\beta$ contains no factor  in the set
$$\{dw^\alpha_1\}\cup\{dw^\alpha_{i_{d+1}},\cdots, dw^\alpha_{i_{N}}\}\cup \{ d\overline{w^\alpha_{1}},d\overline{w^\alpha_{j_{d+2}}},\cdots,d\overline{w^\alpha_{j_N}}\}$$
(or else $\theta^{[1, i_2,\cdots, i_d]}_{[j_1, \cdots, j_d]}\wedge \eta_\beta =0$).
There are three cases :
\begin{itemize}
  \item $k\in \{1, i_2,\cdots, i_d\}\cap\{j_1, \cdots, j_d\}$ : So $k\notin \{i_{d+1},\cdots, i_{N}\}\cup\{1,j_{d+2},\cdots, j_{N}\}.$

Suppose that $\theta^{[1, i_2,\cdots, i_d]}_{[j_1, \cdots, j_d]}\wedge \eta_\beta$ is a nonzero $(N,N-1)$ simple form with coefficient $\frac{\widetilde{K}^{[1, i_2,\cdots, i_d]}_{[j_1, \cdots, j_d]}}{2w_1^\alpha}.$
Then $\eta_\beta$ contains the factor $dw^\alpha_k$ by $k\notin \{i_{d+1},\cdots, i_{N}\}.$
 We may require $\eta_\beta$ does not contain the factor $d\overline{w^\alpha_k}$(otherwise, $\eta=0$ on $\partial T_k(\epsilon)$).
We have
 $$\theta^{[1, i_2,\cdots, i_d]}_{[j_1, \cdots, j_d]}\wedge \eta_\beta=\pm\widetilde{K}^{[1, i_2,\cdots, i_d]}_{[j_1, \cdots, j_d]}  dw^\alpha_k\wedge\frac{dw^\alpha_1\wedge d\overline{w^\alpha_1}}{2w_1^\alpha}\wedge(\bigwedge_{l=2}^{k-1}dw^\alpha_l\wedge d\overline{w^\alpha_{l}}) \wedge(\bigwedge_{k+1}^{N}dw^\alpha_l\wedge d\overline{w^\alpha_{l}})$$
by that $k\notin \{1,j_{d+2},\cdots, j_{N}\}.$
Then, $$\lim_{\epsilon\to 0}\int_{\partial D_\infty(\epsilon)\cap \partial T_k(\epsilon)}\theta^{[1, i_2,\cdots, i_d]}_{[j_1, \cdots, j_d]}\wedge c_\beta\eta_\beta=0$$
by that $k\notin \{i_{d+1},\cdots, i_{N}\}\cup\{1,j_{d+2},\cdots, j_{N}\}$ and the Poincar\'e growth  of the form $(\partial\overline{\partial}\log\Phi_{\alpha})^{N-d}.$

  \item $k\in \{1, i_2,\cdots, i_d\}$ bur $k\notin\{j_1, \cdots, j_d\}$ : So $k\in \{1,j_{d+2},\cdots, j_N\},$  and $i_{d+1}>1.$
Then,
  $$\theta^{[1, i_2,\cdots, i_d]}_{[j_1, \cdots, j_d]}\wedge \eta_\beta=0\,\, \mbox{ on } \,\,   \partial T_k(\epsilon).$$

  \item $k\in \{j_1, \cdots, j_d\}$ but $k\notin \{1, i_2,\cdots, i_d\}$ :  So $k\in \{i_{d+1},\cdots, i_N\}$ but $k\notin \{1,j_{d+2},\cdots, j_N\}$ and $i_{d+1}>1.$
  We require that $\eta_\beta$ has no factor $d\overline{w_{k}^\alpha}$(otherwise, $\theta^{[1, i_2,\cdots, i_d]}_{[j_1, \cdots, j_d]}\wedge \eta_\beta=0$ on $\partial T_k(\epsilon)$).
Then, we have
  \begin{eqnarray*}
       && \theta^{[1, i_2,\cdots, i_d]}_{[j_1, \cdots, j_d]}\wedge \eta_\beta \\
       &=& \pm\frac{(-g-1)^{N-d}\det(T^{[1, i_2,\cdots, i_d]}_{[j_1, \cdots, j_d]})}{4^{N-d}(\prod\limits_{d+1\leq l\leq N}^{i_l\neq k} w_{i_l}^\alpha)(\prod\limits_{l=d+1}^N \overline{w_{j_l}} )(F^{\alpha})^{2(N-d)}(\log|w_1^\alpha|,\cdots,\log|w_N^\alpha|)} \\
       && \times \frac{dw_k^\alpha}{w_k^\alpha} \wedge \frac{dw^\alpha_1\wedge d\overline{w^\alpha_1}}{2w_1^\alpha}\wedge(\bigwedge_{i=2}^{k-1}  dw_i^\alpha\wedge d\overline{w_i^\alpha})\wedge(\bigwedge_{i=k+1}^{N}  dw_i^\alpha\wedge d\overline{w_i^\alpha})  \mbox{ on }    \partial T_k(\epsilon).
\end{eqnarray*}
The lemma \ref{degree-determine-polynomal} says that
  $$\deg_k \det(T^{[1, i_2,\cdots, i_d]}_{[j_1, \cdots, j_d]})\leq 2(N-d)\deg_k F^\alpha-1.$$
Therefore, we  obtain
$$\lim_{\epsilon\to 0}\int_{\partial D_\infty(\epsilon)\cap \partial T_k(\epsilon)}\theta^{[1, i_2,\cdots, i_d]}_{[j_1, \cdots, j_d]}\wedge c_\beta\eta_\beta=0.$$
by the generalized Cauchy integral formula and the Poincar\'e growth of $(\partial\overline{\partial}\log\Phi_{\alpha})^{N-d}.$
\end{itemize}
\end{proof}

By the lemma \ref{lemma-on-degree-of-polynomial}, we can write
$$F^{\alpha}(x_1,\cdots, x_N)=S^{\alpha,1}(x_2,\cdots,  x_N)x_1^{\deg_1 F^\alpha}+ \mbox{ terms of lower degree of $x_1$ },$$
where $S_1$ is a homogenous polynomial in $\R[x_1,\cdots, \widehat{x_i},\cdots, x_N]$ of degree $(g-\deg_1 F^\alpha).$
For any integers $l,m $ in $[1,N],$ we define
a  $N\times N$ matrix $A^{\alpha}(1):=(A^{\alpha}(1)_{l,m})$ by setting
\begin{equation}\label{matrix-polynomial}
    A^{\alpha}(1)_{l,m}:=S^{\alpha,1}\frac{\partial^2 S^{\alpha,1}}{\partial x_l\partial x_m}- \frac{\partial S^{\alpha,1}}{\partial x_l}\frac{\partial S^{\alpha,1}}{\partial x_m}.
\end{equation}

Here is a direct consequence of the lemma \ref{degree-determine-polynomal}:
\begin{lemma}\label{residue-lemma}
The coefficient of term $x_1^{2(N-d)\deg_1 F^{\alpha}}$ in $\det(T^{[1, i_2,\cdots, i_d]}_{[1,i_2,\cdots, i_d]})$ is
$\det(P^{[1, i_2,\cdots, i_d]}_{[1,i_2,\cdots, i_d]}) $
where $P:=A^{\alpha}(1).$
\end{lemma}

We write the  real $(N-d,N-d)$-form
\begin{equation}\label{terms-volume-form-1}
  C_{N-d}(\partial\overline{\partial}\log\Phi_{\alpha})^{N-d}=(N-d)!\sum_{[i_1,\cdots, i_d],[j_1,\cdots, j_d] } \varsigma^{[i_1,\cdots, i_d]}_{[j_1,\cdots, j_d]}
\end{equation}
with $ \varsigma^{[i_1,\cdots, i_d]}_{[j_1,\cdots, j_d]}= \overline{\varsigma^{[j_1,\cdots, j_d]}_{[i_1,\cdots, i_d]}}.$
Each term in $C_{N-d}(\partial\overline{\partial}\log\Phi_{\alpha})^{N-d}$ can be written
\begin{equation}\label{terms-volume-form-2}
  \varsigma^{[i_1,\cdots, i_d]}_{[j_1,\cdots, j_d]}=c^{[ i_1,\cdots, i_d]}_{[j_1, \cdots, j_d]}\det(K^{[ i_1,\cdots, i_d]}_{[j_1, \cdots, j_d]})
\bigwedge_{l=d+1}^N(dw^\alpha_{i_l}\wedge d\overline{w^\alpha_{j_l}}),
\end{equation}
such that $c^{[ i_1,\cdots, i_d]}_{[j_1, \cdots, j_d]}$ is the sign $\pm$ depending on $([i_1,\cdots, i_d],[j_1,\cdots, j_d]),$ and
$$c^{[ i_1,\cdots, i_d]}_{[j_1, \cdots, j_d]}=c_{[ i_1,\cdots, i_d]}^{[j_1, \cdots, j_d]},\,\,\,  \,\, c^{[ i_1,\cdots, i_d]}_{[i_1, \cdots, i_d]}=1.$$

\begin{lemma}\label{lemma-intersection-calculation-4}
Let $P:=A^{\alpha}(1)$ be a $N\times N$ matrix of polynomials given in \ref{matrix-polynomial}.
Given a $d$-tuple $[\alpha_1,\cdots,\alpha_d]$ where $1\leq \alpha_1,\cdots,\alpha_d\leq N,$
we always use $N-d$ tuple $[\alpha_{d+1},\cdots,\alpha_N]$ to represent $[\alpha_1,\cdots,\alpha_d]^{\circ},$ where
$\{\alpha_{d+1},\cdots,\alpha_N\}=\{1,\cdots,N\}\setminus \{\alpha_1,\cdots,\alpha_d\}.$

For any two $d$-tuples $[1, i_2,\cdots, i_d],[1, j_2,\cdots, j_d],$ we define
a rational function 
\begin{equation}\label{rational-function-volume-residue}
    \xi^{[1, i_2,\cdots, i_d]}_{[1, j_2,\cdots, j_d]}(x_2,\cdots, x_N)=\frac{\det(P^{[1, i_2,\cdots, i_d]}_{[1,j_2,\cdots, j_d]})(x_2,\cdots, x_N)}{(S^{\alpha, 1})^{2(N-d)}(x_2,\cdots, x_N)},
\end{equation}
and we define a real $(N-d,N-d)$ form
$$\delta_{[1, j_2,\cdots, j_d]}^{[1, i_2,\cdots, i_d]}:=\frac{C_{N-d}}{2}
((\prod\limits_{l=d+1}^{N}w_{i_l}^\alpha \overline{w_{j_l}^\alpha})\bigwedge\limits_{l=d+1}^{N}(dw^\alpha_{i_l}\wedge d\overline{w^\alpha_{j_l}})
+(\prod\limits_{l=d+1}^{N}w_{j_l}^\alpha \overline{w_{i_l}^\alpha})\bigwedge\limits_{l=d+1}^{N}(dw^\alpha_{j_l}\wedge d\overline{w^\alpha_{i_l}})).
$$

Let $\eta$ be an arbitrary real smooth  $(d-1,d-1)$ form on $\overline{\sA}_{g,\Gamma}.$ We have :
\begin{eqnarray*}
    && \lim_{\epsilon \to 0}\frac{\sqrt{-1}}{4\pi} C_{N-d}\int_{\partial D_\infty(\epsilon)\cap U_\alpha}(\partial\overline{\partial}\log\Phi_{g,\Gamma})^{N-d}\wedge\eta \wedge(\partial-\overline{\partial})\log ||s_1||_1.\\
 &=&(-1)^{N-d+1}(\frac{g+1}{4})^{N-d}(N-d)!\\
&&\times\big\{ \sum_{[1,i_2, \cdots, i_d]}\int_{\{w_1=0\}}\xi^{[1, i_2,\cdots, i_d]}_{[1, i_2,\cdots, i_d]}(\log|w_2^\alpha|,\cdots,\log|w_N^\alpha|)\frac{\delta_{[1, i_2,\cdots, i_d]}^{[1, i_2,\cdots, i_d]}\wedge\eta}{\prod\limits_{l=d+1}^{N}|w_{i_l}^\alpha|^4}\\
&&+\sum_{[1,i_2, \cdots, i_d], [1,j_2,\cdots,j_d]}^{[1,i_2, \cdots, i_d]\neq[1,j_2,\cdots,j_d]}c^{[ i_1,\cdots, i_d]}_{[j_1, \cdots, j_d]}\int_{\{w_1=0\}}\xi^{[1, i_2,\cdots, i_d]}_{[1, j_2,\cdots, j_d]}(\log|w_2^\alpha|,\cdots,\log|w_N^\alpha|)\frac{\delta_{[1, j_2,\cdots, j_d]}^{[1, i_2,\cdots, i_d]}\wedge\eta}{\prod\limits_{l=d+1}^{N}|w_{i_l}^\alpha w_{j_l}^\alpha|^2 }\big\}
\end{eqnarray*}
where each $c^{[ i_1,\cdots, i_d]}_{[j_1, \cdots, j_d]}$ is a sign defined in \ref{terms-volume-form-1}-\ref{terms-volume-form-2}.
\end{lemma}
\begin{proof}
For any two $d$-tuples $[1,i_2,\cdots, i_d], [1,j_2,\cdots, j_d],$ we have
\begin{eqnarray*}
   \frac{dw^\alpha_1}{2w_1^\alpha}\wedge\varsigma^{[1,i_2,\cdots, i_d]}_{[1,i_2,\cdots, i_d]}&=& C_{N-d}\det(K^{[1, i_2,\cdots, i_d]}_{[1,i_2, \cdots, i_d]})\frac{dw^\alpha_1}{2w_1^\alpha}\wedge
\bigwedge_{l=d+1}^N(dw^\alpha_{i_l}\wedge d\overline{w^\alpha_{i_l}}) \\
   \frac{dw^\alpha_1}{2w_1^\alpha}\wedge\varsigma^{[1,i_2,\cdots, i_d]}_{[1,j_2,\cdots, j_d]} &=& c^{[1, i_2,\cdots, i_d]}_{[1,j_2, \cdots, j_d]}C_{N-d}\det(K^{[1, i_2,\cdots, i_d]}_{[1,j_2, \cdots, j_d]})\frac{dw^\alpha_1}{2w_1^\alpha}\wedge
\bigwedge_{l=d+1}^N(dw^\alpha_{i_l}\wedge d\overline{w^\alpha_{j_l}}).
\end{eqnarray*}

 Since $C_{N-d}(\partial\overline{\partial}\log\Phi_{\alpha})^{N-d}$ has Poincar\'e growth on $D_\infty,$ we get
$$|\det(K^{[1, i_2,\cdots, i_d]}_{[1,j_2, \cdots, j_d]})|<\frac{C}{\prod_{l=d+1}^{N}(|w_{i_l}^{\alpha}||\log |w_{i_l}^{\alpha}|)(|w_{j_l}^{\alpha}||\log |w_{j_l}^{\alpha}|)}\,\,
\mbox{ on } V_0\cap U^*_\alpha,$$
where $V_0$ is a sufficient small neighborhood of the origin point $(0,\cdots, 0)$ in $U_\alpha$ and $C$ is a constant depending on $V_0.$
Thus the $|\lim\limits_{\epsilon \to 0}\int_{\partial D_\infty(\epsilon)\cap U_\alpha}\theta([1, i_2,\cdots, i_d])\wedge \eta|$ is finite.
Moreover, we get that
$$\det(K^{[1, i_2,\cdots, i_d]}_{[1,j_2, \cdots, j_d]})= \frac{(-g-1)^{N-d}}{4^{N-d}}\frac{\det(T^{[1, i_2,\cdots, i_d]}_{[1,j_2,\cdots, j_d]})(\log|w_1^\alpha|,\cdots,\log|w_N^\alpha|)}{(\prod\limits_{l=d+1}^{N}w_{i_l}^\alpha \overline{w_{j_l}^\alpha})(F^{\alpha})^{2(N-d)}(\log|w_1^\alpha|,\cdots,\log|w_N^\alpha|)}\,\,\mbox{on }  \partial T_1(\epsilon)$$

Using the generalized Cauchy integral formula, we obtain that
\begin{eqnarray*}
   && \lim_{\epsilon \to 0}\int_{\partial D_\infty(\epsilon)\cap \partial T_1(\epsilon)}\frac{dw^\alpha_1}{2w_1^\alpha}\wedge\varsigma^{[1,i_2,\cdots, i_d]}_{[1,i_2,\cdots, i_d]}\wedge\eta\\
   &=& 2\pi \sqrt{-1} \frac{(-g-1)^{N-d}}{4^{N-d}}\int_{\{w_1=0\}}\xi^{[1, i_2,\cdots, i_d]}_{[1, i_2,\cdots, i_d]}(\log|w_2^\alpha|,\cdots,\log|w_N^\alpha|)\frac{C_{N-d}\bigwedge\limits_{l=d+1}^{N}(dw^\alpha_{i_l}\wedge d\overline{w^\alpha_{i_l}})\wedge\eta}{\prod\limits_{l=d+1}^{N}|w_{i_l}^\alpha|^2}\\
&=& 2\pi \sqrt{-1} \frac{(-g-1)^{N-d}}{4^{N-d}}\int_{\{w_1=0\}}\xi^{[1, i_2,\cdots, i_d]}_{[1, i_2,\cdots, i_d]}(\log|w_2^\alpha|,\cdots,\log|w_N^\alpha|)\frac{\delta_{[1, i_2,\cdots, i_d]}^{[1, i_2,\cdots, i_d]}\wedge\eta}{\prod\limits_{l=d+1}^{N}|w_{i_l}^\alpha|^4}
\end{eqnarray*}
and that if $[1,i_2,\cdots, i_d]\neq [1,j_2,\cdots, j_d]$ then
\begin{eqnarray*}
   && \frac{1}{2}\lim_{\epsilon \to 0}\int_{\partial D_\infty(\epsilon)\cap \partial T_1(\epsilon)}\frac{dw^\alpha_1}{2w_1^\alpha}\wedge(\varsigma^{[1,i_2,\cdots, i_d]}_{[1,j_2,\cdots, j_d]}+ \varsigma^{[1,j_2,\cdots, j_d]}_{[1,i_2,\cdots, i_d]})\wedge\eta\\
   &=& c^{[1, i_2,\cdots, i_d]}_{[1,j_2, \cdots, j_d]}2\pi \sqrt{-1} \frac{(-g-1)^{N-d}}{4^{N-d}}
   \int_{\{w_1=0\}}\xi^{[1, i_2,\cdots, i_d]}_{[1, j_2,\cdots, j_d]}(\log|w_2^\alpha|,\cdots,\log|w_N^\alpha|)\frac{\delta_{[1, j_2,\cdots, j_d]}^{[1, i_2,\cdots, i_d]}\wedge\eta}{\prod\limits_{l=d+1}^{N}|w_{i_l}^\alpha w_{j_l}^\alpha|^2 }
\end{eqnarray*}
We then have that
\begin{eqnarray*}
 &&\frac{C_{N-d}}{(n-d)!} \lim_{\epsilon \to 0}\int_{\partial D_\infty(\epsilon)\cap U_\alpha}\partial\log |w_1^\alpha|\wedge(\partial\overline{\partial}\log\Phi_{\alpha})^{N-d}\wedge\eta\\
  &=& C_{N-d}\lim_{\epsilon \to 0}\int_{\partial D_\infty(\epsilon)\cap \partial T_1(\epsilon)} I\wedge\eta +C_{N-d}\lim_{\epsilon \to 0}\int_{\partial D_\infty(\epsilon)\cap \partial T_1(\epsilon)} II\wedge\eta\\
  &=&\sum_{[1,i_2, \cdots, i_d]} \lim_{\epsilon \to 0}\int_{\partial D_\infty(\epsilon)\cap \partial T_1(\epsilon)}\frac{dw^\alpha_1}{2w_1^\alpha}\wedge\varsigma^{[1, i_2,\cdots, i_d]}_{[1, i_2,\cdots, i_d]}\wedge\eta \\
&&+ \sum_{[1,i_2, \cdots, i_d], [1,j_2,\cdots,j_d]}^{[1,i_2, \cdots, i_d]\neq[1,j_2,\cdots,j_d]}\frac{1}{2}\lim_{\epsilon \to 0}\int_{\partial D_\infty(\epsilon)\cap \partial T_1(\epsilon)}\frac{dw^\alpha_1}{2w_1^\alpha}\wedge(\varsigma^{[1, i_2,\cdots, i_d]}_{[1, j_2,\cdots, j_d]}+\varsigma^{[1, j_2,\cdots, j_d]}_{[1, i_2,\cdots, i_d]} )\wedge\eta.
\end{eqnarray*}
Finally, we obtain that
\begin{eqnarray*}
  &&  \lim_{\epsilon \to 0}\frac{\sqrt{-1}}{4\pi} C_{N-d}\int_{\partial D_\infty(\epsilon)\cap U_\alpha}(\partial\overline{\partial}\log\Phi_{g,\Gamma})^{N-d}\wedge\eta\wedge(\partial-\overline{\partial})\log ||s_1||_1.\\
 &=&\frac{-1}{2\pi}\mathrm{Im}\big(C_{N-d}\lim_{\epsilon \to 0}\int_{\partial D_\infty(\epsilon)\cap U_\alpha}\partial\log |w_1^\alpha|\wedge(\partial\overline{\partial}\log\Phi_{\alpha})^{N-d}\wedge \eta\big).\\
&=&(-1)^{N-d+1}(\frac{g+1}{4})^{N-d}(N-d)!\\
&&\times\big\{ \sum_{[1,i_2, \cdots, i_d]}\int_{\{w_1=0\}}\xi^{[1, i_2,\cdots, i_d]}_{[1, i_2,\cdots, i_d]}(\log|w_2^\alpha|,\cdots,\log|w_N^\alpha|)\frac{\delta_{[1, i_2,\cdots, i_d]}^{[1, i_2,\cdots, i_d]}\wedge\eta}{\prod\limits_{l=d+1}^{N}|w_{i_l}^\alpha|^4}\\
&&+\sum_{[1,i_2, \cdots, i_d], [1,j_2,\cdots,j_d]}^{[1,i_2, \cdots, i_d]\neq[1,j_2,\cdots,j_d]}c^{[ i_1,\cdots, i_d]}_{[j_1, \cdots, j_d]}\int_{\{w_1=0\}}\xi^{[1, i_2,\cdots, i_d]}_{[1, j_2,\cdots, j_d]}(\log|w_2^\alpha|,\cdots,\log|w_N^\alpha|)\frac{\delta_{[1, j_2,\cdots, j_d]}^{[1, i_2,\cdots, i_d]}\wedge\eta}{\prod\limits_{l=d+1}^{N}|w_{i_l}^\alpha w_{j_l}^\alpha|^2 }\big\}.
\end{eqnarray*}
\end{proof}

We have the following generalization of the theorem \ref{non-ample-logarithmical-cotangent-bundle}.
\begin{theorem}\label{recurrence-intersection-formula}
Let $\Gamma\subset \Sp(g,\Z)$ be a neat arithmetic subgroup and $\Sigma_{\mathfrak{F}_0}$ a
$\overline{\Gamma_{\mathfrak{F}_0}}$(or $\mathrm{GL}(g,\Z)$)-admissible polyhedral decomposition of $C(\mathfrak{F}_0)$ regular with respect to $\Gamma.$
Let $\overline{\sA}_{g,\Gamma}$ be the toroidal compactification  of $\sA_{g,\Gamma}:=\mathfrak{H}_{g}/\Gamma$ constructed by  $\Sigma_{\mathfrak{F}_0}.$

Assume that the boundary divisor $D_{\infty}:=\overline{\sA}_{g,\Gamma}\setminus\sA_{g,\Gamma}$ is  simple normal crossing. Let $d$ be an integer  with  $1\leq d\leq \dim_\C\sA_{g,\Gamma}-1.$
Let $D_1,\cdots, D_{d}$ be $d$ irreducible components of the boundary divisor $D_\infty.$
For each $D_i,$ let $||\cdot||_i$ be an arbitrary Hermitian metric on the line bundle $[D_i].$

There is
\begin{eqnarray*}
   &&(K_{\overline{\sA}_{g,\Gamma}}+D_\infty)^{\dim_\C\sA_{g,\Gamma}-d}\cdot D_1\cdots D_{d} \\
   &=&\int_{D_i}\mathrm{Res}_{D_i}((\frac{\sqrt{-1}}{2\pi}\partial\overline{\partial}\log\Phi_{g,\Gamma})^{^{\dim_\C\sA_{g,\Gamma}-d}}) \wedge(\bigwedge_{1\leq j\leq d,j\neq i }c_1([D_{j}],||\cdot||_j))
\end{eqnarray*}
for any integer $i\in [1, \dim_\C\sA_{g,\Gamma}],$ where the operator $\mathrm{Res}_{D_i}$ is defined in \ref{restrict-rule}.
Moreover, the intersection number $$(K_{\overline{\sA}_{g,\Gamma}}+D_\infty)^{\dim_\C\sA_{g,\Gamma}-d}\cdot D_1\cdots D_{d}=0$$
if one of the following conditions is satisfied :
(i)\, $g=2;$ \,\,(ii)\, there is a $D_i\in \{D_1,\cdots, D_{d}\}$ constructed from an edge $\rho_i$ in $\Sigma_{\mathfrak{F}_{\min}}$ for some minimal cusp $\mathfrak{F}_{\min}$ such that $\mathrm{Int}(\rho)\subset C(\mathfrak{F}_{\min}).$
\end{theorem}
\begin{myrem}
Let $\rho=\R_{\geq 0}E$ be an edge in $\Sigma_{\mathfrak{F}_0}$ where $E$ is a semi-positive $g\times g$ matrix with rational coefficients.
By the lemma \ref{edge-to-boundary-component}, if $\rank(E)\neq g$(i.e, $\mathrm{Int}(\rho)\subset C(\mathfrak{F}_{0})$) then $\rank(E)=1.$
\end{myrem}
\begin{proof}
Using the proposition \ref{key-lemma-on-recurrence}, we can get the integral formula for $i=1$ immediately by putting $\eta:=\bigwedge\limits_{j=2}^{d}c_1([D_{j}],||\cdot||_j)$ into the lemmas \ref{lemma-intersection-calculation-1}, \ref{lemma-intersection-calculation-2}, \ref{lemma-intersection-calculation-3}, \ref{lemma-intersection-calculation-4}.
\end{proof}

The theorem \ref{Infity-divisor-on-toroidal-compactification} and the corollary \ref{intersection-theory-boundary-divisor} obviously imply the following fact.

\noindent{\bf Claim *} :{\it  Let $D_1,\cdots, D_d$ be $d$($\leq \dim_\C \sA_{g,\Gamma}$) different irreducible components of $D_\infty.$ We have that the set $\mathfrak{D}:=\bigcap\limits_{i=1}^{d} D_{i}$ is not empty
if and only if there exists a minimal cusp $\mathfrak{F}_{\min}$ and a top-dimensional cone $\sigma_{\max}$ in the polyhedral decomposition $\Sigma_{\mathfrak{F}_{\min}}$ of the convex cone $C(\mathfrak{F}_{\min})$ such that each $D_i$ corresponds one-one to an edge of $\sigma_{\max}.$
Moreover, if $\mathfrak{D}\neq \emptyset$ then every local chart $U_\alpha$ in \ref{local-chart-toroidal-compactification} satisfying
$U_\alpha \cap \mathfrak{D}\neq \emptyset$ is constructed by
a top-dimensional cone $\sigma_{\max}$ in $\Sigma_{\mathfrak{F}_{\min}}$ for some minimal cusp $\mathfrak{F}_{\min}.$ 
}\\

For any two different irreducible components $D_\alpha,D_\beta$ of $D_\infty=\bigcup\limits_{i}D_i,$
we define $$(D_\alpha\cap D_\beta)_\infty:=\bigcup\limits_{i\neq \alpha, \beta} D_\alpha\cap D_\beta\cap D_i \,\,\mbox{ and }\,\,(D_\alpha\cap D_\beta)^*:=(D_\alpha\cap D_\beta)\setminus(D_\alpha\cap D_\beta)_\infty.$$
Let $p$ be an arbitrary positive integer. By the lemma \ref{lemma-on-degree-of-polynomial}, we can define a form $\mathrm{Res}_{D_i\cap D_j}(\mathrm{Res}_{D_i}((\partial\overline{\partial}\log\Phi_{g,\Gamma})^p))$ on $(D_i\cap D_j)^*$ for any two different irreducible components $D_i,D_j$ of $D_\infty$  with $D_i\cap D_j\neq \emptyset.$  Using similar arguments in the proposition \ref{key-lemma-on-recurrence}, we get that  $\mathrm{Res}_{D_i\cap D_j}(\mathrm{Res}_{D_i}((\partial\overline{\partial}\log\Phi_{g,\Gamma})^p))$ is a closed smooth form on $(D_i\cap D_j)^*$ having Poincar\'e growth on $(D_i\cap D_j)_\infty.$
\begin{lemma}\label{lemma-on-recurrence-integral}
Let $p$  be a positive integer. Let $D_i,D_j$ be two different irreducible components of $D_\infty$ such that $D_i\cap D_j\neq \emptyset.$
Both $\mathrm{Res}_{D_i\cap D_j}(\mathrm{Res}_{D_i}((\frac{\sqrt{-1}}{2\pi}\partial\overline{\partial}\log\Phi_{g,\Gamma})^p))$ and $\mathrm{Res}_{D_i\cap D_j}(\mathrm{Res}_{D_j}((\frac{\sqrt{-1}}{2\pi}\partial\overline{\partial}\log\Phi_{g,\Gamma})^p))$ are positive closed currents on $D_i\cap D_j$ such that
$$\mathrm{Res}_{D_i\cap D_j}(\mathrm{Res}_{D_i}((\frac{\sqrt{-1}}{2\pi}\partial\overline{\partial}\log\Phi_{g,\Gamma})^p))=\mathrm{Res}_{D_i\cap D_j}(\mathrm{Res}_{D_j}((\frac{\sqrt{-1}}{2\pi}\partial\overline{\partial}\log\Phi_{g,\Gamma})^p)).$$
\end{lemma}
\begin{proof}
Let $i=1,j=2$ for convenience.
There exists a minimal cusp $\mathfrak{F}_{\min}$ and a top-dimensional cone $\sigma_{\max}$ in the polyhedral decomposition $\Sigma_{\mathfrak{F}_{\min}}:=\{\sigma_\alpha\}_\alpha$ of the convex cone $C(\mathfrak{F}_{\min})$ such that  $D_1, D_2$ correspond to two different edges $\rho_1:=\R_{\geq 0}E_1,\rho_2:=\R_{\geq 0}E_2  $ of $\sigma_{\max}.$ Let $(U_\alpha,(w_1^\alpha,\cdots, w_N^\alpha))$
be a local chart as \ref{local-chart-toroidal-compactification} such that
$\{ w^\alpha_i=0\}=U_\alpha\cap D_i$ for $i=1,2.$
In this chat,we write the form $\Phi_{g,\Gamma}$ on  $U^*_\alpha:=U_\alpha \setminus D_\infty$ as
$$
 \Phi_\alpha=\frac{(\frac{\sqrt{-1}}{2})^N2^{\frac{g(g-1)}{2}}\bigwedge\limits_{1\leq i \leq N} dw_{i}^\alpha\wedge d\overline{w_{i}^\alpha}}{(\prod_{1\leq i \leq N}|w_{i}^{\alpha}|^2)(F^{\alpha}(\log|w_1^{\alpha}|,\cdots,\log|w_N^{\alpha}|))^{g+1}}.
$$
For convenience, let $\mathfrak{F}_{\min}$ be the standard minimal cusp $\mathfrak{F}_0.$ Then, $E_1$ and $E_2$ are semi-positive $g\times g$ symmetric real matrices. By the lemma \ref{edge-to-boundary-component}, we have
$$\rank(E_i)= 1 \mbox{ or } g  \,\,\,\mbox{ for } i=1,2.$$
Then, we need to check the following three cases.
\begin{itemize}
  \item $\rank E_1=g$ or $\rank E_2=g$ : Let $\rank E_1=g.$ Then we have
  $$\mathrm{Res}_{D_1}((\partial\overline{\partial}\log\Phi_{g,\Gamma})^p)=0\,\,\mbox{ and }\,\mathrm{Res}_{D_1\cap D_2}(\mathrm{Res}_{D_2}((\partial\overline{\partial}\log\Phi_{g,\Gamma})^p))=0
    \, \,\mbox{ on } U_\alpha^*$$
  since $F^{\alpha}(x_1,x_2,\cdots, x_N)=\mbox{Constant}\cdot x_1^g+ \mbox{ terms of lower dgeree of } x_1.$
  \item $\rank E_1=\rank E_2=1$ and there is an orthogonal $g\times g$ matrix $O$ such that
  $$O^TE_1O=\mathrm{diag}[\lambda_{1,1},0,\cdots, 0](\lambda_{1,1}>0)\mbox{ and }O^TE_2O=\mathrm{diag}[\lambda_{2,1},0,\cdots, 0](\lambda_{2,1}>0) : $$
  Then, the homogenous polynomial $F^\alpha$ has the form
  $$F^{\alpha}(x_1,x_2,\cdots, x_N)=S(x_3,\cdots, x_N)\cdot(\lambda_{1,1}x_1+ \lambda_{2,1}x_2)+ \mbox{ terms without } x_1 \mbox{ and } x_2, $$
  and so $ \mathrm{Res}_{D_1\cap D_2}(\mathrm{Res}_{D_1}((\partial\overline{\partial}\log\Phi_{g,\Gamma})^p))=\mathrm{Res}_{D_1\cap D_2}(\mathrm{Res}_{D_2}((\partial\overline{\partial}\log\Phi_{g,\Gamma})^p))=0
    \mbox{ on } U_\alpha^*.$
  \item Otherwise, the homogenous polynomial $F^\alpha$ has the form
  $$F^{\alpha}(x_1,x_2,\cdots, x_N)=S(x_3,\cdots, x_N)\cdot x_1x_2+ \mbox{ terms without } x_1 \mbox{ and } x_2, $$
  and so $ \mathrm{Res}_{D_1\cap D_2}(\mathrm{Res}_{D_1}((\partial\overline{\partial}\log\Phi_{g,\Gamma})^p))=\mathrm{Res}_{D_1\cap D_2}(\mathrm{Res}_{D_2}((\partial\overline{\partial}\log\Phi_{g,\Gamma})^p))
   \,\, \mbox{ on }\,\, U_\alpha^*.$
\end{itemize}
\end{proof}

We observe that these $\mathrm{Res}_{D_i\cap D_j}(\mathrm{Res}_{D_i}(\partial\overline{\partial}\log\Phi_{g,\Gamma}))$'s($i\neq j$) and $\mathrm{Res}_{D_i}(\partial\overline{\partial}\log\Phi_{g,\Gamma})$'s  all have similar type as that of $\partial\overline{\partial}\log\Phi_{g,\Gamma}.$
\begin{theorem}\label{multiple-intersection-formula}
Let $\Gamma\subset \Sp(g,\Z)$ be a neat arithmetic subgroup and $\Sigma_{\mathfrak{F}_0}$ a
$\overline{\Gamma_{\mathfrak{F}_0}}$(or $\mathrm{GL}(g,\Z)$)-admissible polyhedral decomposition of $C(\mathfrak{F}_0)$ regular with respect to $\Gamma.$
Let $\overline{\sA}_{g,\Gamma}$ be the toroidal compactification  of $\sA_{g,\Gamma}:=\mathfrak{H}_{g}/\Gamma$ constructed by  $\Sigma_{\mathfrak{F}_0}.$

Assume that the boundary divisor $D_{\infty}:=\overline{\sA}_{g,\Gamma}\setminus\sA_{g,\Gamma}$ is  simple normal crossing. Let $d$ be an integer  with  $1\leq d\leq \dim_\C\sA_{g,\Gamma}-1$ and
let $D_1,\cdots, D_{d}$ be any $d$ different irreducible components of the boundary divisor $D_\infty$ such that $\bigcap\limits_{k=1}^d D_k\neq \emptyset.$
We have :
\begin{myenumi}
\item Let $i_1,\cdots, i_d$ be $d$ positive integers.
If $d\geq g-1$ and $\dim_\C\sA_{g,\Gamma}-\sum\limits_{k=1}^d i_k\geq 2 $(or if $d\geq g$ and $\dim_\C\sA_{g,\Gamma}-\sum\limits_{k=1}^d i_k=1$) then the intersection number $$D_1^{i_1}\cdots D_{d}^{i_d}\cdot(K_{\overline{\sA}_{g,\Gamma}}+D_\infty)^{\dim_\C\sA_{g,\Gamma}-\sum_{k=1}^d i_k}=0.$$


\item The divisor $K_{\overline{\sA}_{g,\Gamma}}+D_\infty$ is not ample on $\overline{\sA}_{g,\Gamma}.$

\item Define $D^{(1)}:=D_1, D^{(2)}:=D^{(1)}\cap D_2,\cdots, D^{(d)}:=D^{(d-1)}\cap D_d.$
If $d< g-1$ then there is
\begin{eqnarray*}
   && (K_{\overline{\sA}_{g,\Gamma}}+D_\infty)^{\dim_\C\sA_{g,\Gamma}-d}\cdot D_1\cdots D_{d}\\
   &=& \int_{\bigcap\limits_{k=1}^d D_k}\mathrm{Res}_{D^{(d)}}(\mathrm{Res}_{D^{(d-1)}}\cdots
(\mathrm{Res}_{D^{(1)}}((\frac{\sqrt{-1}}{2\pi}\partial\overline{\partial}\log\Phi_{g,\Gamma})^{\dim_\C\sA_{g,\Gamma}-d}))
\cdots)
\end{eqnarray*}
where  the  $\mathrm{Res}_{D^{(k)}}(\cdots)\,\,k=1,\cdots,d $ are current on $D^{(k)}$ defined recursively  by the lemma \ref{lemma-on-degree-of-polynomial}.
\end{myenumi}
\end{theorem}
\begin{proof}
The statement (2) is a direct consequence of the statement (1). For each $i,$ let $||\cdot||_i$ be an arbitrary Hermitian metric on the line bundle $[D_i].$
We have that
\begin{eqnarray*}
   && \int_{D_1}\mathrm{Res}_{D_1}((\partial\overline{\partial}\log\Phi_{g,\Gamma})^{^{\dim_\C\sA_{g,\Gamma}-d}}) \wedge(\bigwedge_{l=2}^{ \dim_\C\sA_{g,\Gamma}}c_1([D_{l}],||\cdot||_l))\\
   &=& \int_{D_1\cap D_2}\mathrm{Res}_{D_1\cap D_2}(\mathrm{Res}_{D_1}((\partial\overline{\partial}\log\Phi_{g,\Gamma})^{^{\dim_\C\sA_{g,\Gamma}-d}})) \wedge(\bigwedge_{l=3}^{ \dim_\C\sA_{g,\Gamma}}c_1([D_{l}],||\cdot||_l))
\end{eqnarray*}
by the estimates in the lemmas \ref{lemma-intersection-calculation-1}-\ref{lemma-intersection-calculation-4}. Therefore, we can finish proving the statements (1) and (3) by recursion since all local volume functions have only degree $g.$
\end{proof}

\begin{example}Suppose that $d$ different irreducible components $D_1,\cdots. D_d$ of $D_\infty:=\overline{\sA}_{g,\Gamma}\setminus\sA_{g,\Gamma}$ has $\bigcap\limits_{l=1}^d D_l\neq \emptyset.$
Let $(U_\alpha,(w_1^\alpha,\cdots, w_N^\alpha))$($N=\frac{g(g+1)}{2}$)
be a local chart as \ref{local-chart-toroidal-compactification} such that
$\{ w^\alpha_i=0\}=U_\alpha\cap D_i$ for all $i=1,\cdots, d.$
The following is an algorithm processor to produce $\chi_\alpha:=\mathrm{Res}_{D^{(d)}}(\cdots (\mathrm{Res}_{D^{(1)}}((\partial\overline{\partial}\log\Phi_{g,\Gamma})^{N-d})))|_{U_\alpha\cap\bigcap\limits_{l=1}^d D_l}.$
\begin{description}
  \item[00]Begin

  \item[10] In put the local volume form on $U_\alpha^*:=U_\alpha \setminus D_\infty$ :
  $$
 \Phi_\alpha=\frac{(\frac{\sqrt{-1}}{2})^N2^{\frac{g(g-1)}{2}}\mathrm{vol}_{\Gamma}(\sigma_{\max})^2\bigwedge\limits_{1\leq i \leq N} dw_{i}^\alpha\wedge d\overline{w_{i}^\alpha}}{(\prod_{1\leq i \leq N}|w_{i}^{\alpha}|^2)(F^{\alpha}(\log|w_1^{\alpha}|,\cdots,\log|w_N^{\alpha}|))^{g+1}},
$$
where $F^{\alpha}\in \R[x_1,\cdots, x_{N}]$ of degree $g$ is the local volume function.
\item[20] Let $k=0$ and let $S_0(x_1,\cdots, x_{N}) :=F^{\alpha}(x_1,\cdots, x_{N}).$

  \item[30] Let  $f(x_1,\cdots, x_{N}):=S_k(x_{k+1},\cdots, x_{N})$ and $q=k+1.$

  \item[40] Let $k=q$  and let $n:=\deg_k f(x_1,\cdots, x_{N}).$  Write
      $$f= S_k(x_{k+1},\cdots,x_N)x_k^{n}+\mbox{terms of lower degree of $x_k$}.$$

  \item[50] Let $P=(P_{lm})_{1\leq l,m\leq N}$ be a $N\times N$ matrix given by
$P_{lm}=S_k\frac{\partial^2 S_k}{\partial x_l\partial x_m}-\frac{\partial S_k}{\partial x_l}\frac{\partial S_k}{\partial x_m}. $

  \item[60]Let $g_k(x_{k+1},\cdots x_N):=\det(P^{[1,\cdots,d]}_{[1,\cdots,d]})$
  (Thus, $\deg S_k\leq g-k$ and $\deg g_k \leq 2(N-d)(\deg S_k-1)\leq2(N-d)(g-k-1);$ if  $\deg S_k=0$ then $g_k\equiv0;$ if $\deg S_k=1$ and $N-d\geq 2$ then $g_k\equiv0$).

  \item[70] If $k<d$ then  goto {\bf 30}.

  \item[80]Output
\begin{eqnarray*}
  \chi_\alpha &=& (-1)^{N-d}(\frac{g+1}{4})^{N-d}(N-d)!\\
   &&\times \frac{g_d(\log|w_{d+1}^\alpha|, \cdots, \log|w_{N}^\alpha|)}{S_d(\log|w_{d+1}^\alpha|, \cdots, \log|w_{N}^\alpha|)^{2(N-d)}}\frac{\bigwedge\limits_{i=d+1}^N dw^\alpha_{i}\wedge d\overline{w^\alpha_{i}}}{(\prod\limits_{i=d+1}^{N}|w^\alpha_i|)^2}.
\end{eqnarray*}

   \item[90] End.
\end{description}
\end{example}

\vspace{0.5cm}

{\small

}

\end{document}